\input ./style/arxiv-general.cfg
\documentclass[aap,MSNbibl,seceqn,nameyear,dvips]{arximspdf}
\makeatletter
   \@ifpackageloaded{graphicx}{}{\usepackage{graphicx}}
\makeatother
\usepackage{mathbh}
\usepackage{stfloats}

\doi{10.1214/14-AAP1056} 
\volume{25}
\issue{5}
\pubyear{2015}
\firstpage{2563}
\lastpage{2625}
\docsubty{FLA}

\makeatletter
  \fnbelowfloat
\newcommand{\BMO}{\mathrm{BMO}}
\newcommand{\rrvert}{\vert}
\newcommand{\rrVert}{\Vert}
\newcommand{\llvert}{\vert}
\newcommand{\llVert}{\Vert}
\renewcommand{\mid}{|}
\newtheorem{theorem}{Theorem}[section]
\newtheorem{lemma}[theorem]{Lemma}
\newtheorem{proposition}[theorem]{Proposition}
\newtheorem{corollary}[theorem]{Corollary}
\newproclaim{definition}[theorem]{Definition}
\newproclaim{remark}[theorem]{Remark}
\newproclaim{example}[theorem]{Example}
\newcommand{\mathbbm}{\mathbh}
\newcommand{\Std}{\mathrm{St.d.}}
\newcommand{\N}{\mathbb{N}} 
\newcommand{\R}{\mathbb{R}} 
\newcommand{\F}{\mathcal{F}} 
\renewcommand{\S}{\mathcal{S}}
\newcommand{\bD}{\mathbb{D}}
\newcommand{\bE}{\mathbb{E}}
\newcommand{\bL}{\mathbb{L}}
\newcommand{\bN}{\mathbb{N}}
\newcommand{\bP}{\mathbb{P}}
\newcommand{\bR}{\mathbb{R}}
\newcommand{\cB}{\mathcal{B}}
\newcommand{\cF}{\mathcal{F}}
\newcommand{\cH}{\mathcal{H}}
\newcommand{\cL}{\mathcal{L}}
\newcommand{\cN}{\mathcal{N}}
\newcommand{\cR}{\mathcal{R}}
\newcommand{\cS}{\mathcal{S}}
\newcommand{\cT}{\mathcal{T}}
\newcommand{\cY}{\mathcal{Y}}
\newcommand{\cZ}{\mathcal{Z}}
\newcommand{\ti}{{t_{i}}}
\newcommand{\tip}{{t_{i+1}}}
\newcommand{\tN}{{t_N}}
\newcommand{\ip}{{i+1}}
\def\o{\theta}
\def\bT{\bar{\Theta}}
\newcommand{\ud}{\mathrm{d}}
\newcommand{\uds}{\mathrm{d}s}
\newcommand{\udt}{\mathrm{d}t}
\newcommand{\udu}{\mathrm{d}u}
\newcommand{\udr}{\mathrm{d}r}
\newcommand{\udw}{\mathrm{d}W}
\newcommand{\udws}{\mathrm{d}W_s}
\newcommand{\udwr}{\mathrm{d}W_r}
\newcommand{\1}{\mathbbm{1}}
\makeatother

\begin{document}
\begin{frontmatter}

\title{Time discretization of FBSDE with polynomial growth drivers and
reaction--diffusion PDEs\thanksref{T3}}
\runtitle{Numerics for FBSDE with polynomial growth drivers}

\begin{aug}
\author[A]{\fnms{Arnaud}~\snm{Lionnet}\ead[label=e1]{arnaud.lionnet@maths.ox.ac.uk}\thanksref{T1}},
\author[B]{\fnms{Gon\c calo} \snm{dos Reis}\corref{}\ead[label=e2]{G.dosReis@ed.ac.uk}\thanksref{T2}}
\and
\author[B]{\fnms{Lukasz} \snm{Szpruch}\ead[label=e3]{l.szpruch@ed.ac.uk}\ead[label=u1,url]{http://www.foo.com}}
\runauthor{A. Lionnet, G. dos Reis and L. Szpruch}
\affiliation{University of Oxford,
University of Edinburgh and CMA/FCT/UNL,
and
University of Edinburgh}
\address[A]{A. Lionnet\\
Oxford-Man Institute\\
University of Oxford\\
Eagle Hause, Walton Well Road\\
Oxford OX2 6ED\\
United Kingdom\\
\printead{e1}}
\address[B]{G. dos Reis\\
L. Szpruch\\
School of Mathematics\\
University of Edinburgh\\
Edinburgh, EH9 3JZ\\
United Kingdom\\
\printead{e2}\\
\phantom{E-mail: }\printead*{e3}}
\end{aug}
\thankstext{T3}{Supported by the Engineering and Physical Sciences
Research Council (UK) which has funded this work as part of the
Numerical Algorithms and Intelligent Software Centre (NAIS) under Grant
EP/G036136/1.}
\thankstext{T1}{Supported in part by the Engineering and Physical
Sciences Research Council (UK) via the Grant EP/P505216/1,
as well as by the Oxford-Man Institute.}
\thankstext{T2}{Supported by the \emph{Funda\c{c}\~{a}o para a Ci\^
{e}ncia e a Tecnologia} (Portuguese
Foundation for Science and Technology) through the project
PEst-OE/MAT/UI0297/2014
(Centro de Matem\'atica e Aplica\c{c}\~{o}es CMA/FCT/UNL).}

\received{\smonth{9} \syear{2013}}
\revised{\smonth{6} \syear{2014}}

%
\begin{abstract}
In this paper, we undertake the error analysis of the
time discretization of systems of Forward--Backward Stochastic Differential
Equations (FBSDEs) with drivers having polynomial
growth and that are also monotone in the state variable.

We show with a counter-example that the natural explicit Euler
scheme may diverge, unlike in the canonical Lipschitz driver case.
This is due to the lack of a certain stability property of the Euler scheme
which is essential to obtain convergence.
However, a thorough analysis of the family of $\theta$-schemes reveals that
this required stability property can be recovered if the scheme is
sufficiently implicit. As a by-product of our analysis, we shed some
light on
higher order approximation schemes for FBSDEs under non-Lipschitz condition.
We then return to fully explicit schemes and show that an appropriately tamed
version of the explicit Euler scheme enjoys the required stability property
and as a consequence converges.

In order to establish convergence of the several discretizations, we extend
the canonical path- and first-order variational regularity results
to FBSDEs with polynomial growth drivers which are also monotone. These
results are of independent interest for the theory of FBSDEs.
\end{abstract}

%
\end{frontmatter}

\setcounter{footnote}{3}

\section{Introduction}\label{sec1}
There is currently a long literature on the numerical approximation of FBSDE
with Lipschitz conditions [\citet{BouchardTouzi2004},
\citet{CrisanManolarakis2012}, \citet{GobetTurkedjiev2011},\break 
\citeauthor{Chassagneux2012} (\citeyear{Chassagneux2012,Chassagneux2013}) and references within].
In this article, we address the case of FBSDEs with drivers having polynomial
growth in the state variable, which has not been studied before,
and provide customized analysis of various implicit and explicit schemes.
The importance of FBSDEs with nonlinear drivers is due to the fruitful
connection between FBSDEs and partial differential equations (PDEs).
Many biological and physical phenomena are modeled using PDEs of parabolic
type, say for
$(t,x)\in[0,T]\times\bR^d$
\[
-\partial_t v(t,x) - \cL v(t,x) -f \bigl(t,x,v(t,x),(\nabla v
\sigma) (t,x) \bigr)=0,\qquad v(0,x)=g(x),
\]
with $\cL$ a second-order elliptic differential operator and certain
measurable functions $f$ and $g$. A very large
class of such equations can be linked to the solution process
$\Theta^{t,x}=(X^{t,x},Y^{t,x},Z^{t,x})$ of certain forward--backward
stochastic
differential equations (FBSDEs) with the\vspace*{1pt} following type of dynamics for
$(t,x)\in[0,T]\times\bR^d$,
$s\in[t,T]$ and $W$ a Brownian-motion
%
\begin{eqnarray}
\label{canonicSDE} X^{t,x}_s&=& x + \int_t^s
b\bigl(r,X^{t,x}_r\bigr)\,\udr+ \int_t^s
\sigma\bigl(r,X^{t,x}_r\bigr)\,\udwr,
\\
\label{canonicFBSDE} Y^{t,x}_s&=& g\bigl(X^{t,x}_T
\bigr) + \int_s^T f\bigl(r,\Theta_r^{t,x}
\bigr) \,\udr- \int_s^T Z^{t,x}_r\,\udwr,
\end{eqnarray}
via the so-called nonlinear Feynman--Kac formula:
$v(T-t,x)=Y^{t,x}_t$ [see, e.g., \citet{ElKarouiPengQuenez1997}].

In many applications of interest, like reaction--diffusion type
equations, the
function $f$ is a polynomial (in $v$), for example, the Allen--Cahn
equation, the FitzHugh--Nagumo equations (with or without recovery) or the
standard nonlinear heat and Schr\"odinger equation [see
\citet{Henry1981}, \citet{Rothe1984}, \citet{EstepLarsonWilliams2000},
\citet{Kovacs2011} and references].

Motivated by these applications, we look further at the connection
between parabolic PDEs and FBSDEs with monotone drivers $f$ of polynomial
growth [see \citet{Pardoux1999}, \citet{BriandCarmona2000} and
\citet{BriandDelyonHuEtAl2003}]. By monotonicity,
we mean that $\langle v'-v, f(v')-f(v)\rangle\leq\mu\llvert v'-v
\rrvert^2$, for some
$\mu\ge0$, and any $v,v'$ (one can also find the terminology that $f$ is
one-sided Lipschitz). We extend the above mentioned works by providing further
regularity estimates for the \mbox{FBSDE} in question (modulus of continuity, path
and variational regularity). Then we proceed to a thorough analysis of
various numerical methods that open the door to Monte Carlo methods for
solving numerically the corresponding PDEs.


The work and results we present should be understood as a first step in the
numerical analysis of FBSDE with monotone drivers of polynomial growth, wider
than the Lipschitz driver BSDE
setting, with the intent of deepening the applicability of FBSDEs to
reaction--diffusion equations. Moreover, we work without
assuming knowledge on the density function or the moment generating function
of the forward process $X$. In some applications where $X$ is simply the
Brownian motion, it is possible to derive a numerical solver that takes
advantage on this knowledge; see, for example, \citet{ZhangGunzburgerZhao2013}. The
work we develop aims at black-box type algorithms which do not take advantage
of any of the specific forms the FBSDEs coefficients may take.

\subsection*{A motivating example}
To better understand why the explicit Euler scheme seems not to be suitable
for approximating the solution to BSDEs with non-Lipschitz drivers, let us
consider the following simple example (for further details and notational
setup, see Section~\ref{secPreliminaries} and Appendix~\ref{appdxmotivatignexample}):
%
\begin{equation}
\label{counterexampleFBSDE} Y_t = \xi- \int_{t}^{1}
Y_s^3 \,\uds- \int_{t}^{1}
Z_s \,\ud W_s,\qquad t\in[0,1]
\end{equation}
with the terminal condition $\xi\in\cF_1$. For any
$\xi\in L^p$ for $p\geq2$, there exists\footnote{Existence and uniqueness
follows from Section~2 in \citet{Pardoux1999} or
Theorem~\ref{theo-existenceuniquenesscanonicalFBSDE} below.} a
unique (square-integrable) solution $(Y,Z)$ to the above BSDE.

Fix the number of time-discretization points to be $N+1>0$. The explicit
Euler scheme for the above equation with uniform time step $h=1/N$ is,
with the notation $Y_i:= Y_{i/N}$, given by
%
\begin{eqnarray}\label{eqeulerschememotivating}
Y_i = \bE\bigl[ Y_{\ip} -
Y_{\ip}^3 h \mid\cF_i \bigr]= \bE\bigl[
Y_{\ip}\bigl(1 - hY_{\ip}^2\bigr) |
\cF_i \bigr],
\nonumber\\[-8pt]\\[-8pt]
\eqntext{i=0,\ldots,N-1,}
\end{eqnarray}
where $Y_N=\xi$.

It is a simple calculation
(see Appendix~\ref{appdxmotivatignexample} for the details)
to show that if
%
\begin{equation}
\label{doubleexponentialimpact} \xi\ge2\sqrt{N} \qquad\mbox{then }
\llvert Y_i\rrvert
\ge2^{2^{N-i}}\sqrt{N} \qquad\mbox{for } i=0,\ldots,N.
\end{equation}
With this simple computation in mind, it is possible to
show that there exists a random variable $\xi$ whose moments of any
order are
finite and for which the explicit Euler scheme diverges. The result below
is a corollary of Lemma~\ref{lemma-counterexample} that can be found in
Appendix~\ref{appdxmotivatignexample}.

\begin{lemma}
\label{lemmaeulerdoesntwork}
Let $\pi^{N}$ be the uniform grid over the interval $[0,1]$ with $N+1$
points, $N$ an even number ($t=1/2$ is common to all grids $\pi^N$).
For any $\xi\in L^p(\cF_1)$, for $p\geq2$, let $(Y,Z)$ denote the solution
to (\ref{counterexampleFBSDE}).

Then there exists a random variable $\xi\in L^p\setminus L^\infty$
for any
$p\geq2$ such that
\[
\lim_{N\to\infty} \bE\bigl[ \bigl\llvert Y_{1/2}^{(N)}
\bigr\rrvert\bigr]=+\infty,
\]
where $Y^{(N)}_{1/2}$ is the Euler approximation of
$Y$ on the time
point $t=1/2$ via
(\ref{eqeulerschememotivating}) over the
grids $\pi^{N}$.
\end{lemma}

The special random variable $\xi$ we work with is normally distributed and
it is known that $\bP[\llvert\xi\rrvert >2\sqrt{N}]$ is
exponentially small (see Lemma~\ref{lemmapropertiesofBM}). What our counter-example shows is that although
$\xi$ may take very large values on an event with exponentially small
probability, the impact of these very large values when propagated through
the Euler explicit scheme is doubly-exponential [see
(\ref{doubleexponentialimpact})].

This double-exponential impact is precisely a consequence of the
superlinearity of the driver. In general, the terminal condition
$\xi$ is an unbounded random variable (RV) so there is a positive probability
of the scenario where $\xi\ge2\sqrt{N}$ no matter how small a
time-step we
choose. This indicates that, in general, the explicit Euler scheme
may diverge, as it happens in SDE context
\citet{HutzenthalerJentzenKloeden2011}.
Therefore, one needs to seek alternative (e.g., implicit) approximations
for BSDE with
polynomial drivers that are also monotone and/or find conditions under which
it is possible for the explicit scheme to work, as explicit schemes have
certain computational advantages over implicit ones.

\subsection*{Our contribution}
\begin{itemize}
\item We extend the canonical Zhang path regularity theorem [see
\citet{MaZhang2002a}, \citet{ImkellerDosReis2010}], originally proved under
Lipschitz assumptions, to our
polynomial growth monotone driver setting proving in between all the required
stochastic smoothness results; essentially all first-order variations
of the
solution processes and estimates on the modulus of continuity.

\item For our non-Lipschitz setting, we provide a thorough analysis of the
family of $\o$-schemes, where $\o\in[0,1]$ characterizes the degree of
implicitness of the scheme.
Contrary to the FBSDEs with Lipschitz driver we show that choosing
$\theta\ge1/2$ is essential to ensure the stability
of the scheme, in a similar way to the SDE context [see
\citet{mao2012strong}].
This is to our knowledge the first result
in the numerical BSDEs literature that shows a superior stability of the
implicit scheme over the standard explicit one. We also generalize the concept
of stability for discretization schemes [see that in \citeauthor{Chassagneux2012} (\citeyear{Chassagneux2012,Chassagneux2013})].
This, among others things, paves a way for deriving higher order
approximations
schemes for FBSDEs with non-Lipschitz drivers. As an example, we prove a
higher
order of convergence for the trapezoidal scheme (the case $\o=1/2$).

\item We construct an appropriately tamed version of the
explicit Euler scheme for which the required stability property can be
recovered. This allows us to obtain convergence of the scheme. Interestingly
enough, in the special case where the driver of the FBSDEs does not
depend on the
SDE solution it is enough to appropriately tame the terminal condition,
leaving the rest of the Euler approximation unchanged.
\end{itemize}

As a rule of thumb, implicit schemes tend to be more
robust than explicit ones. Unfortunately implicit schemes involve
solving an
implicit equation, which creates an extra layer of complexity when compared
to explicit schemes. A secondary aim of this work is to distinguish
under which conditions explicit and implicit schemes can be used.

As standard in numerical analysis, we derive the global error estimates of
various numerical schemes by analyzing their one-step errors and stability
properties
(which allows us to study how errors propagate with time).
We formulate the \emph{Fundamental Lemma} [following the nomenclature from
\citet{milstein2004stochastic}] that states how to estimate the global
error of a stable approximation scheme in terms of its local errors. The
lemma
is proved
under minimal assumptions. We stress that a similar approach has been
used in \citet{ChassagneuxCrisan2012} and \citeauthor{Chassagneux2012} (\citeyear{Chassagneux2012,Chassagneux2013}); however, their results are not sufficiently general
to deal with non-Lipschitz drivers.

The structure of the global error estimate given by the Fundamental Lemma
allows us to study in a very easy and transparent
way the special case of the \mbox{$\o$-}scheme with $\o=1/2$ (trapezoidal rule)
which has a higher order of convergence. In this context, we also conjecture
a candidate for the second-order scheme.

Concerning the implementation of the presented schemes, we propose an
alternative estimator of the component $Z$ whose standard deviation, contrary
to usual estimator, does not explode as the time step vanishes.

Finally, we note that in proving convergence for the mostly-implicit schemes,
we prove $L^p$-type uniform bounds for the scheme, thus extending the
classical $L^2$-bound obtained previously for the discretization of Lipschitz
FBSDEs [see \citet{BouchardTouzi2004}, \citet{GobetTurkedjiev2011} and
references therein].

This work is organized as follows.
In Section~\ref{secPreliminaries}, we define notation and
recall standard results from the literature.
In Section~\ref{secregularity}, we establish first-order variational results for the solution of the FBSDEs as well as stating
the path regularity results required for the study of numerical schemes within
the FBSDE framework. The remaining sections contain the discussion of several
numerical schemes: in Section~\ref{sectionnumericaldiscretization},
we define
the numerical discretization procedure and state general estimates for
integrability and on the local errors.
In Section~\ref{sectionthetaimplicitschemes}, we establish the convergence
of the implicit dominating schemes and in Section~\ref{sectionexplicitscheme}
the convergence of the tamed explicit scheme [after the terminology of
\citet{HutzenthalerJentzenKloeden2012}]. In Section~\ref{sectionnumerics},
we give some numerical examples.

%
\section{Preliminaries}
\label{secPreliminaries}

\subsection{Notation}
\label{subsecNotation}

Throughout let us fix $T>0$. We work on a canonical Wiener space
$(\Omega,
\cF, \bP)$ carrying a $d$-dimensional Wiener process $W = (W^1,\break \ldots, W^d)$
restricted to the time interval $[0,T]$. We denote by
$\cF=(\cF_t)_{ t\in[0,T]}$ its natural filtration enlarged in the
usual way
by the $\bP$-zero sets and by $\bE$ and $\bE[\cdot\mid\cF_t]=\bE
_t[\cdot]$ the
usual expectation and conditional expectation operator, respectively.\vspace*{1pt}

For vectors $x =(x^1,\ldots, x^d)$ in the Euclidean space $\bR^d$, we
denote by
$\llvert\cdot\rrvert $ and $\langle\cdot,\cdot\rangle$ the
canonical Euclidean norm and
inner product
(resp.) while $\llVert\cdot\rrVert$ is the matrix
norm in $\bR^{k\times
d}$ (when no ambiguity arises we use $\llvert\cdot\rrvert $
as $\llVert\cdot
\rrVert$); for
$A\in\bR^{k\times d}$ $A^*$ denotes the transpose of $A$; $I_d$
denotes the $d$-dimensional identity matrix. For a map $b\dvtx \bR^m\to\bR
^d$, we
denote by $\nabla b$ its $\bR^{d\times m}$-valued Jacobi matrix
(gradient in
case $d=1$) whenever it
exists. To denote the $j$th first derivative of $b(x)$ for $x\in\bR
^m$, we
write $\nabla_{x_j} b$ (valued in $\bR^{d\times1}$). For
$b(x,y)\dvtx \bR^m\times\bR^d\to
\bR^k$, we write $\nabla_x h$ or $\nabla_y h$ to refer to its Jacobi matrix
(gradient if $k=1$) with relation to $x$~and~$y$, respectively. $\Delta$
denotes the canonical Laplace operator.

We define the following spaces for $p> 1$, $q\geq1$, $n,m,d,k\in\bN$:
$C^{0,n}([0,T]\times\bR^d,\bR^k)$ is the space of continuous functions
endowed with the $\llVert\cdot\rrVert_\infty$-norm
that are $n$-times continuously differentiable in the spatial variable; $C^{0,n}_b$ contains all
bounded functions of $C^{0,n}$; the first superscript $0$ is dropped for
functions independent of time; $L^{p}(\cF_t,\bR^d)$, $t\in[0,T]$, is the
space of $d$-dimensional
\mbox{$\cF_t$-}measurable RVs $X$ with norm
$\llVert X\rrVert_{L^p} = \bE[ \llvert X\rrvert
^p]^{1/p} < \infty$;
$L^\infty$ refers to the subset of essentially bounded RVs;
$\cS^{p}([0,T]\times\bR^d)$ is the space of $d$-dimensional measurable
$\cF$-adapted processes $Y$ satisfying $\llVert Y\rrVert
_{\cS^p} =
\bE[\sup_{t\in[0,T]}\llvert Y_t\rrvert ^p]^{1/p}
<\infty$; $\cS^\infty$ refers to the subset of $\cS^{p}(\bR^d)$ of
absolutely
uniformly bounded processes; $\cH^{p}([0,T]\times\bR^{n\times d})$ is
the space of $d$-\break dimensional
measurable $\cF$-adapted processes $Z$ satisfying $\llVert Z
\rrVert_{\cH^p}=\break \bE[ (\int_0^T\llvert Z_s\rrvert ^2 \,\uds
)^{p/2}]^{1/p}<\infty$;
$\bD^{k,p}(\bR^d)$ and $\bL_{k,d}(\bR^d)$ are the spaces of
Malliavin differentiable RVs and processes; see
Appendix~\ref{appendix-malliavin-calculus}.

%
%

\subsection{Setting}
We want to study the forward--backward SDE system with dynamics
(\ref{canonicSDE})--(\ref{canonicFBSDE}), for $(t,x)\in[0,T]\times
\bR^d$ and
$\Theta^{t,x}:=(X^{t,x},Y^{t,x},Z^{t,x})$.
Here we work, for $s\in[t,T]$, with the filtration $\cF^t_s:=\sigma
(
W_r-W_t\dvtx  r\in[t, s] )$, completed with the $\bP$-null measure
sets of
$\cF$. Concerning the
functions appearing in (\ref{canonicSDE}) and
(\ref{canonicFBSDE}) we will work with the following assumptions.
\begin{longlist}[(HY0$_{\mathrm{loc}}$):]
\item[(HX0).] $b\dvtx [0,T]\times\bR^{d}\to\bR^d$, $\sigma\dvtx [0,T]\times
\bR^{d}\to\bR^{d\times
d}$ are $1/2$-H\"older continuous in their time
variable, are Lipschitz continuous in their spatial variables,
satisfy $\llVert b(\cdot,0)\rrVert_\infty+\llVert
\sigma(\cdot,0)\rrVert_\infty<\infty
$, and hence
satisfy
$\llvert b(\cdot,x)\rrvert +\llvert\sigma(\cdot,x)\rrvert \leq K(1+\llvert x\rrvert )$ for some $K>0$.

\item[(HY0).]
$g\dvtx \bR^d\to\bR^k$ is a Lipschitz function of linear growth;
$f\dvtx [0,T]\times\bR^d\times\bR^k \times\bR^{k\times d} \to\bR^k$
is a
continuous function and for some $L,L_x,L_y,L_z>0$ for all
$t,t',x,x',y,y',z,z'$ it holds that
%
\begin{eqnarray}\label{HY0monotonicity}
\nonumber
\exists m\geq1\qquad\bigl\llvert f(t,x,y,z)\bigr\rrvert&\leq& L+
L_x\llvert x\rrvert+L_y\llvert y\rrvert
^m+L_z\llVert z \rrVert,
\\
\qquad \bigl\langle y'-y, f\bigl(t,x,y',z
\bigr)-f(t,x,y,z) \bigr\rangle& \le& L_y\bigl\llvert
y'-y\bigr\rrvert^2,
\nonumber\\[-8pt]\\[-8pt]\nonumber
\bigl\llvert f(t,x,y,z)-f\bigl(t',x',y,z'
\bigr)\bigr\rrvert&\leq& L_t \bigl\llvert t-t'\bigr
\rrvert^{1/2}
\\
\nonumber
&&{} + L_x\bigl\llvert x-x'\bigr\rrvert
+L_z \bigl\llVert z-z'\bigr\rrVert.
\end{eqnarray}

\item[(HY0$_{\mathrm{loc}}$).] \textup{(HY0)} holds and, given $L_y$, it holds for
all $t,x,y,y',z$ that
%
\begin{equation}
\label{HY0loc-loclip} \bigl\llvert f(t,x,y,z) - f\bigl(t,x,y',z\bigr)
\bigr\rrvert\le L_y \bigl( 1 + \llvert y\rrvert^{m-1} +
\bigl\llvert y'\bigr\rrvert^{m-1} \bigr)\bigl\llvert
y-y'\bigr\rrvert.
\end{equation}

\item[(HXY1).] \textup{(HX0)}, \textup{(HY0$_{\mathrm{loc}}$)} hold; $g\in C^1$ and $b,\sigma, f\in
C^{0,1}$.
\end{longlist}

We state next a useful consequence of the monotonicity
condition (\ref{HY0monotonicity}).

\begin{remark}
\label{rmkoneside}
Under assumption \textup{(HY0)}, for all $t,x,y,y',z,z'$ and any $\alpha>0$,
we have
\begin{eqnarray*}
&& \bigl\langle y' - y, f\bigl(t,x,y',z'
\bigr) - f(t,x,y,z) \bigr\rangle
\\
&& \qquad= \bigl\langle y' - y, f\bigl(t,x,y',z'
\bigr) \pm f\bigl(t,x,y,z'\bigr)- f(t,x,y,z) \bigr\rangle
\\
&& \qquad\le L_y \bigl\llvert y' - y \bigr\rrvert
^2 + L_z \bigl\llvert y' - y \bigr\rrvert
\bigl\llvert z' -z \bigr\rrvert
\\
&&\qquad \le(L_y+\alpha)
\bigl\llvert y' - y \bigr\rrvert^2 + \frac
{L_z^2}{4 \alpha}
\bigl\llvert z' -z \bigr\rrvert^2.
\end{eqnarray*}
Moreover,
%
\begin{eqnarray}\label{remmomenttrick}
&& \bigl\langle y, f(t,x,y,z) \bigr\rangle\nonumber
\\
&&\qquad = \bigl\langle y-0,
f(t,x,y,z)-f(t,x,0,z) \bigr\rangle+ \bigl\langle y, f(t,x,0,z) \bigr
\rangle
\nonumber\\[-8pt]\\[-8pt] \nonumber
&&\qquad  \le L_y \llvert y\rrvert^2 + \llvert y
\rrvert\bigl( L +L_x\llvert x\rrvert+ L_z\llvert z
\rrvert\bigr)
\\
&&\qquad  \le(L_y+\alpha)\llvert y\rrvert^2 +
\frac{3L^2}{4\alpha
} + \frac{3L_x^2}{4\alpha}\llvert x\rrvert^2 +
\frac{3L_z^2}{4\alpha}\llvert z\rrvert^2.\nonumber
\end{eqnarray}
\end{remark}

\subsection{Basic results}
In this subsection, we recall several auxiliary results
concerning the solution of (\ref{canonicSDE})--(\ref{canonicFBSDE})
that will
become useful later. These results follows from \citet{Pardoux1999} and
\citet{BriandCarmona2000}.

%
\begin{theorem}[(Existence and uniqueness)]
\label{theo-existenceuniquenesscanonicalFBSDE}
Let \textup{(HX0)} and \textup{(HY0)} hold. Then FBSDE (\ref{canonicSDE})--(\ref{canonicFBSDE})
has a unique solution $(X,Y,Z) \in\cS^p\times\cS^{p} \times
\cH^{p}$ for any $p\geq2$. Moreover, it holds for some constant
$C_p>0$ that
%
\begin{eqnarray}\label{eqmoment-estim-for-polyy-BSDE}
\llVert Y\rrVert_{\cS^{p}}^{p}
+\llVert Z
\rrVert_{\cH^{p}}^{p} &\leq& C_p \bigl\{ \bigl
\llVert g(X_T)\bigr\rrVert_{L^{p}}^{p} + \bigl
\llVert f(\cdot,X_\cdot,0,0) \bigr\rrVert_{\cH^{p}}^{p}
\bigr\}
\nonumber\\[-8pt]\\[-8pt]\nonumber
&\leq& C_p\bigl(1+\llvert x\rrvert^{p}\bigr).\nonumber
\end{eqnarray}
\end{theorem}

\begin{pf}
The existence and uniqueness results for SDE (\ref{canonicSDE}) follow from
standard SDE literature. The existence and uniqueness result for the BSDE
follows from Proposition 2.2 in \citet{Pardoux1999}, since the SDE results
imply
that $X\in\cS^p$ for any $p \geq2$, along with linear growth in $x$
of $g$
and $f$. The estimates for $Y\in\cS^p$ for any $p\geq2$ and $Z\in
\cH^p$
follow from the pathwise inequality
%
\begin{eqnarray}
\label{eqpathwise-estim-for-polyy-BSDE} && \llvert Y_t\rrvert^{2} +
\biggl(1-
\frac{3L_z^2}{2\alpha} \biggr) \bE_t \biggl[ \int_t^T
\llvert Z_u\rrvert^2\,\udu\biggr]
\nonumber\\[-8pt]\\[-8pt]\nonumber
&& \qquad \leq C_{\alpha,T,t} \bE_t \biggl[
\bigl\llvert g(X_T)\bigr\rrvert^{2} +\int
_t^T \frac{3}{4\alpha} \bigl\llvert
f(u,X_u,0,0)\bigr\rrvert^{2}\,\udu\biggr],
\end{eqnarray}
where $C_{\alpha,T,t}=\exp\{{2(L_y+\alpha)(T-t)}\}$, for any $\alpha
> 0$
and
$t\in[0,T]$. This last inequality follows from the proof of
Proposition 2.2
and
Exercise 2.3 in \citet{Pardoux1999} [see also Theorem 3.6 in
\citet{BriandCarmona2000}].
\end{pf}

We now state a result concerning a priori estimates for BSDEs.

%
\begin{theorem}[(A priori estimate)]
\label{theo-aprioriestimate}
Let $p\geq2$ and for $i\in\{1,2\}$, let $\Theta^i=(X^i,Y^i,Z^i)$ be
the solution
of FBSDE (\ref{canonicSDE})--(\ref{canonicFBSDE}) with functions
$b^i,\sigma^i,g^i,f^i$ satisfying \textup{(HX0)--(HY0)}. Then there
exists $C_p>0$ depending only on $p$ and the constants in the
assumptions such that for $i\in\{1,2\}$
%
\begin{eqnarray}\label{eqapriori-for-diff-of-polyy-BSDE}
&& \bigl\llVert Y^1-Y^2\bigr
\rrVert
_{\cS^{p}}^{p} +\bigl\llVert Z^1-Z^2
\bigr\rrVert_{\cH^{p}}^{p}\nonumber
\\
&&\qquad \leq
C_p \biggl\{ \bE\biggl[ \bigl\llvert g^1
\bigl(X^1_T\bigr)-g^2\bigl(X^2_T
\bigr) \bigr\rrvert^{p}
\\
\nonumber
&&\hspace*{31pt}\quad\qquad{}+ \biggl( \int_0^T
\bigl\llvert f^1\bigl(s,X^1_s,Y^i_s,Z^i_s
\bigr)- f^2\bigl(s,X^2_s,Y^i_s,Z^i_s
\bigr)\bigr\rrvert\,\uds\biggr)^{p} \biggr] \biggr\}.
\end{eqnarray}
\end{theorem}

\begin{pf}
See Proposition 3.2 and Corollary 3.3 in
\citet{BriandCarmona2000}.
\end{pf}

%
\begin{corollary}[(Markov property and sample path continuity)]
\label{coromarkov+cont}
Let \textup{(HX0)} and \textup{(HY0)} hold.
The mapping $(t,x)\mapsto Y^{t,x}_t(\omega)$ is continuous. There
exist two
$\cB([0,T]) \otimes\cB(\bR^k)$ and $\cB([0,T])\otimes
\cB(\bR^{k\times d})$ measurable deterministic functions $u$ and $v$
(resp.) s.t.
%
\begin{eqnarray}
\label{Yismarkovian} Y^{t,x}_s&=&u\bigl(s,X^{t,x}_s
\bigr),\qquad s\in[t,T], \ud\bP\mbox{-a.s.},
\nonumber\\[-8pt]\\[-8pt]\nonumber
Z^{t,x}_s&=& v\bigl(s,X^{t,x}_s
\bigr) \sigma\bigl(s,X^{t,x}_s\bigr),\qquad s\in[t,T], \ud\bP
\times\ud s\mbox{-a.s.}
\end{eqnarray}
Moreover, the Markov property holds
$Y^{t,x}_{t+h}=Y^{t+h,X^{t,x}_{t+h}}_{t+h}$
for any $h\geq0$ and $u \in C^{0,0}([0,T]\times\bR^k)$.
\end{corollary}

\begin{pf}
See Section~3 in \citet{Pardoux1999}. The sample path continuity of
$Y^{t,x}_t$ follows
from the mean-square continuity of $(Y^{t,x}_s)_{s\in[t,T]}$ for $x\in
\bR^k$,
\mbox{$0\leq t\leq s\leq T$}, which in turn follows from inequality
(\ref{eqapriori-for-diff-of-polyy-BSDE}),
combined with the Lipschitz property of $x\mapsto g(x)$ and
$(t,x)\mapsto
f(t,x,\cdot,\cdot)$ along with the continuity properties of
$(t,x)\mapsto X^{t,x}_\cdot$ solution to (\ref{canonicSDE}).

The Markov property follows from Remark 3.1 \citet{Pardoux1999} and the
continuity of $u(t,x)$ is implied by that of $Y^{t,x}_t$.
\end{pf}

\subsection{Nonlinear Feynman--Kac formula}\label{subsec-nonlinearfeymankacformula}

As pointed out in the\break \hyperref[sec1]{Introduction}, our aim is to deepen the connection
between FBSDEs and PDEs via the so-called nonlinear Feynman--Kac
formula, that is,
we study the probabilistic representation of the solution to a
class of parabolic PDEs on $\mathbb{R}^{k}$ with polynomial growth
coefficients that are associated with FBSDE
(\ref{canonicSDE})--(\ref{canonicFBSDE}). For
$(t,x)\in[0,T]\times\bR^d$, denote by $\cL$ the infinitesimal
generator of
the Markov process $X^{t,x}$ solution to (\ref{canonicSDE})
%
\begin{equation}
\cL:=\frac{1}2 \sum_{i,j=1}^d
\bigl(\bigl[\sigma\sigma^*\bigr]_{ij}\bigr) (t,x) \partial_{x_i x_j}^2
+\sum_{i=1}^d b_i(t,x)
\partial_{x_i},
\end{equation}
and consider for a function $v=(v_1,\ldots,v_k)$ the following system of
backward semi-linear parabolic PDEs for $i\in\{1,\ldots,k\}$: $v(T,x)=g(x)$ and
%
\begin{equation}
\label{eqviscosityFBSDE} -\partial_t v_i(t,x) - \cL
v_i(t,x)-f_i \bigl(t,x,v(t,x),(\nabla v \sigma) (t,x)
\bigr)=0. 
\end{equation}
In rough, it can be easily proved using It\^o's formula that if $v\in
C^{1,2}([0,T]\times\bR^d; \bR^k)$ solves the above PDE then
$Y_t:=v(t,X_t)$ and $Z_t:= (\nabla v \sigma)(t,X_t)$ solves BSDE
(\ref{canonicFBSDE}) [see Proposition 3.1 in \citet{Pardoux1999}]. But
the more
interesting result is the converse one, that is, that
$u(t,x):=Y^{t,x}_t$ is the
solution of the PDE (in some sense). It was established in
Theorem 3.2 of \citet{Pardoux1999} (recalled next) that indeed
$(t,x)\mapsto
Y^{t,x}_t$ is the
viscosity solution of the PDE.

\begin{theorem}
\label{theo-viscosity}
Let \textup{(HX0)}, \textup{(HY0)} hold and take $(t,x)\in[0,T] \times\bR^d$.
Furthermore, assume that the $i$th
component of the driver function $f$ depends only on the $i$th row of the
matrix $z\in\bR^{k\times d}$, that is, $f_i(t,x,y,z)=f_i(t,x,y,z^i)$.

Then $u(t,x):=Y^{t,x}_t$ is a continuous function of $(t,x)$ that grows at
most polynomially at infinity and is a viscosity solution of
(\ref{eqviscosityFBSDE}) [in the sense of
Definition~3.2 in \citet{Pardoux1999}].
\end{theorem}

%
\begin{remark}[(Multi-dimensional case)]
The proof of Theorem~\ref{theo-viscosity} relies on a BSDE comparison theorem
that holds only in the case $k=1$ (i.e., when $Y$ is one-dimensional).
Nonetheless,
with the restriction imposed by \textup{(HY0)},
it is still possible to use the
said comparison theorem to prove Theorem~\ref{theo-viscosity}, we
point the
reader to Theorem 2.4 and Remark 2.5 in \citet{Pardoux1999}.
\end{remark}

It is possible to show that $(t,x)\mapsto Y^{t,x}_t$ is the solution to
(\ref{eqviscosityFBSDE}) not only in the viscosity sense, but
also in weak sense (in weighted Sobolev spaces), this has been done in
\citet{MatoussiXu2008} and \citet{ZhangZhao2012}.

\subsection{Examples}

One equation covered by our setting is the FitzHugh--Nagumo PDE with
recovery, used in biology and related to the modeling of the electrical
distribution of the heart or the potential in neurons.

\begin{example}[(The FH--N equation with recovery)]
Let $(t,x)\in[0,T]\times\bR^d$, $g=(g_u, g_v)$, $f=(f_u,f_v)$
and $g,f,(u,v)\dvtx [0,T]\times\bR^d\to\bR^2$. The FH--N PDE has the dynamics:
$u(T,\cdot)=g_u(\cdot)$, $ v(T,\cdot)=g_v(\cdot)$ and
\[
-\partial_t u - \tfrac{1}{2}\Delta u- f_u(u,v)=
0,\qquad-\partial_t v - \Delta v-f_v(u,v) = 0,
\]
where $f_u(u,v)=u-u^3+v$ and
$f_v(u,v)= u-v$. $f$ clearly satisfies \textup{(HY0)} and \textup{(HY0$_{\mathrm{loc}}$)}.
\end{example}

A simpler setup of the above model is its one-dimensional version.

\begin{example}[(FH--N equation without recovery)]
\label{exampleFH-N}
For $(t,x)\in[0,T]\times\bR$ the FH--N equation without recovery is
described by
%
\begin{equation}
\label{FH--N} -\partial_t u-\tfrac{1}{2}\Delta u-\bigl(c
u^3+ b u^2 -a u \bigr)=0, \qquad u(T,x)=g(x).
\end{equation}
When $c=-1$, $b=1+a$, $a\in\bR$ and with the choice of
$g(x)= ( 1+ e^x )^{-1}$,
one can verify that the $C^\infty_b$ solution $u$ to (\ref{FH--N}) is
given by
%
\begin{equation}
\label{exampleSOLtoPDE} \quad u(t,x)= \bigl( 1+ \exp\bigl\{ x - (1/2-a)
(T-t) \bigr\}
\bigr)^{-1}\in C^\infty_b\bigl([0,T]\times\bR
\bigr).
\end{equation}
The FBSDE corresponding to this PDE is given by
(\ref{canonicSDE})--(\ref{canonicFBSDE}) with the
following data:
\begin{eqnarray*}
b(\cdot,\cdot)&=&0;\qquad
\sigma(\cdot,\cdot)=1;\qquad
f(t,x,y,z)=cy^3+by^2-ay;\qquad
\\
c&=&-1;\qquad
b=1+a,
\end{eqnarray*}
%
and the terminal condition function $g$ is given
above.
Both \textup{(HX0)} and \textup{(HY0$_{\mathrm{loc}}$)} hold (for any $a$, notice that
$u\geq
0$ for any $a$) and the
theory we develop throughout applies to this class of examples.
We will use the case $a=-1$ in our simulations.
\end{example}

\section{Representation results, path regularity and other properties}\label{secregularity}
As seen before $u(t,x):=Y^{t,x}_t$ is a viscosity solution of PDE
(\ref{eqviscosityFBSDE}). If $u\in C^{1,2}$, we would also
obtain the representation of the process $Z$ as $Z^{t,x}_t=(\nabla_x u
\sigma)(t,x)$,
but in view of Theorem~\ref{theo-viscosity} we have not given meaning
to $\nabla_x u$. The
main aim of this section is to first prove some representation formulas,
that express $Z$ as a function of $Y$ and $X$, then use
these representation formulas to obtain the so-called $L^2$- (and $L^p$-)
path regularity
results needed to prove the convergence of the numerical discretization of
FBSDE (\ref{canonicSDE})--(\ref{canonicFBSDE}) in the later sections. A
by-product of these results is the existence of $\nabla_x u$.

\subsection{Differentiability in the spatial parameter}
Take the system (\ref{canonicSDE})--(\ref{canonicFBSDE}) into
account. We now
show that the smoothness of the FBSDE parameters $b,\sigma,g,f$
carries over
to the solution process $\Theta=(X,Y,Z)$.

%
\begin{theorem}
\label{theoexistencenablay}
Let \textup{(HXY1)} hold and $(t,x)\in[0,T]\times\bR^d$.

Then $u$ [from (\ref{Yismarkovian})] is continuously differentiable in its
spatial
variable. Moreover, the triple $\nabla_x
\Theta^{t,x}=(\nabla_x X^{t,x},\nabla_x Y^{t,x},\nabla_x
Z^{t,x})\in
\cS^p\times\cS^p\times\cH^p$ for any $p\geq2$ and solves for
$0\leq t\leq
s\leq T$
%
\begin{equation}
\label{differentiatedFBSDE} \cases{
\displaystyle\nabla_x X^{t,x}_s = I_d +\int_t^s (
\nabla_x b) \bigl(r, X^{t,x}_r\bigr)
\nabla_x X^{t,x}_r \,\udr
\vspace*{3pt}\cr
\displaystyle \phantom{\nabla_x X^{t,x}_s =}{} +\int
_t^s (\nabla_x \sigma) \bigl(r,
X^{t,x}_r\bigr)\nabla_x X^{t,x}_r\,\udwr,
\vspace*{3pt}\cr
\displaystyle\nabla_{x_i} Y_s^{t,x} = (
\nabla_x g) \bigl(X_T^{t,x}\bigr)
\nabla_{x_i} X^{t,x}_T-\int_s^T
\nabla_{x_i} Z_r^{t,x}\,\ud W_r
\cr
\displaystyle\phantom{\nabla_{x_i} Y_s^{t,x} =}{} +
\int_t^T F\bigl(r,\nabla_{x_i}
\Theta_r^{t,x}\bigr) \,\ud r}
\end{equation}
for $i\in\{1,\ldots,d\}$ and with\footnote{The term $(\nabla_z
f)(\cdot,\Theta) \cdot\Gamma$ can be better understood if one
interprets $z$ in $f$ not as in $\bR^{k\times d}$ but as $(\bR^d)^k$, that is,
$f$ receives not a matrix but its $\bR^d$-valued $k$ lines.}
\begin{eqnarray*}
&& F\dvtx  (\omega,r,x,\chi, \Upsilon, \Gamma)
\\
&&\qquad \mapsto(\nabla_x f )
\bigl(r,\Theta^{t,x}_r\bigr) \cdot\chi+ (
\nabla_y f ) \bigl(r,\Theta^{t,x}_r\bigr)
\cdot\Upsilon+ (\nabla_z f ) \bigl(r,\Theta^{t,x}_r
\bigr) \cdot\Gamma.
\end{eqnarray*}
There exists a positive constant $C_p$ independent of $x$ such that
%
\begin{equation}
\label{estimatenablaynablaz} \sup_{(t,x)\in[0,T] \times\bR^d} \bigl
\llVert\bigl(
\nabla_x Y^{t,x},\nabla_x Z^{t,x}
\bigr)\bigr\rrVert_{\cS
^p\times\cH^p}\leq C_p.
\end{equation}
Furthermore, for $u$ as in (\ref{Yismarkovian}) we have for $x\in\bR
^d$ and
$0\leq t\leq s\leq T$
%
\begin{eqnarray}
\label{repynablaunablax} \nabla_x Y^{t,x}_s&=&(
\nabla_x u) \bigl(s,X^{t,x}_s\bigr)
\nabla_x X^{t,x}_s, \qquad\mbox{$\bP$-a.s.}\quad\mbox{and}
\nonumber\\[-8pt]\\[-8pt]
\nonumber
\llVert\nabla_x u\rrVert_\infty&<&\infty.
\end{eqnarray}
\end{theorem}

We recall that $\nabla_x Y^{t,x}$ is $\bR^{k\times d}$-valued and
$\nabla_{x_i} Y^{t,x}$ denotes its $i$th column we use a similar notation follows
for $\nabla_x X$ and $\nabla_x Z$.

\begin{pf*}{Proof of Theorem \ref{theoexistencenablay}}
Throughout fix $(t,x)\in[0,T]\times\bR^d$ and let
$\{e_i\}_{i\in\{1,\ldots,d\}}$ be the canonical unit vectors of $\bR
^d$. Let
$i\in\{1,\ldots,d\}$.

The results concerning SDE (\ref{canonicSDE}) follow from those in
Section~2.5 in \citet{ImkellerDosReis2010}. We start by showing that
the partial
derivatives $(\nabla_{x_i} Y^{t,x},\nabla_{x_i} Z^{t,x})$ for any
$i$ exist, then we will show the full differentiability. We
start by proving that (\ref{differentiatedFBSDE}) has indeed a
solution for
every $i$. Unfortunately, the driver of~(\ref{differentiatedFBSDE})
does not satisfy \textup{(HY0)}, and hence we cannot quote Theorem~\ref
{theo-existenceuniquenesscanonicalFBSDE} directly; we use a more
general result from \citet{BriandDelyonHuEtAl2003}. We remark though,
that the techniques used to obtain moment estimates of the form of
(\ref{eqmoment-estim-for-polyy-BSDE}) and (\ref
{eqapriori-for-diff-of-polyy-BSDE}) are the same in both
\citet{BriandDelyonHuEtAl2003} and \citet{Pardoux1999}.

FBSDE (\ref{differentiatedFBSDE}) has a unique solution
$\Xi^{t,x,i}:= (\nabla_{x_i} X^{t,x}, U^{t,x,i}, V^{t,x,i}) \in\cS
^p \times
\cS^p \times\cH^p$ for any $p\geq2$, where $(U^i,V^i)$ replaces
$(\nabla_{x_i} Y,\nabla_{x_i} Z)$.
This follows by a direct application of Theorem 4.2 in
\citet{BriandDelyonHuEtAl2003}. It is easy to see that under
\textup{(HXY1)} the conditions (H1)--(H5) in \citeauthor{BriandDelyonHuEtAl2003} [(\citeyear{BriandDelyonHuEtAl2003}), pages~\mbox{118--119}]
are satisfied. First, under \textup{(HXY1)}, standard SDE theory [see, e.g.,
Theorem 2.4 in \citet{ImkellerDosReis2010}] ensures that $\nabla_x
X\in\cS^p$
for all $p\geq2$, which along with $\nabla_x g, \nabla_x f\in
C^{0,0}_b$, implies in turn that the terminal condition $(\nabla_x
g)(X_T^{t,x})\nabla_{x_i} X^{t,x}_T\in L^p_{\cF_T}$ and the term
$ (\nabla_x
f ) (\cdot,\Theta^{t,x}_\cdot) \nabla_{x_i} X^{t,x}_\cdot=
F(\cdot,\nabla_{x_i} X^{t,x}_\cdot,0,0) \in\cS^p$ for any $p\geq
2$. Given
the linearity of $F$ and the Lipschitz property of $f$ in its $z$-variable,
it follows that $F$ is uniformly Lipschitz in
$\Gamma$.\vadjust{\goodbreak} Moreover, since $f$ satisfies
(\ref{HY0monotonicity}) it implies that $F$ is monotone\footnote{This
follows easily from the differentiability of $f$, its monotonicity in
$y$ and
the definition of directional derivative.} in $\Upsilon$, that is,
%
\begin{eqnarray}
\quad&& \bigl\langle\Upsilon-\Upsilon', (\nabla_y f )
\bigl(\cdot,\Theta^{t,x}_\cdot\bigr) \cdot\bigl(\Upsilon-
\Upsilon'\bigr) \bigr\rangle\leq L_y \bigl\llvert
\Upsilon-\Upsilon'\bigr\rrvert^2
\qquad \forall \Upsilon,\Upsilon' \in\bR^k.
\end{eqnarray}
The continuity of $\Upsilon\mapsto F(r,x,\chi,\Upsilon,\Gamma)$ is
also clear.
Finally, the linearity of $F$, the fact that
$\Theta\in\cS^p\times\cS^{p} \times\cH^{p}$ for any $p\geq
2$ and (\ref{HY0loc-loclip}) implies that condition
(H5) in \citet{BriandDelyonHuEtAl2003} is also satisfied, that is, that
for any
$R> 0$, $\sup_{\llvert\Upsilon\rrvert \leq R}\llvert
F(r,x,\nabla_{x_i}X^{t,x}_r,\Upsilon,0)
-F(r,x,\nabla_{x_i}X^{t,x}_r,0,0)\rrvert \in L^1([t,T]\times
\Omega)$.\vspace*{1pt}
We are therefore under the conditions of
Theorem 4.2 in \citet{BriandDelyonHuEtAl2003}, as claimed.

In view of (\ref{remmomenttrick}) and the linearity of $F$ one can
obtain moment estimates in the style of
(\ref{eqmoment-estim-for-polyy-BSDE}) by following arguments similar to
those in the proof of Theorem~\ref{theo-existenceuniquenesscanonicalFBSDE}
[recall that (\ref{remmomenttrick}) takes in this case a very simple form].
In view of (\ref{eqmoment-estim-for-polyy-BSDE}), we have (recall that
$\nabla X\in\cS^p$ for all $p\geq2$)
%
\begin{eqnarray}
\nonumber
\label{eq-boundsUV} && \bigl\llVert U^i\bigr\rrVert
_{\cS^{p}}^{p} +\bigl\llVert V^i \bigr\rrVert
_{\cH^{p}}^{p}
\\
&&\qquad \leq C_p \bigl\{ \bigl\llVert(
\nabla_x g) \bigl(X_T^{t,x}\bigr)
\nabla_{x_i} X^{t,x}_T\bigr\rrVert
_{L^{p}}^{p} + \bigl\llVert(\nabla_x f ) \bigl(
\cdot,\Theta^{t,x}_\cdot\bigr) \nabla_{x_i}
X^{t,x}_\cdot\bigr\rrVert_{\cH^{p}}^{p}
\bigr\}
\\
&& \qquad\leq C_p \bigl\llVert\nabla_{x_i}
X^{t,x}\bigr\rrVert_{\cS^{p}}^{p} \leq
C_p,\nonumber
\end{eqnarray}
where $C_p$ does \emph{not} depend on $x$, $t$ or $i$.

In order to obtain results on the first-order variation of the
solution, we
follow standard BSDE techniques used already in \citet{ImkellerDosReis2010},
\citet{BriandConfortola2008a} or \citet{dosReisReveillacZhang2011}; we
start by studying the behavior of $\Theta^{t,x+\varepsilon e_i}
-\Theta^{t,x}$ for any $\varepsilon>0$. Take
$h\in\bR^d$. Via the stability of SDEs and inequality
(\ref{eqapriori-for-diff-of-polyy-BSDE}) [and \textup{(HY0)}], it is clear
that a
constant $C_p>0$ independent of $x$ exists such that
%
\begin{eqnarray}\label{eqstabilitythetavarepsilontheta}
\lim_{h \to0} \bigl\llVert
\Theta^{t,x+h}
-\Theta^{t,x}\bigr\rrVert_{\cS^p\times
\cS^p\times\cH^p}
&\leq& \lim
_{h\to0} C_p \bigl\llVert X^{x+h}-X^x
\bigr\rrVert_{\cS^p}
\nonumber\\[-8pt]\\[-8pt]\nonumber
&\leq& \lim_{h\to0} C_p \llvert h \rrvert= 0.
\end{eqnarray}
Define
\begin{eqnarray*}
\delta\Theta^{\varepsilon, i} &:=&\bigl(\delta X^{\varepsilon,i},\delta
Y^{\varepsilon,i},\delta Z^{\varepsilon,i}\bigr)
\\
&:=&\bigl(\Theta
^{t,x+\varepsilon
e_i}-
\Theta^{t,x}\bigr)/\varepsilon- \bigl(\nabla_{x_i}
X^{t,x},U^{t,x,i},V^{t,x,i}\bigr)
\end{eqnarray*}
for which
%
\begin{eqnarray}
\label{aux-FBSDEfordifference}
\nonumber
\delta Y^{\varepsilon,i}_s &=& \biggl[
\frac{1}\varepsilon\bigl(g\bigl(X^{t,x+\varepsilon e_i}_T\bigr) - g
\bigl(X^{t,x}_T\bigr) \bigr) -(\nabla_x g)
\bigl(X_T^{t,x}\bigr)\nabla_{x_i}
X^{t,x}_T \biggr] -\int_s^T
\delta Z^{\varepsilon,i}_r\,\udw_r\hspace*{-10pt}
\\
&&{} + \int_s^T \biggl[
\frac{1}\varepsilon\bigl(f\bigl(r,\Theta_r^{t,x+\varepsilon e_i}
\bigr)- f\bigl(r,\Theta_r^{t,x}\bigr) \bigr)
\\
&&\hspace*{33pt}{} - F\bigl(r,x,
\nabla_{x_i} X^{t,x}_r,U^{t,x,i}_r,V^{t,x,i}_r
\bigr) \biggr]\,\ud r.\nonumber
\end{eqnarray}
Using the differentiability of the involved functions, we can
re-write (\ref{aux-FBSDEfordifference}) as a linear
FBSDE with random coefficients satisfying in its essence a \textup{(HY0)} type
assumption: for $s\in[t,T]$, $j\in\{1,\ldots,d\}$
%
\begin{equation}
\label{differencedFBSDE} \qquad\cases{
\displaystyle \delta X^{\varepsilon,j}_s = 0 +
\int_t^s \bigl[ b^{\varepsilon,j}_x(r)
\delta X^{\varepsilon,j}_r + \delta\nabla b^\varepsilon_r
\nabla_{x_j} X^{t,x}_r \bigr] \,\udr
\vspace*{3pt}\cr
\displaystyle\phantom{\delta X^{\varepsilon,j}_s =} {} +\int_t^s \bigl[
\sigma^{\varepsilon,j}_x(r)\delta X^{\varepsilon,j}_r +
\delta\nabla\sigma^\varepsilon_r \nabla_{x_j}
X^{t,x}_r \bigr] \,\udwr,
\vspace*{3pt}\cr
\displaystyle\delta Y^{\varepsilon,i}_s = \bigl[ g^{\varepsilon,i}_x(T)\delta
X^{\varepsilon,i}_T + \delta\nabla g^\varepsilon_T
\nabla_{x_i} X^{t,x}_T \bigr] -\int
_s^T \delta Z^{\varepsilon,i}_r\,\ud
W_r
\vspace*{3pt}\cr
\displaystyle\phantom{\delta Y^{\varepsilon,i}_s =}
{}+\int_s^T
\bigl[f^{\varepsilon,i}_x(r)\delta X^{\varepsilon,i}_r
+f^{\varepsilon,i}_y(r)\delta Y^{\varepsilon,i}_r
+f^{\varepsilon,i}_z(r)\delta Z^{\varepsilon,i}_r
\vspace*{3pt}\cr
\displaystyle\hspace*{107pt}{}+\delta\nabla f^\varepsilon_r \cdot\bigl(
\nabla_{x_i} X^{t,x}_r,U^{t,x,i}_r,V^{t,x,i}_r
\bigr) \bigr] \,\ud r,}
\end{equation}
where $\delta\nabla f$ and $\delta\nabla\varphi$ denote the differences
\[
\delta\nabla f^{\varepsilon}_\cdot:= \bigl( f^{\varepsilon,i}_x,f^{\varepsilon,i}_y, f^{\varepsilon,i}_z \bigr) (
\cdot) - (\nabla_x f,\nabla_y f,\nabla_z
f ) \bigl(\cdot,\Theta^{t,x}_\cdot\bigr)
\]
and
\[
\delta\nabla\varphi^{\varepsilon}_\cdot:= \varphi^{\varepsilon,i}_x
(\cdot) - \nabla_x \varphi\bigl(\cdot,\Theta^{t,x}_\cdot
\bigr),
\]
for $\varphi\in\{b,\sigma,g\}$
(with some abuse of notation) and $r\in[t,T]$,
and where we defined
\begin{eqnarray*}
\varphi^{\varepsilon,i}_x(r)&:=& \int_0^1
(\nabla_x \varphi) \bigl(r,(1-\lambda)X^{t,x}_r
+ \lambda X^{t,x+\varepsilon e_i}_r \bigr)\,\ud\lambda
\\
&=& \int_0^1 (\nabla_x\varphi)
\bigl(r, X^{t,x}_r + \lambda\bigl(X^{t,x+\varepsilon e_i}_r-X^{t,x}_r
\bigr) \bigr)\,\ud\lambda,
\end{eqnarray*}
and $f_*^{\varepsilon,i}$ for $*\in\{x,y,z\}$ in the following
way:
\begin{eqnarray*}
f^{\varepsilon,i}_z(r) &:=& \int_0^1
(\nabla_z f) \bigl(r, X^{t,x+\varepsilon e_i}_r,Y^{t,x+\varepsilon e_i}_r, Z_r^{t,x} + \lambda
\bigl(Z^{t,x+\varepsilon e_i }_r-Z^{t,x}_r\bigr)
\bigr)\,\ud\lambda,
\\
f^{\varepsilon,i}_y(r) &:=& \int_0^1
(\nabla_y f) \bigl(r, X^{t,x+\varepsilon e_i}_r,Y_r^{t,x} + \lambda\bigl(Y^{t,x+\varepsilon e_i }_r-Y^{t,x}_r
\bigr),Z_r^{t,x}\bigr)\,\ud\lambda,
\\
f^{\varepsilon,i}_x(r) &:=& \int_0^1
(\nabla_x f) \bigl(r, X^{t,x}_r + \lambda
\bigl(X^{t,x+\varepsilon e_i}_r-X^{t,x}_r\bigr),Y_r^{t,x},Z_r^{t,x}\bigr)\,\ud
\lambda.
\end{eqnarray*}
The assumptions imply immediately that
$b^{\varepsilon,i}_x,\sigma^{\varepsilon,i}_x,
f^{\varepsilon,i}_x,f^{\varepsilon,i}_z$ are uniformly bounded,
while $f^{\varepsilon,i}_y \in\cS^p$, $p\geq2$ (thanks to
HY$0_{\mathrm{loc}}$). Furthermore, using estimate
(\ref{eqmoment-estim-for-polyy-BSDE}) [along with $\llVert
X^{t,x}\rrVert^p_{\cS
^p}\leq
C_p(1+\llvert x\rrvert ^p)$], (\ref{eq-boundsUV}),
(\ref{eqstabilitythetavarepsilontheta}), the continuity of
$\varphi\in\{b,\sigma,g\}$ and its derivative it is easy to
see that, in combination with the dominated convergence theorem, one has
%
\begin{eqnarray}
\label{eqlineintegraltransstability}
&&\lim_{\varepsilon\to0} \bigl\{ \bigl\llVert
\varphi^{\varepsilon,i}_x(\cdot)-\nabla_x \varphi\bigl(
\cdot,\Theta^{t,x}_\cdot\bigr) \bigr\rrVert_{\cS^p}
\nonumber\\[-8pt]\\[-8pt]\nonumber
&&\hspace*{21pt}{} + \bigl\llVert\bigl( f^{\varepsilon,i}_x,f^{\varepsilon,i}_y, f^{\varepsilon,i}_z \bigr) (
\cdot) -(\nabla_x f,\nabla_y f,\nabla_z
f) \bigl(\cdot,\Theta^{t,x}_\cdot\bigr)\bigr\rrVert
_{\cH^p} \bigr\} =0.
\end{eqnarray}
We remark that in the above limit a localization argument for the convergence
of $f^{\varepsilon,i}_y(\cdot)$ to $\nabla_y f(\cdot,\Theta_\cdot
)$ is
required, namely that we work inside a ball (of any given radius) centered
around $x$ in which all points $x+\varepsilon e_i\in\bR^d$ as
$\varepsilon$
vanishes are contained. We do not detail the argumentation since it is
similar to that given in, for example, \citet{ImkellerDosReis2010},
\citet{BriandConfortola2008a} or \citet{dosReisReveillacZhang2011}.

With this in mind we return to (\ref{aux-FBSDEfordifference}), written
in the
form of (\ref{differencedFBSDE}), and since it is a linear FBSDE satisfying
the monotonicity condition (\ref{HY0monotonicity}) we have via Corollary 3.3
in \citet{BriandCarmona2000} [essentially our moment
estimate (\ref{eqmoment-estim-for-polyy-BSDE}) for FBSDE
(\ref{differencedFBSDE})] in combination with
(\ref{eq-boundsUV}), (\ref{eqstabilitythetavarepsilontheta})
and (\ref{eqlineintegraltransstability}), that for any $i$
\[
\lim_{\varepsilon\to0} \biggl\llVert\frac{1}\varepsilon\bigl(
\Theta^{t,x+\varepsilon e_i}-\Theta^{t,x}\bigr) - \bigl(\nabla_{x_i}
X^{t,x}, U^{t,x,i}, V^{t,x,i}\bigr) \biggr\rrVert
_{\cS^p\times\cS^p\times\cH^{p}} =0 \qquad\forall p\geq2.
\]
Since the limit exists we identify $(\nabla_{x_i}
Y^{t,x},\nabla_{x_i} Z^{t,x})$ with $(U^{t,x,i},V^{t,x,i})$ and, moreover,
estimate (\ref{eq-boundsUV}) implies estimate (\ref{estimatenablaynablaz}).
Furthermore, the above limit implies in particular that (take $s=t$)
\begin{eqnarray*}
\nabla_{x_i} u(t,x) &=& \lim_{\varepsilon\to0}
\frac{1}\varepsilon\bigl[u(t,x+\varepsilon e_i)-u({t,x})
\bigr]
\\
& = &\lim_{\varepsilon\to0} \frac{1}\varepsilon
\bigl[Y^{t,x+\varepsilon e_i}_t-Y^{t,x}_t\bigr] =
\nabla_{x_i} Y^{t,x}_t.
\end{eqnarray*}
Observing that the RHS of (\ref{eq-boundsUV}) is a constant independent
of $t\in[0,T]$, $x\in\bR^d$ and $i\in\{1,\ldots,d\}$ we can
conclude that
%
\begin{equation}
\label{equationuxisbounded} \llVert\nabla_{x_i} u\rrVert_{\infty}
=\sup
_{(t,x)\in[0,T]\times\bR^d} \bigl\llvert\nabla_{x_i} Y^{t,x}_t
\bigr\rrvert<\infty.
\end{equation}

It is clear that $(\nabla_{x_i} Y^{t,x}_s)_{s\in[t,T]}$ is continuous
in its
time parameter as it is a solution to a BSDE; we now focus on the
continuity of $x\mapsto\nabla_{x_i} Y^{t,x}_t$. Let \mbox{$x,x'\in\bR^d$}. The
difference $\nabla_{x_i} Y^{t,x}-\nabla_{x_i} Y^{t,x'}$ is the solution
to a \emph{linear} FBSDE following from (\ref{differentiatedFBSDE}). As
before, it is easy to adapt the computations and apply Corollary 3.3 in
\citet{BriandCarmona2000} [essentially\vspace*{1.5pt} our moment
estimate (\ref{eqapriori-for-diff-of-polyy-BSDE}) for FBSDEs
(\ref{differentiatedFBSDE})] to the difference $\nabla_{x_i}
Y^{t,x}_s-\nabla_{x_i} Y^{t,x'}_s$ yielding
\begin{eqnarray*}
&& \bigl\llVert\nabla_{x_i} Y^{t,x}-\nabla_{x_i}
Y^{t,x'}\bigr\rrVert_{\cS^2}^2
\\
&&\qquad \leq C_p \biggl\{ \bigl\llVert(\nabla_{x} g)
\bigl(X_T^{t,x}\bigr)\nabla_{x_i}
X^{t,x}_T - (\nabla_{x} g) \bigl(X_T^{t,x'}
\bigr)\nabla_{x_i} X^{t,x'}_T \bigr\rrVert
_{L^2}^2
\\
&&\hspace*{51pt}{}+ \bE\biggl[ \biggl(\int_0^T
\bigl\llvert F\bigl(r,x,\nabla_{x_i} X^{t,x}_r,
\nabla_{x_i} Y^{t,x}_r,\nabla_{x_i}
Z^{t,x}_r\bigr)
\\
&&\hspace*{99pt}{}- F\bigl(r,x',
\nabla_{x_i} X^{t,x'}_r,\nabla_{x_i}
Y^{t,x}_r,\nabla_{x_i} Z^{t,x}_r
\bigr) \bigr\rrvert\,\uds\biggr)^p \biggr] \biggr\}.
\end{eqnarray*}
Given the known results on SDEs, the linearity of $F$, (\ref{eq-boundsUV}),
the continuity of the derivatives of $f$ and
(\ref{eqstabilitythetavarepsilontheta}), dominated convergence theorem yields
that $\llVert\nabla_{x_i} Y^{t,x}-\nabla_{x_i} Y^{t,x'}\rrVert_{\cS
^2}^2 \to
0$ as
$x'\to x$ uniformly on compact sets. This mean-square continuity of
$\nabla_{x_i} Y^{t,x}$
implies in particular that $\nabla_{x_i} Y^{t,x}_t=\nabla_{x_i}
u(t,x)$ is
continuous. In conclusion, we just proved that for any
$i\in\{1,\ldots,d\}$ the partial derivatives $\nabla_{x_i} u$ exist
and are
continuous; hence, standard multi-dimensional real analysis implies
that $u$
is continuously differentiable in its spatial variables. This argumentation
is similar to that in the proof of Corollary~\ref{coromarkov+cont}.

We are left to prove (\ref{repynablaunablax}). Note
that for any $\varepsilon>0$ we have
$(Y^{t,x+\varepsilon e_i}_s-Y^{t,x}_s)/\varepsilon=(u(s,X^{t,
x+\varepsilon e_i} _s)-u(s,X^ { t, x } _s))/\varepsilon$. By sending
$\varepsilon\to0$ and using the (continuous) differentiability of
$u$, we
have $\nabla_x Y^{t,x}_s=(\nabla_x u)(s,X^{t,x}_s)\nabla_x
X^{t,x}_s$. Hence, as the RHS of (\ref{eq-boundsUV}) is a constant
independent
of $t\in[0,T]$, $x\in\bR^d$ and $i$ we can conclude (let $s\searrow t$)
that
$\llVert\nabla_x u\rrVert_{\infty}
=\sup_{(t,x)\in[0,T]\times\bR^d} \llvert\nabla_x
Y^{t,x}_t\rrvert <\infty$.
\end{pf*}

\subsection{Malliavin differentiability}

As in the previous section, we show a form of regularity of the solution
$\Theta$ to (\ref{canonicSDE})--(\ref{canonicFBSDE}), namely the stochastic
variation of $\Theta$ in the sense of Malliavin's calculus.

%
\begin{theorem}[(Malliavin differentiability)]
\label{theoexistencemalliavin}
Let \textup{(HXY1)} hold. Then the
solution $\Theta=(X,Y,Z)$ of (\ref{canonicSDE})--(\ref
{canonicFBSDE}) verifies:
\begin{itemize}
\item
$X\in\mathbb{L}^{1,2}$ and
$DX$ admits a version
$(u,t) \mapsto D_u X_t$ satisfying for $0\le u\le t\le T$
\begin{eqnarray*}
D_u X_t &=& \sigma(u, X_{u})+\int
_u^t (\nabla_x b) (s,
X_s) D_u X_s \,\uds+\int_u^t
(\nabla_x \sigma) (s, X_s) D_u
X_s\,\udws.
\end{eqnarray*}
Moreover, for any $p\geq2$ there exists $C_p>0$ such that
%
\begin{eqnarray}
\label{DXSpestimate} \sup_{u\in[0,T]}\llVert D_u X\rrVert
_{\cS^p}^p&\leq& C_p\bigl(1+\llvert x\rrvert
^p\bigr).
\end{eqnarray}
\item
For any $0\leq t\leq T$, $x\in\bR^m$ we have $(Y,Z)\in
\mathbb{L}^{1,2}\times(\mathbb{L}^{1,2} )^d$.
A version of $(D Y,D Z)_{0\leq u,t\leq T}$ satisfies: for $t<u\le T$,
$D_u Y_t=0$ and $D_u Z_t = 0$, and for~$0 \leq u\leq t$,
%
\begin{eqnarray}\label{aux1order-mall-dif-bsde}
D_u Y_t &=& (\nabla_x g)
(X_T) D_u X_T+ \int_t^T
\bigl\langle(\nabla f) (s,\Theta_s), D_u
\Theta_s \bigr\rangle\,\ud s
\nonumber\\[-8pt]\\[-8pt]\nonumber
&&{}  - \int_t^T
D_u Z_s \,\ud W_s.
\end{eqnarray}
Moreover, $(D_t Y_t)_{0\le t
\leq T}$ defined by the above equation is a version of\break $(Z_t)_{0\le
t\leq
T}$.
\item The following representation holds for any $0\leq u\leq t\leq T$ and
$x\in\R^m$:
%
\begin{eqnarray}
\label{repDXnablaXnablaXsigma} D_u X_t &=& \nabla_x
X_t (\nabla_x X_u)^{-1}
\sigma(u, X_u)\1_{[0,u]}(t),
\\
\label{repDYnablaYnablaXsigma} D_u Y_t &=& \nabla_x
Y_t (\nabla_x X_u)^{-1}
\sigma(u,X_u),\qquad\mbox{a.s.},
\\
\label{z-rep} Z_t&=&\nabla_x Y_t (
\nabla_x X_t)^{-1}\sigma(s,X_t),
\qquad\mbox{a.s.}
\end{eqnarray}
\end{itemize}
\end{theorem}

%
\begin{remark}[($Y$ is already in $\bL^{1,2}$)]
\label{remYalreadMalldiff}
Via Theorem~\ref{theoexistencenablay}, we know that $u\in
C^{0,1}$. Under \textup{(HXY1)} it is known that $X\in
\bL^{1,2}$ [see \citet{nualart2006}], hence using the chain rule [for
Malliavin calculus, see Proposition 1.2.3 in \citet{nualart2006}] we
obtain $Y_\cdot=u(\cdot,X_\cdot)\in
\bL_{1,2}$. A careful analysis of Theorem~\ref{theoexistencenablay}
and the
results about $\nabla_x u$ show that indeed $X,Y\in\bL^{1,p}$ for
all $p\geq2$
[just combine (\ref{DXSpestimate}) with (\ref{jensenforMallCalculus}) as
described in Appendix~\ref{appendix-malliavin-calculus}].

Using the fact that $X,Y\in\bL^{1,2}$, the statement of Theorem
\ref{theoexistencemalliavin} follows easily if the driver $f$ in
(\ref{canonicFBSDE}) does not depend on $z$. One would argue in the
following way: for any $t\in[0,T]$
\begin{eqnarray*}
&& \biggl(g(X_T) -Y_t + \int
_t^T f(r,X_r,Y_r)
\,\udr\biggr)_{t\in[0,T]} \in\bL^{1,2}
\\
&&\qquad  \Rightarrow\biggl(\int_t^T
Z_r\,\udwr\biggr)_{t\in[0,T]} \in\bL^{1,2}
\Leftrightarrow Z\in\bL^{1,2},
\end{eqnarray*}
this follows from the definition of the BSDE (\ref{canonicFBSDE})
itself and
Theorem~\ref{malldiffstochintegrals}. The dynamics of
(\ref{aux1order-mall-dif-bsde}) and the representation formulas
(\ref{repDYnablaYnablaXsigma}), (\ref{z-rep}) follow by arguments
similar to
those given below.
\end{remark}

\begin{pf*}{Proof of Theorem~\ref{theoexistencemalliavin}}
The first part of the statement is trivial as it follows from standard SDE
theory; see, for example, \citet{nualart2006} or Theorem 2.5 in
\citet{ImkellerDosReis2010}. To prove the other statements of the
theorem, we
will use
an identification trick by taking advantage of the fact we already know that
$Y\in\bL^{1,2}$ (see Remark~\ref{remYalreadMalldiff}).

Let $(X,Y,Z)$ be the solution of (\ref{canonicSDE})--(\ref
{canonicFBSDE}) and
define the following BSDE:
%
\begin{equation}
\label{auxTrickidentifiedBDSE} U_t= g(X_T)+\int_t^T
\widehat{f}(r,V_r)\,\udr- \int_t^T
V_r\,\udwr,
\end{equation}
where the driver $\widehat{f}\dvtx \Omega\times[0,T]\times\bR^d\to\bR
$ is
defined as
%
\begin{equation}
\widehat{f}(t,v):=f(t,X_t,Y_t,v)=f
\bigl(t,X_t,u(t,X_t),v \bigr).
\end{equation}
It\vspace*{2pt} is clear that: $g(X_T)\in\bD^{1,2}$, $f(\cdot,X_\cdot,Y_\cdot,0)\in
\bL^{1,p}$ for all $p\geq2$ (see Remark~\ref{remYalreadMalldiff})
and that
$v\mapsto\widehat{f}(\cdot,v)$ is a Lipschitz continuous function,
all these
imply in particular via Lipschitz BSDE theory [see
Theorem 2.1, Proposition 2.1 in \citet{ElKarouiPengQuenez1997}] that there
exists a pair $(U,V)\in\cS^2\times\cH^2$ solving
(\ref{auxTrickidentifiedBDSE}). Furthermore, Theorem
\ref{theo-existenceuniquenesscanonicalFBSDE} in
\citet{ElKarouiPengQuenez1997} states that the solution to
(\ref{canonicFBSDE}) is unique, and hence the solution of
(\ref{auxTrickidentifiedBDSE}) verifies $(U,V)=(Y,Z)$.

Proposition 5.3 in \citet{ElKarouiPengQuenez1997}, yields the existence
of the Malliavin derivatives $(DU,DV)$ of $(U,V)$ with the following
dynamics.
Set $\Xi:=(X,Y,V)$, then for $t<u\le T$ we have $D_u U_t = 0$, $D_u
V_t =
0$\vadjust{\goodbreak} and for $0 \leq u\leq t$
\begin{eqnarray*}
D_u U_t &=& (\nabla_x g) (X_T)
D_u X_T + \int_t^T
\bigl\langle(\nabla f) (s,\Xi_s), (D_u \Xi_s)
\bigr\rangle\,\ud s - \int_t^T D_u
V_s \,\ud W_s.
\end{eqnarray*}
Since $(U,V)=(Y,Z)$ then from the above BSDE for $(DU,DV)$ follows BSDE~(\ref{aux1order-mall-dif-bsde}). Moreover, Proposition 5.9 in
\citet{ElKarouiPengQuenez1997} yields (\ref{repDYnablaYnablaXsigma}) and
(\ref{z-rep}) for $(U,V)$ which carry out for $(Y,Z)$.
\end{pf*}

\subsection{Representation results}
Here, we combine the results of the two previous subsections to obtain
representation formulas that will allow us to establish the path regularity
properties of $Y$ and $Z$ required for the convergence proof of the numerical
discretization.

\begin{theorem}
\label{theo-ZinSp}
Let $\textup{(HXY1)}$ hold, then the following representation holds:
%
\begin{eqnarray}
\label{eqrepZNo1} Z^{t,x}_s &=& (\nabla_x u
\sigma) \bigl(s,X^{t,x}_s\bigr),\qquad  0\leq t\leq s\leq T, \ud\bP\mbox{-a.s.},
\\
\label{eqrepZNo2} &=& \nabla_x Y^{t,x}_s \bigl(
\nabla_x X^{t,x}_s \bigr)^{-1}\sigma
\bigl(s,X^{t,x}_s\bigr),\qquad  0\leq t\leq s\leq T, \ud
\bP\mbox{-a.s.},
\end{eqnarray}
and $\llVert Z\rrVert_{\cS^q}^q\leq C_q(1+\llvert
x\rrvert ^q)$, $q\geq2$.

Assume that only \textup{(HX0)} and \textup{(HY0$_{\mathrm{loc}}$)} hold, then for some
$C>0$ it holds
$\llvert Z_t\rrvert \leq C\llvert\sigma(X_t)\rrvert$ $\udt\otimes
\ud\bP$-a.s. and in particular
%
\begin{equation}
\label{ZSpestimate} \llvert Z_t\rrvert\leq C\bigl(1+\llvert
X_t\rrvert\bigr),\qquad\udt\otimes\ud\bP\mbox{-a.s.}
\end{equation}
\end{theorem}

\begin{pf}
We first prove all the results under \textup{(HXY1)}, then argue via mollification
that (\ref{ZSpestimate}) holds under \textup{(HX0)}--(HY$0_{\mathrm{loc}})$.
\begin{longlist}
\item[\emph{Proof under} \textup{(HXY1)}.]
The representation $Z=\nabla Y (\nabla X)^{-1}\sigma(\cdot,X)$
follows from
Theorem~\ref{theoexistencemalliavin}, while from Theorem
\ref{theoexistencenablay}, we have
\begin{eqnarray*}
Z^{t,x}_s &=& \nabla_x Y^{t,x}_s
\bigl(\nabla_x X^{t,x}_s\bigr)^{-1}
\sigma\bigl(s,X^{t,x}_s\bigr)
\\
&=& (\nabla_x u) \bigl(s,X^{t,x}_s\bigr)
\bigl(\nabla_x X^{t,x}_s \bigl(
\nabla_x X^{t,x}_s \bigr)^{-1} \bigr)
\sigma\bigl(s,X^{t,x}_s\bigr)
\\
&=&(\nabla_x u) \bigl(s,X^{t,x}_s\bigr) \sigma
\bigl(s,X^{t,x}_s\bigr).
\end{eqnarray*}
Since all the involved processes (in the RHS) are continuous, we can
identify $Z$ with its continuous version. Moreover, as all the
processes in
the
RHS belong to $\cS^p$ for all $p\geq2$ it follows that $Z\in\cS^p$
for all
$p\geq2$. Combining H\"older's inequality with the fact that $X,\nabla
X\in
\cS^p$ for all $p\geq2$ and estimate
(\ref{estimatenablaynablaz}), leads to (\ref{ZSpestimate}), that is,
%
\begin{eqnarray}
\nonumber
\llVert Z\rrVert_{\cS^p} &=& \bigl\llVert
\nabla_x Y^{t,x}_\cdot\bigl(\nabla_x
X^{t,x}_\cdot\bigr)^{-1}\sigma\bigl(
\cdot,X^{t,x}_\cdot\bigr)\bigr\rrVert_{\cS^p}
\\
\label{eqzinSpaux01}
&\leq& C_p \bigl\llVert
\nabla_x Y^{t,x}\bigr\rrVert_{\cS^{3p}} \bigl\llVert(
\nabla_x X)^{-1}\bigr\rrVert_{\cS^{3p}} \bigl\llVert
1+X^{t,x}\bigr\rrVert_{\cS^{3p}}
\\
&\leq& C_p\bigl(1+\llvert x\rrvert\bigr).\nonumber
\end{eqnarray}
A careful inspection of the used inequalities shows that the constant $C_p$
in~(\ref{eqzinSpaux01}) depends only on the several constants
appearing in
the assumptions \textup{(HX0)}--\textup{(HY0$_{\mathrm{loc}}$)}.
\end{longlist}

\begin{longlist}
\item[\emph{Proof of} (\ref{ZSpestimate}) \emph{under} \textup{(HX0)}--(HY0${}_{\mathrm{loc}}$).]
In this step, we rely on a standard mollification arguments similar to
those in
the proof of Theorem 5.2 in \citet{ImkellerDosReis2010}. Note that a driver
satisfying \textup{(HY0$_{\mathrm{loc}}$)} once mollified will still satisfy
assumption
\textup{(HY0$_{\mathrm{loc}}$)} with the same constants.

Take $b^n,\sigma^n,g^n,f^n$ as mollified versions of $b,\sigma,g,f$ in
their spatial variables such that the mollified functions satisfy uniformly
(in $n$) \textup{(HX0)} and \textup{(HY0$_{\mathrm{loc}}$)}, with uniform Lipschitz and
monotonicity constants. Theorem~\ref{theo-existenceuniquenesscanonicalFBSDE}
ensures that $\Theta=(X^n,Y^n,Z^n)\in\cS^p\times\cS^p\times\cH
^p$ for any
$p\geq2$ and solves (\ref{canonicSDE})--(\ref{canonicFBSDE}) with
$b^n,\sigma^n,g^n,f^n$ replacing $b,\sigma,g,f$. Since the mollified
functions satisfy \textup{\textup{(HXY1)}}, it follows from the above proof that for each
fixed $n$ we have $Z^n \in\cS^p$. Moreover, in view of
(\ref{eqapriori-for-diff-of-polyy-BSDE}) and the standard theory of
SDEs it
is rather simple to deduce that $\Theta^n\to\Theta$ as $n\to\infty
$ in
$\cS^{p}\times\cS^p\times\cH^{p }$ for all $p\geq2$. Let $u^n$
denote the solution to the PDE linked to FBSDE (\ref
{canonicSDE})--(\ref{canonicFBSDE}) with data $b^n,\sigma^n,g^n,f^n$
and we drop the superscript $(t,x)$ and work with $(X^n,Y^n,Z^n)$.

From (\ref{eqrepZNo1}), we have
$\llvert Z^{n}_s\rrvert =\llvert (\nabla_x u^n
\sigma^n )(s,X^{n}_s)\rrvert $ at least
$\uds\otimes\ud\bP$-a.s.
From (\ref{equationuxisbounded}) [or (\ref{estimatenablaynablaz})],
we can conclude that $\llvert\nabla_x Y^{t,x,n}_t\rrvert =
\llvert\nabla_x u^n(t,x)\rrvert \leq C$,
with $C$ independent of $n$, and hence quite easily that
%
\begin{equation}
\label{aux-upper-bound-for-z} \bigl\llvert Z^{n}_s\bigr\rrvert\leq C
\bigl\llvert\sigma^n\bigl(s,X^{n}_s\bigr)
\bigr\rrvert\leq C \bigl(1+\bigl\llvert X^{n}_s\bigr
\rrvert\bigr), \qquad\uds\otimes\ud\bP\mbox{-a.s.},
\end{equation}
where we last used the linear growth condition of $\sigma^n$.

Finally combine: the pointwise convergence of $\sigma^n\to\sigma$
(knowing that all $\sigma^n$ and $\sigma$ have the same Lipschitz
constant); the fact that $X^n\to X$ in $\cS^p$ (standard SDE stability
theory); and
Theorem~\ref{theo-aprioriestimate} yielding that $Z^n \to Z$ in $\cH
^p$ to conclude that (\ref{aux-upper-bound-for-z}) holds in the limit.\quad\qed
\end{longlist}\noqed
\end{pf}

%
%
%

\subsection{Path regularity results}
Now let $\pi$ be a partition of the interval $[0,T]$, say
$0=t_0<\cdots<t_i<\cdots<T_N=T$, and mesh
size $\llvert\pi\rrvert = \max_{i=0,\ldots,N-1} (\tip
-\ti)$. Given $\pi$, define
$r_\pi= \llvert\pi\rrvert / (\min_{i=0,\ldots,N-1}
(\tip-\ti) )$.

Let $Z$ be the control process in the solution to BSDE
(\ref{canonicFBSDE}), under \textup{(HX0)--(HY0)}. We define a set of random variables
$\{\bar Z_\ti\}_{\ti\in\pi}$ termwise given by
%
\begin{eqnarray}
\label{Z-bar-ti-pi} \bar{Z}_\ti&=&\frac{1}{\tip-\ti}\bE\biggl[\int
_\ti^\tip Z_s\,\uds\Big|
\cF_\ti\biggr],\qquad0\le i\le N-1\quad\mbox{and}
\nonumber\\[-8pt]\\[-8pt]\nonumber
\bar Z_\tN&=&Z_T.
\end{eqnarray}

The RV $Z_T$ can be obtained using
(\ref{eqrepZNo1}), namely $Z_T=(\nabla_x g)(X_T)\sigma(T,X_T)$
when $g\in C^1$. If $g$ is only Lipschitz continuous then one easily sees
that
a RV $G\in L^\infty(\cF_T)$ exists such that $Z_T=G \sigma(T,X_T)$.
In any
case, under \textup{(HX0)} and \textup{(HY0)} it easily follows that
%
\begin{eqnarray}\label{boundforZT}
\bar{Z}_\tN&=& Z_T \in L^p(
\cF_T)\qquad \mbox{for any } p\geq2 \quad\mbox{and}
\nonumber\\[-8pt]\\[-8pt]\nonumber
\bar{Z}_\ti&\in& L^2\qquad\mbox{for any }\ti\in\pi.
\end{eqnarray}
It is not difficult to show that $\bar{Z}_\ti$ is
the best $\cF_\ti$-measurable square integrable RV approximating $Z$ in
$\cH^2([\ti,\tip])$, that is,
%
\begin{equation}
\label{eqbarZbestH2approx} \bE\biggl[ \int_\ti^\tip\llvert
Z_s-\bar{Z}_\ti\rrvert^2 \,\uds\biggr] = \inf
_{\xi\in L^2(\Omega,\cF_\ti)} \bE\biggl[\int_\ti^\tip
\llvert Z_s-\xi\rrvert^2 \,\uds\biggr].
\end{equation}
Let now $\bar{Z}_t:= \bar{Z}_\ti$ for $t\in[\ti, \tip)$, $0\le
i\le N-1$.
It is equally easy to see that $\bar Z$ converges to $Z$ in $\cH^2$ as
$\llvert\pi\rrvert $ vanishes: since $Z$ is adapted, the
family of
processes $Z^\pi$ indexed by our partition defined by $Z^\pi_t=Z_\ti
$ for
$t\in
[\ti,\tip)$ converges to
$Z$ in $\cH^2$ as $\llvert\pi\rrvert $ goes to zero. Since
$\{\bar Z\}$ is the best
$\cH^2$-approximation of $Z$, we obtain
\[
\llVert Z-\bar Z \rrVert_{\cH^2} \leq\bigl\llVert Z-Z^\pi
\bigr\rrVert_{\cH^2}\to0\qquad\mbox{as } \llvert\pi\rrvert\to0,
\]
although without knowing the rate of this convergence.

The next result expresses the modulus of continuity (in the time variable)
for
$Y$ and $Z$.

%
\begin{theorem}[(Path regularity)]
\label{thregularity}
Let \textup{(HX0)}, \textup{(HY0$_{\mathrm{loc}}$)} hold. Then the unique
solution $(X,Y,Z)$ to
(\ref{canonicSDE})--(\ref{canonicFBSDE}) satisfies
$(X,Y,Z)\in\cS^{p}\times\cS^p\times\cH^{p }$ for all $p\geq2$. Moreover:
\begin{longlist}[(iii)]
\item[(i)]
for any $p\geq2$ there exists a constant $C_p>0$ such that for $0\leq
s\leq t\leq T$
we have
%
\begin{equation}
\label{pathregY} \bE\Bigl[\sup_{s\leq u\leq t} \llvert
Y_u-Y_s\rrvert^p \Bigr]\leq C_p
\bigl(1+\llvert x\rrvert^p\bigr)\llvert t-s\rrvert^{p/2};
\end{equation}
\item[(ii)]
for any $p\geq2$ there exists a constant $C_p>0$ such that for any
partition $\pi$ of $[0,T]$ with mesh size $\llvert\pi\rrvert $
%
\begin{eqnarray}
\label{pathregZ}
\nonumber
\qquad&& \sum_{i=0}^{N-1}
\bE\biggl[ \biggl(\int_{t_i}^{t_{i+1}}\llvert
Z_t-Z_{t_i}\rrvert^2\,\udt
\biggr)^{p/2} + \biggl(\int_{t_i}^{t_{i+1}}
\llvert Z_t-Z_{\tip}\rrvert^2\,\udt
\biggr)^{p/2} \biggr]
\nonumber\\[-8pt]\\[-8pt]\nonumber
&&\qquad \leq C_p \bigl(1+\llvert x\rrvert^{p}
\bigr) \llvert\pi\rrvert^{p/2};
\end{eqnarray}
\item[(iii)] in particular, there exists a constant $C$ such that for any
partition $\pi= \{0=t_0<\cdots<t_N=T\}$ of the interval $[0,T]$ with mesh
size
$\llvert\pi\rrvert$ we have
\begin{eqnarray*}
\mathrm{REG}_\pi(Y)^2&:=& \max
_{0\leq i\leq N-1} \sup_{ t\in[\ti,\tip]} \bigl\{ \bE\bigl[ \llvert
Y_t-Y_\ti\rrvert^2 \bigr] + \bE\bigl[
\llvert Y_t-Y_\tip\rrvert^2 \bigr] \bigr\}
\\
&\leq& C \llvert\pi\rrvert
\end{eqnarray*}
and $
\sum_{i=0}^{N-1}
\bE[\int_\ti^\tip\llvert Z_s-\bar{Z}_\ti\rrvert
^2 \,\uds
]\leq C \llvert\pi\rrvert$. Moreover, if $r_{\pi}$
remains bounded\footnote{
This is trivially satisfied for the uniform grid for which $r_\pi=1$.}
as $\llvert\pi\rrvert \rightarrow0$ then
\begin{eqnarray*}
\mathrm{REG}_\pi(Z)^2&:=& \sum
_{i=0}^{N-1} \bE\biggl[\int_\ti^\tip
\llvert Z_s-\bar{Z}_\ti\rrvert^2 \,\uds\biggr]
\\
&&{} + \sum_{i=0}^{N-1} \bE
\biggl[\int_\ti^\tip\llvert Z_s-
\bar{Z}_\tip\rrvert^2 \,\uds\biggr]
\\
&\leq& C \llvert\pi
\rrvert.
\end{eqnarray*}
\end{longlist}
\end{theorem}

\begin{pf}
Fix $(t,x)\in[0,T]\times\bR^d$, take $s\in[t,T]$ and throughout
this proof
we
work with $\Theta^{t,x}$ and $\nabla_x \Theta^{t,x}$; to avoid a notational
overload we omit the super- and subscript and write $\Theta$ and
$\nabla
\Theta$.
Under the theorem's assumptions, $(X,Y,Z)\in\cS^{p}\times\cS
^p\times\cH^{p }$ for all $p\geq2$ and (\ref{ZSpestimate}) holds.
We first prove points (i)~and~(ii) under assumption \textup{\textup{(HXY1)}}, then we use the
same mollification argument as in the proof of (\ref{ZSpestimate}) to recover
the case \textup{(HX0)}--\textup{(HY0$_{\mathrm{loc}}$)}. We then explain how (iii) is
obtained.

\begin{longlist}
\item[\emph{Proof of} (i) \emph{under} \textup{(HXY1)}.] From Theorem~\ref{theo-ZinSp}
follows $Z\in\cS^q$ for any $q\geq2$.
Writing the BSDE for the difference $Y_u-Y_s$ for $0\leq s\leq u\leq
T$, we have
\begin{eqnarray*}
Y_u-Y_s &=&\int_s^u
f(r,\Theta_r)\,\udr-\int_s^u
Z_r\,\udwr
\\
&\leq&\int_s^u K \bigl(1+\llvert
X_r \rrvert+\llvert Y_r\rrvert^m+
\llvert Z_r\rrvert\bigr)\,\udr-\int_s^u
Z_r\,\udwr.
\end{eqnarray*}
Taking absolute values, the $\sup$ over $u\in[s,t]\subseteq[0,T]$, power
$p$, expectations and Jensen's inequality leads, for some constant
$C_p>0$, to
\begin{eqnarray*}
&& \bE\Bigl[\sup_{u\in[s,t]}\llvert Y_u-Y_s
\rrvert^p \Bigr]
\\
&&\qquad \leq C_p \biggl\{\llvert t-s\rrvert^p \bigl( 1+\bigl
\llVert(X,Y,Z)\bigr\rrVert_{\cS^p\times\cS^p \times\cS^p}^p \bigr)
+\bE\biggl[
\sup_{u\in[s,t]}\biggl\llvert\int_s^u
Z_r\,\udwr\biggr\rrvert^p \biggr] \biggr\}.
\end{eqnarray*}
Applying BDG to the last term in the
RHS, then (\ref{ZSpestimate}) yields
\begin{eqnarray*}
&& \bE\biggl[ \sup_{u\in[s,t]}\biggl\llvert\int
_s^u Z_r\,\udwr\biggr\rrvert
^p \biggr]
\\
&&\qquad \leq C_p \bE\biggl[ \biggl(\int
_s^t \llvert Z_r\rrvert
^2\,\ud r \biggr)^{p/2} \biggr]
\\
&&\qquad \leq C_p \bE\biggl[ \biggl(\int_s^t
\llvert1+X_r\rrvert^2\,\ud r \biggr)^{p/2}
\biggr] \leq C_p \llvert t-s\rrvert^{p/2} \llVert X
\rrVert_{\cS^p}^p.
\end{eqnarray*}
It then follows that
\[
\bE\Bigl[\sup_{u\in[s,t]}\llvert Y_u-Y_s
\rrvert^p \Bigr] \leq C_p \bigl\{\llvert t-s\rrvert
^p+\llvert t-s\rrvert^{p/2} \bigr\} \leq
C_p\bigl(1+\llvert x \rrvert^{p}\bigr)\llvert t-s\rrvert
^{p/2}.
\]
\end{longlist}

\begin{longlist}
\item[\emph{Proof of} (ii) \emph{under} \textup{(HXY1)}.]
To prove the desired inequality, we use
the representation (\ref{z-rep}) [alternatively (\ref{eqrepZNo2})].
We first
estimate the difference $\bE[ (\int_\ti^\tip
\llvert Z_s-Z_\ti\rrvert ^2\,\uds)^{p/2}]$. The
difference $Z_s-Z_\ti$ can be written
as $Z_s-Z_\ti=I_1+I_2$ with $I_2:= (\nabla Y_s -\nabla
Y_\ti)
(\nabla X_\ti)^{-1}\sigma(\ti,X_\ti)$ and
\begin{eqnarray*}
I_1&:=& \nabla Y_s \bigl\{ \bigl( (\nabla X_s
)^{-1} - (\nabla X_\ti)^{-1} \bigr)
\sigma(s,X_s) + (\nabla X_\ti)^{-1} \bigl[
\sigma(s,X_s)-\sigma(\ti,X_\ti) \bigr] \bigr\}.
\end{eqnarray*}
The estimation of $I_1$ is rather easy as it relies on H\"older's inequality
combined with (\ref{estimatenablaynablaz}), \textup{(HX0)}, Theorems 2.3 and
2.4 in \citet{ImkellerDosReis2010} [see proof of Theorem 5.5(i) in
\citet{ImkellerDosReis2010}], in short we have
\[
\bE\bigl[\llvert I_1\rrvert^p\bigr]\leq C_p
\bigl(1+\llvert x \rrvert^p\bigr)\llvert\pi\rrvert^{p/2}.
\]
Concerning the second part, the estimation of $I_2$, it follows from an
adaptation of the proof of Theorem 5.5(ii) in \citet{ImkellerDosReis2010a}.
We
reformulate the main argument and skip the obvious details. Let us start
with
a simple trick, as $s\in[\ti,\tip]$,
%
\begin{eqnarray}\label{eqauxtrickforpathreg}
\nonumber
&& \bE\bigl[ \bigl\llvert(\nabla Y_s-\nabla
Y_\ti) (\nabla X_\ti)^{-1} \sigma(
\ti,X_\ti)\bigr\rrvert^{p} \bigr]
\nonumber\\[-8pt]\\[-8pt]\nonumber
&&\qquad= \bE\bigl[ \bE
\bigl[ \llvert\nabla
Y_s-\nabla Y_\ti\rrvert^{p}\mid
\cF_\ti\bigr] \bigl\llvert(\nabla X_\ti)^{-1}
\sigma(\ti,X_\ti)\bigr\rrvert^{p} \bigr].
\end{eqnarray}
Writing the BSDE for the difference $\nabla Y_s-\nabla Y_\ti$ for $\ti
\leq s
\leq\tip$, we have for some constant $C>0$
\[
\bE\bigl[ \llvert\nabla Y_s-\nabla Y_\ti\rrvert
^{p}\mid\cF_\ti\bigr] \leq C \bE[
\widehat{I}_{[\ti,\tip]}\mid\cF_\ti],
\]
where
\[
\widehat{I}_{[\ti,\tip]} := \biggl(\int_\ti^\tip
\bigl\llvert(\nabla f) (r,\Theta_r)\bigr\rrvert\llvert\nabla
\Theta_r\rrvert\,\ud r \biggr)^{p} + \biggl(\int
_\ti^\tip\llvert\nabla Z_r\rrvert
^2\,\ud r \biggr)^{p/2},
\]
where we used the conditional BDG inequality and maximized
over the time interval $[\ti,\tip]$.

Combining these last two inequalities and observing that since $\nabla
X_\ti$
and $\sigma(X_\ti)$ are $\cF_\ti$-adapted, we can drop the conditional
expectation from (\ref{eqauxtrickforpathreg}). Hence, for some $C>0$,
\begin{eqnarray*}
&& \sum_{i=0}^{N-1}  \bE\biggl[ \biggl(\int
_\ti^\tip\llvert I_2\rrvert
^2\,\uds\biggr)^{p/2} \biggr]
\\
&&\qquad \leq C \llvert\pi\rrvert^{p/2-1} \sum_{i=0}^{N-1}
\int_\ti^\tip\bE\bigl[ \llvert I_2
\rrvert^{p} \bigr]\,\uds
\\
&&\qquad \leq C \llvert\pi\rrvert^{p/2-1} \sum_{i=0}^{N-1}
\llvert\pi\rrvert\bE\bigl[ \bigl\llvert(\nabla X_\ti)^{-1}
\sigma(\ti,X_\ti)\bigr\rrvert^{p} \widehat{I}_{[\ti,\tip]}
\bigr]
\\
&&\qquad \leq C \llvert\pi\rrvert^{p/2} \bE\Biggl[ \sup_{0\leq t\leq T}
\bigl\llvert(\nabla X_t)^{-1}\sigma(t,X_t)
\bigr\rrvert^{p} \sum_{i=0}^{N-1}
\widehat{I}_{[\ti,\tip]} \Biggr]
\\
&&\qquad \leq C \llvert\pi\rrvert^{p/2} \bigl\llVert(\nabla
X)^{-1} \bigr\rrVert_{\cS^{3p}}^{1/3} \llVert1+X
\rrVert_{\cS^{3p}}^{1/3} \llVert\widehat{I}_{[0,T]}
\rrVert_{L^1}
\\
&&\qquad \leq C \bigl(1+\llvert x \rrvert^p\bigr)
\llvert\pi\rrvert^{p/2}.
\end{eqnarray*}
The last line follows from standard inequalities (sum of powers is less than
the power of the sum), the growth conditions on
$\nabla f$ and the fact that for any $q\geq2$ we
have: $X,\nabla X,(\nabla X)^{-1}\in\cS^{q}$, $Y,\nabla Y \in\cS
^{q}$, (\ref{ZSpestimate}) and $\nabla Z \in\cH^{q}$.

Collecting now the estimates, we obtain the desired result for the difference
$Z_s-Z_\ti$. To have the same estimate for the difference $Z_s-Z_\tip
$ we need
only to repeat the above calculations with a minor change in order to
incorporate the $Z_\tip$: one writes $Z_s-Z_\tip$ with the help of
$I_1^\ip$
and $I_2^\ip$, which are $I_1$ and $I_2$, respectively, but with $\tip$ instead
of $\ti$. The estimate for $I_1^\ip$ follows from SDE theory in the same
fashion as for $I_1$ above; concerning $I_2^\ip$ one just needs
another small
trick,
%
\begin{eqnarray}
\nonumber
I_2^\ip&=& (\nabla Y_s -\nabla
Y_\tip) (\nabla X_\tip)^{-1}\sigma(
\tip,X_\tip)
\\
\label{eqmodcontzipeq1} &\leq&\bigl(\llvert\nabla Y_s\rrvert+\llvert
\nabla
Y_\tip\rrvert\bigr) \bigl[ (\nabla X_\tip
)^{-1}\sigma(\tip,X_\tip) - (\nabla X_\ti
)^{-1}\sigma(\ti,X_\ti) \bigr]\hspace*{-20pt}
\\
\label{eqmodcontzipeq2} &&{} + (\nabla Y_s -\nabla Y_\tip) (
\nabla X_\ti)^{-1}\sigma(\ti,X_\ti).
\end{eqnarray}
The rest of the proof follows just like before, like $I_1$ for
(\ref{eqmodcontzipeq1}) and like $I_2$ for~(\ref{eqmodcontzipeq2}).
\end{longlist}

\emph{Final  step}---(i)  \textit{and}  (ii)
\textit{under}  \textup{(HX0)}--(HY0${}_{\mathrm{loc}}$)---\textit{arguing  via  mollification}:\quad
Here, we follow the same setup as in the proof of (\ref{ZSpestimate})
under \textup{(HX0)}--\textup{(HY0$_{\mathrm{loc}}$)} (see Theorem~\ref{theo-ZinSp}).

Take $b^n,\sigma^n,g^n,f^n$ as mollified versions of $b,\sigma,g,f$ in
their spatial variables such that the mollified functions satisfy uniformly
(in $n$) \textup{(HX0)} and \textup{(HY0$_{\mathrm{loc}}$)}, with uniform Lipschitz and
monotonicity constant. From the proof of Theorem~\ref{theo-ZinSp}, we know
that $\Theta=(X^n,Y^n,Z^n)\in\cS^p\times\cS^p\times\cH^p$ for
any $p\geq
2$ and $\Theta^n\to\Theta$ as $n\to\infty$ in
$\cS^{p}\times\cS^p\times\cH^{p }$ for all $p\geq2$.

For each $n\in\bN$ estimates (\ref{pathregY}) and (\ref{pathregZ})
hold for
$\Theta^n$. Since $b^n,\sigma^n,g^n,f^n$ satisfy \textup{(HX0)} and
\textup{(HY0$_{\mathrm{loc}}$)} uniformly in
$n$ then it is easy to check that the constants appearing on the RHS of
(\ref{pathregY}) and (\ref{pathregZ}) are independent of $n$. Hence,
by taking
the limit of $n\to\infty$ in (\ref{pathregY}) and (\ref{pathregZ})
and given
the convergence $\Theta^n\to\Theta$ as $n\to\infty$ (and the continuity
of the involved functions) the statement follows.

\begin{longlist}
\item[\emph{Proof of} (iii) \textit{under} \textup{(HX0)}--\textup{(HY0$_{\mathrm{loc}}$)}.]
The estimates concerning
$Y$ and
$\bar Z_\ti$ follow trivially from (\ref{pathregY}) on the one hand,
and (\ref{pathregZ}) combined with (\ref{eqbarZbestH2approx}) on the other
hand.
For the difference $Z_s-\bar Z_\tip$, more care is
required,
\begin{eqnarray*}
&& \sum_{i=0}^{N-1} \bE\biggl[\int
_\ti^\tip\llvert Z_s-
\bar{Z}_\tip\rrvert^2 \,\uds\biggr]
\\
&&\qquad \leq2 \sum
_{i=0}^{N-1} \bE\biggl[\int_\ti^\tip
\llvert Z_s-Z_\tip\rrvert^2 +\llvert
Z_\tip-\bar{Z}_\tip\rrvert^2 \,\uds\biggr]
\\
&&\qquad \leq C\llvert\pi\rrvert+2 \sum_{i=0}^{N-1}
(\tip-\ti)\bE\bigl[\llvert{Z}_\tip-\bar{Z}_\tip\rrvert
^2 \bigr],
\end{eqnarray*}
where the last inequality follows from the proof of (ii).
We next estimate the last term in the RHS, since $\bar Z_\tN=Z_T$ by
construction
\begin{eqnarray*}
&& \sum_{i=0}^{N-1} (\tip-\ti)\bE\bigl[
\llvert{Z}_\tip-\bar{Z}_\tip\rrvert^2 \bigr]
\\
&&\qquad = \sum_{i=0}^{N-2} (\tip-\ti)\bE
\bigl[\llvert{Z}_\tip-\bar{Z}_\tip\rrvert^2
\bigr]
\\
&&\qquad \leq r_\pi\sum_{i=0}^{N-2}
(t_{i+2}-t_\ip)\bE\bigl[\llvert{Z}_\tip-
\bar{Z}_\tip\rrvert^2 \bigr]
\\
&&\qquad \leq r_\pi\sum_{i=0}^{N-2}
\int_\tip^{t_{i+2}}\bE\bigl[\llvert
{Z}_\tip-\bar{Z}_\tip\rrvert^2 \bigr]\,\uds
\\
&&\qquad \leq r_\pi\sum_{j=1}^{N-1}
\int_{t_j}^{t_{j+1}}\bE\bigl[\llvert
{Z}_{t_j}-\bar{Z}_{t_j}\rrvert^2 \bigr]\,\uds
\\
&&\qquad \leq2 r_\pi\sum_{i=0}^{N-1}
\bE\biggl[\int_{\ti}^{\tip} \llvert
Z_s-{Z}_{\ti}\rrvert^2 +\llvert
Z_s-\bar{Z}_{\ti}\rrvert^2 \,\uds\biggr],
\end{eqnarray*}
where we made use of the assumption on the grid. The result now follows by
combining (iii) with the above estimates and having in mind that $r_\pi
$ is
uniform over the partition.\quad\qed
\end{longlist}\noqed
\end{pf}

%
\begin{corollary}
\label{barZinLp}
Let \textup{(HX0)}, \textup{(HY0)} hold and take the family
$\{\bar Z_\ti\}_{\ti\in\pi}$. For any $p\geq1$ there exists
constant $C_p$
independent of $\llvert\pi\rrvert$ such that
\[
\bE\Biggl[ \sum_{i=0}^{N-1} \bigl( \llvert
\bar{Z}_\ti\rrvert^2 (\tip-\ti) \bigr)^p
\Biggr] \le C_p<\infty.
\]
If, moreover, \textup{(HY0$_{\mathrm{loc}}$)} holds then $\max_{\ti\in\pi
} \bE[ \llvert \bar
Z_\ti\rrvert ^{2p} ] \leq C_p <\infty$.
\end{corollary}

\begin{pf}
The second statement follows easily from the definition of
$\bar Z_\ti$ [see~(\ref{Z-bar-ti-pi})] and the fact that
estimate (\ref{ZSpestimate}) holds under \textup{(HY0$_{\mathrm{loc}}$)}.
Moreover, under this assumption the second estimate implies the first.

We leave the proof of the first statement for the interested reader. The
proof is based on standard integral manipulations combining the
definition of
$\bar Z$, Jensen's inequality, the fact that $Z\in\cH^p$ and the
tower property of the conditional expectation
[see Section~4.7.5 in \citet{Lionnet2014}].
\end{pf}


\subsection{Some finer properties}

Here, we discuss properties of the solution to
(\ref{canonicSDE})--(\ref{canonicFBSDE}) in more specific settings.
The first lemma concerns a set-up where $Z$ belongs to $\cS^\infty$ (rather
than $\cH^2$ or $\cS^2$).

\begin{proposition}[(The additive noise case)]
Let \textup{(HX0)}--\textup{(HY0$_{\mathrm{loc}}$)} hold. Assume additionally that
$\sigma(t,x)=\sigma(t)$
for all $(t,x)\in[0,T]\times\bR^d$. Then $Z\in\cS^\infty$.
\end{proposition}

\begin{pf}
Assume first that \textup{\textup{(HXY1)}} also hold. Then the result follows easily
by combining the representation formula (\ref{eqrepZNo1}) with the
2nd part
of (\ref{repynablaunablax}) and injecting that $\sigma$ is uniformly bounded.

Now using a standard mollification argument, as was used in the last
step of
the proof of Theorem~\ref{thregularity}, one easily concludes that
the result
also holds under \textup{(HX0)}--\textup{(HY0$_{\mathrm{loc}}$)}.
\end{pf}
If the initial data $g$ and $f(\cdot,\cdot,0,0)$ are bounded, then so
will be
the $Y$ process; the second component, $Z$ will also satisfy a type of
boundedness condition [see~(\ref{eqZinBMO}) below].

\begin{lemma}[(The bounded setting)]
\label{lemmaboundedsetting}
Let \textup{(HX0)}, \textup{(HY0)} hold and further that $g$ and $(t,x)\mapsto f(t,x,0,0)$
are uniformly bounded then $(Y,Z)\in\cS^\infty\times\cH^2$.

Denoting $\cT_{[0,T]}$ the set of all stopping times $\tau\in[0,T]$,
then $Z$
satisfies further\footnote{This means $Z$ belongs to the so-called
$\cH_{\BMO}$-spaces, see Section~2.3 in \citet{ImkellerDosReis2010} or
Section~10.1 in \citet{Touzi2012}.} for some constant $K_{\BMO}>0$
%
\begin{equation}
\label{eqZinBMO} \sup_{\tau\in\cT_{[0,T]}} \biggl\llVert\bE\biggl[
\int
_\tau^T \llvert Z_s\rrvert
^2\,\uds\Big| \cF_\tau\biggr] \biggr\rrVert_\infty
\leq K_{\BMO}<\infty.
\end{equation}
The constant $K_{\BMO}$ depends only on $\llVert Y\rrVert_{\cS
^\infty}$, the
bounds for
$g$, $f(\cdot,\cdot,0,0)$ and the constants appearing in \textup{(HY0)}.
\end{lemma}

\begin{pf}
The boundedness of $Y$ follows from (\ref{eqpathwise-estim-for-polyy-BSDE})
by using that $g(X_\cdot)$ and $f(\cdot,X_\cdot,0,0)$ are in $\cS
^\infty$. Knowing that
$Y\in\cS^\infty$ we can easily adapt the
proof of Lemma 10.2 in \citet{Touzi2012} to our setting, where we make
use of
the inequality $\llvert z\rrvert \leq1+\llvert z\rrvert ^2$, to
obtain (\ref{eqZinBMO}); an alternative
proof would be to use (\ref{eqpathwise-estim-for-polyy-BSDE}).
\end{pf}
The first of the above results implies that $Z$ is bounded. Such a setting
also includes the case of $\sigma(t,x)=1$ which is
common in many applications in reaction--diffusion equations. The next result
provides another type of control for the growth of the process
$Z$ without the boundedness assumption on $\sigma$.

\begin{proposition}
Let the assumptions of Lemma~\ref{lemmaboundedsetting} hold. Assume
further that $\llvert Z\rrvert ^2$ is a submartingale then
$\llvert Z_t\rrvert \leq{K_{\BMO}}/{\sqrt{T-t}}$, $\forall
t\in[0,T]$ $\bP$-a.s.

In particular, if $\sigma$ is uniformly elliptic and \textup{\textup{(HXY1)}} holds then
there exists $C>0$ such that
$\llvert \nabla_x u(t,x)\rrvert \leq{C}/{\sqrt{T-t}}$,
$\forall(t,x)\in
[0,T)\times
\bR^n$.
\end{proposition}

\begin{pf}
The first statement follows by a careful but rather clean analysis of the
fact that $Z$ satisfies (\ref{eqZinBMO}), which in particular means any
$t\in[0,T]$ $\bP$-a.s.
\begin{eqnarray*}
K_{\BMO} &\geq& \bE\biggl[ \int_t^T
\llvert Z_s\rrvert^2\,\uds\Big| \cF_t \biggr] = \int_t^T \bE\bigl[ \llvert Z_s
\rrvert^2 \mid\cF_t \bigr]\,\uds
\\
& \geq&\int_t^T \llvert Z_t
\rrvert^2\,\uds= \llvert Z_t\rrvert^2 (T-t),
\end{eqnarray*}
where we applied Fubini then used the submartingale property of $Z^2$. The
sought statement now follows by a direct rewriting of the above inequality.
The
second statement in the proposition follows from the first by using the
representation
$Z^{t,x}_t=(\nabla_x u \sigma)(t,x)$ and the ellipticity of $\sigma$.
\end{pf}

\section{Numerical discretization and general estimates}\label{sectionnumericaldiscretization}

In this section and the following ones, we discuss the numerical approximation
of (\ref{canonicSDE})--(\ref{canonicFBSDE}).
We consider a regular partition\footnote{We point out that the results we
state would hold for nonuniform time-steps, but we work with a regular
partition for notational clarity and to keep the focus on the main issues.}
$\pi$ of $[0,T]$ with $N+1$ points $t_i = i h$ for $i=0,\ldots,N$
with $h:=T/N$.

%
\begin{remark}[(On constants)]
Throughout the rest of this work, we introduce a generic constant
$c>0$, that
will always be independent of $h$ or $N$, though it may
depend on the problem's data, namely the constants appearing in the
assumptions, and may change from line to line.
\end{remark}

\subsection{Discretization of the SDE and further setup}\label{subsectiondiscretizationforwardcomponent}

Numerical methods for SDEs with Lipschitz continuous coefficients are well
understood; see Section~10 in \citet{KloedenPlaten1992}.
Therefore, we take as given a family
of random variables $\{X_i\}_{i=0,\ldots, N}$ that approximates
the solution $X$ to (\ref{canonicSDE}) over the grid $\pi$. More
exactly, for
any
$p\ge2$ there exists a constant
$c=c(T,p,x)$ such that
%
\begin{eqnarray}
\label{eqSDEdiscretizationL2}  \sup_{N \in\bN} \max_{i=0,\ldots,N}\bE
\bigl[ \llvert X_i\rrvert^p \bigr] &\le& c
\end{eqnarray}
and
\begin{eqnarray}
\mathrm{ERR}_{\pi,p}(X)&:=& \max
_{i=0,\ldots,N}\bE\bigl[ \llvert X_{\ti
} - X_i
\rrvert^{p} \bigr]^{1/p} \le c h^{\gamma},\qquad\gamma
\geq\frac{1}{2},
\end{eqnarray}
where $\gamma$ is called the rate of the strong convergence and the random
variables
$\{X_\ti\}_{\ti\in\pi}$ are the solution to (\ref{canonicSDE}) on
the grid points $\pi$. Under \textup{(HX0)}, the Euler scheme give an approximation
with
$\gamma=
1/2$. For conditions required for the higher order schemes, we refer to
\citet{KloedenPlaten1992}.
Since the upper bound in the estimate on the error on $X$ does not
depend on $p$,
and since we use only the case $p=2$ in the following, we simplify the
notation to
$\mathrm{ERR}_{\pi}(X) \le c h^\gamma$.

Throughout the rest of this work, we assume that the family
$\{X_i\}_{i=0,\ldots, N}$ has been computed; we denote by
$\{\cF_i\}_{i=0,\ldots, N}$ the associated discrete-time filtration
$\cF_i:=\sigma(X_j,j=0,\ldots,i)$ and with respect to this
filtration we
define the operator $\bE_i[\cdot]:=\bE[\cdot\mid\cF_i]$.

For the analysis of the time-discretization error, we also make use of the
following standard path-regularity estimate for $X$,
which holds under \textup{(HX0)}: there exists a constant $c>0$ such that
%
\begin{eqnarray}
\label{eqSDEpathregularity}
\mathrm{REG}_\pi(X)&:=& \max
_{i=0,\ldots,N-1}  \sup_{\ti\le s \le\tip} \bigl\{ \bE\bigl[
\llvert X_s-X_{\ti}\rrvert^2
\bigr]^{1/2}
+ \bE\bigl[ \llvert X_s-X_{\tip}\rrvert
^2 \bigr]^{1/2} \bigr\} \hspace*{-15pt}
\nonumber\\[-8pt]\\[-8pt]\nonumber
&\le& c h^{1/2}.
\end{eqnarray}

\subsection{Schemes considered and main convergence results}\label{subsectionmainresultsonconvergenceandorganisation}

For the reader's convenience, we state immediately the numerical
schemes under
consideration as well as their convergence rates. The rest of this work deals
with the proofs of the stated results.

Theorem~\ref{thregularity} implies that to approximate
$(Y,Z)$ solution to (\ref{canonicFBSDE}) over $[0,T]$ one needs only to
approximate the family
$\{(Y_\ti,\bar{Z}_\ti)\}_{\ti\in\pi}$ [recall (\ref{Z-bar-ti-pi})]
on the grid $\pi$ via a family of random variables
$\{(Y_i,Z_i)\}_{i=0,\ldots,N}$, the
said numerical approximation. The error criterion we consider is given by
%
\begin{eqnarray}
\label{eqerrornorm-sec4}  \quad \mathrm{ERR}_\pi(Y,Z)&:=& \Biggl( \max
_{i=0, \ldots, N} \bE\bigl[ \llvert Y_\ti- Y_i
\rrvert^2 \bigr] + \sum_{i=0}^{N-1}
\bE\bigl[\llvert\bar{Z}_\ti- Z_i \rrvert
^2 \bigr] h \Biggr)^{1/2}.
\end{eqnarray}

\subsubsection{The implicit-dominant \texorpdfstring{$\theta$}{theta}-schemes of Section~\texorpdfstring{\protect\ref{sectionthetaimplicitschemes}}{5}}

Let $\theta\in[0,1]$. Define $Y_N:=g(X_N)$ and $Z_N:=0$ and, for
$i=N-1,N-2, \ldots, 0$,
%
\begin{eqnarray}
Y_i &:=& \bE_i \bigl[ Y_{i+1}
+ (1-\theta)f(\tip,X_\ip,Y_{i+1},Z_{i+1})h \bigr]
\nonumber\\[-8pt]\label{eqY} \\[-8pt]\nonumber
&&{} +\theta f(\ti,X_i,Y_{i}, Z_{i})h,
\\
\label{eqZ} Z_i &:=& \bE_{{i}} \biggl[
\frac{\Delta W_{{i+1}} }{h } \bigl( Y_{{i+1}} + (1-\theta)f(\tip,X_\ip,Y_{i+1},Z_{i+1})h
\bigr) \biggr],
\end{eqnarray}
where $\Delta W_\ip= W_\tip- W_i$.
The above scheme is the called $\theta$-scheme.
Its derivation is presented in Section~\ref{subsectiondiscretizationbackwardcomponent} and the solvability (in
$Y_i$) of (\ref{eqY}) for $\theta> 0$ is discussed in Section~\ref
{subsectionwelldefinednessandlocalestimate}.
When $\theta= 1$ this is the implicit backward Euler scheme,
when $\theta= 0$ this is the explicit scheme.
For $\theta\in\,]0,1[$ it is a combination of both.
The particular case of $\theta=1/2$ is the trapezoidal scheme which,
we will
show,
has a better convergence rate (under certain conditions).
The convergence rate of the above scheme is summarized in the next result.

%
\begin{theorem}
\label{thmainresult}
Let \textup{(HX0)}, \textup{(HY0$_{\mathrm{loc}}$)} hold as well as the restriction
$h\le\min\{1,[4\theta(L_y+3d\theta
L_z^2 )]^{-1}\}$.
Let $\gamma\geq1/2$ be the order of the approximation
$\{X_i\}_{i = 0, \ldots, N}$ of $X$ as in (\ref{eqSDEdiscretizationL2}).
Then, for the scheme (\ref{eqY})--(\ref{eqZ}) we have:
\begin{longlist}[(ii)]
\item[(i)] For $\theta\in[1/2,1]$, there exists a constant $c$ such that
$\mathrm{ERR}_\pi(Y,Z) \le c h^{1/2}$.
\item[(ii)] Take $\theta=1/2$ and scheme (\ref{eqY}). Assume
that $f\in C^2$, $f(t,x,y,z)=f(y)$ and $\partial^2_{yy}f$ has at
most polynomial growth, then
there exists $c>0$ such that
$\max_{i=0, \ldots, N} \bE[\llvert Y_\ti- Y_i \rrvert
^2]^{1/2} \le c h^{ \min\{ 7/4,\gamma\} }$.
\end{longlist}
\end{theorem}

Reasons why the above theorem only holds for $\theta\ge1/2$---that is to say when the scheme is ``more implicit than explicit''---will be seen later in the proofs in Section~\ref
{sectionthetaimplicitschemes}.
But from the motivating example of the \hyperref[sec1]{Introduction}, we know already
that one
could not have expected convergence of the scheme in general, for all
$\theta\in[0,1]$.

\subsubsection{The tamed explicit scheme of Section~\texorpdfstring{\protect\ref{sectionexplicitscheme}}{6}}\label{subsec-tamed-euler}
By inspecting the proof of Lemma~\ref{lemma-counterexample}, we see that
the unboundedness of $g(X_T)$ plays the key role in the explosion. In
Section~\ref{sectionexplicitscheme}, we analyze a tamed version of the
fully explicit ($\theta=0$) scheme (\ref{eqY})--(\ref{eqZ}).

For any level $L>0$, we define the truncation function $T_L\dvtx \bR\to
\bR$, $x \mapsto-L \lor x \land L $. We denote similarly its
extension as a function from $\bR^d$ to $\bR^d$ (projection on the
ball of
radius $L$).
We consider the following scheme: define $Y_N:=T_{L_h} (g(X_N) )$,
$Z_N:=0$,
and for $i=N-1, \ldots, 0$,
%
\begin{eqnarray}
\label{eqY-tamedexplicit} Y_i &:=& \bE_i \bigl[ Y_{i+1}
+ f \bigl(\tip,T_{K_h}(X_\ip),Y_{i+1},Z_{i+1}
\bigr)h \bigr],
\\
\label{eqZ-tamedexplicit} Z_i &:=& \bE_{{i}} \biggl[
\frac{\Delta W_{{i+1}} }{h } \bigl(Y_{{i+1}} + f \bigl(\tip,T_{K_h}(X_\ip),Y_{i+1},Z_{i+1}
\bigr)h \bigr) \biggr],
\end{eqnarray}
where the levels $L_h$ and $K_h$ satisfy
$e^{c_1 T} ( L_h^2 + c_2 T + c_2 T K_h^2 ) \le
{h}^{-1/({m-1})}$, with
\begin{eqnarray*}
&&c_1 = 2 \bigl(L_{y}+ 12 d L_z^2
+ 2L_y^2 \bigr) \quad\mbox{and}\quad c_2 =
\max\biggl\{ \frac{L^2}{4 d L_z^2 }, \frac{L_x^2}{4 dL_z^2 } \biggr\}.
\end{eqnarray*}
For $h\le h^*$, where $h^*$ satisfies $e^{c_1 T} c_2 T \le
({h^*})^{-1/({m-1})}/3$ and $h^*\leq1/(32dL_z^2)$ we can take
\begin{eqnarray*}
L_h &=& \frac{1}{\sqrt{3}} e^{-(1/2)c_1 T} \biggl(
\frac{1}{h} \biggr)^{1/(2(m-1))} \quad\mbox{and}\quad K_h =
\frac{1}{\sqrt{3}} \frac{e^{-(1/2)c_1 T}}{\sqrt{c_2 T}} \biggl
(\frac{1}{h}
\biggr)^{1/(2(m-1))}.
\end{eqnarray*}

Concerning the scheme (\ref{eqY-tamedexplicit})--(\ref
{eqZ-tamedexplicit}), we
have the following convergence rate.

\begin{theorem}
\label{thconvergencerateforexplicitscheme}
Let \textup{(HX0)}, \textup{(HY0$_{\mathrm{loc}}$)} hold and $h\leq h^*$. Assume that the order
$\gamma$ of the approximation $\{X_i\}_{i = 0, \ldots, N}$ of $X$ is at
least $1/2$ [see (\ref{eqSDEdiscretizationL2})].
Then for the controlled explicit scheme
(\ref{eqY-tamedexplicit})--(\ref{eqZ-tamedexplicit}), there exists a
constant c such that $\mathrm{ERR}_\pi(Y,Z) \le c h^{1/2}$.
\end{theorem}

\subsubsection{Modus operandi for the proofs and organization of rest of the paper}

The proof of the above results is a (long)
two-step procedure.
The first step is contained in the rest of this section since it is a general
argument common to most discretization schemes.
The second one is scheme-specific, hence the separation into
Sections~\ref{sectionthetaimplicitschemes} and~\ref{sectionexplicitscheme}.
We now describe the said procedure.

Before one is able to state a global error estimate for
(\ref{eqerrornorm-sec4}), one needs to find the local error
estimates, that is,
the distance between the solution
and its approximation over one time interval $[\ti,\tip]$.
This local error has two components.
The first is the \emph{one-step discretization error} following from
approximating the involved
integrals over $[\ti,\tip]$ by some quadrature rule.
The second is the backward propagation of the error due to
not having at time $\tip$ the true solution to compute the
approximation at time $\ti$ and we coin it \emph{stability error}.

In the next subsection, we give the
\emph{Fundamental Lemma for
convergence} (Lemma~\ref{lemfundamental}) that explains how to
aggregate the
one-step discretization error and the stability error for each $[\ti,\tip]$
into a single estimate with (\ref{eqerrornorm-sec4}) on its LHS. This
later allows us to derive the convergence rates.

The estimation of the one-step discretization error is common to both schemes.
This is done in Section~\ref{subsectionlocaldiscretizationerror} and
the general result is stated in Proposition~\ref{propglobaldiscretizationerror}.
Left to Sections~\ref{sectionthetaimplicitschemes} and
\ref{sectionexplicitscheme} is the scheme-specific stability analysis
[i.e., the estimation of $\cR^{\cS}(H)$ in
(\ref{eqauxfullestimateerrorYYZZ}) below].
Sections~\ref{sectionthetaimplicitschemes} and
\ref{sectionexplicitscheme} follow the same structure:
(1) one first shows some uniform global integrability for the scheme;
(2) then one studies the local (one-step) stability of the scheme; this shows
how the error propagates in just one backward step, and yields an expression
for the terms $H_j$ composing the stability remainder (see Definition
\ref{defstab} below);
(3) one finally estimates the stability remainder $\cR^{\cS}(H)$.
Once this is done, one can inject the results into estimate
(\ref{eqauxfullestimateerrorYYZZ})
given by the Fundamental Lemma~\ref{lemfundamental}; and finally estimate
the RHS of (\ref{eqauxfullestimateerrorYYZZ}) as a function of the time-step
$h$, hence obtaining the convergence rate.

At the end of Section~\ref{sectionthetaimplicitschemes}, we discuss the
fully second-order discretization scheme when $f$ is allowed to depend
only on
$y$ and we discuss as well a variance reduction trick for the
computation of
the
involved conditional expectations.

\subsection{Fundamental Lemma for convergence}\label{subsectionfundamentallemma}

The goal of this section is to present a very general but clear result
estimating the global error (\ref{eqerrornorm-sec4}) of a scheme
for BSDE (\ref{canonicFBSDE}).
Although this type of analysis has already been used in the context of
Lipschitz BSDEs
[see, e.g., \citet{CrisanManolarakis2012}, \citeauthor{Chassagneux2012}  (\citeyear{Chassagneux2012,Chassagneux2013})], we generalize it to
the non-Lipschitz framework we are working with. More precisely, the
Fundamental Lemma we present below allows us to cope with schemes which
lack stability in the sense of \citet{Chassagneux2013}.\hskip.2pt\footnote{See Definition
2.1 in \citet{Chassagneux2013} with
$\zeta_i^Y=\zeta_i^{Z}=0$ for $i=0,\ldots,N-1$.}

\subsubsection{Abstract formulation of a scheme and description of the local error}\label{subsubsectionabstractschemeanddecompositionoferror}

In abstract terms, a discretization scheme for a BSDE generates recursively
(and backward in time) a family of random variables
$\{(Y_i,Z_i)\}_{i=0,\ldots,N}$
approximating $\{(Y_\ti,\bar{Z}_\ti)\}_{\ti\in\pi}$ via some
operators
$\Phi_i\dvtx  L^2(\F_{\ip})\times L^{2}(\F_{\ip}) \rightarrow L^2(\F
_{i}) \times
L^{2}(\F_{i})$,
$i\in\{N-1,\ldots,0\}$.
One starts with an initial approximation $(Y_N,Z_N)$ and for
$i=N-1,\ldots,0$ computes $(Y_i,Z_i):=\Phi_i(Y_\ip,Z_\ip)$.
[Compare with
(\ref{eqY})--(\ref{eqZ}) or
(\ref{eqY-tamedexplicit})--(\ref{eqZ-tamedexplicit}).]

Since $(Y_i,Z_i)$ is obtained via $\Phi_i$
from the input $(Y_\ip,Z_\ip)$, we introduce the following
notation: for any $i=0,1,\ldots,N-1$, given a $\cF_\tip$-measurable input
$(\cY,\cZ)$, the pair $(Y_{i,(\cY,\cZ)},Z_{i,(\cY,\cZ)})$ denotes the
associated output of $\Phi_i(\cY,\cZ)$.
Writing $(Y_i,Z_i)$ without specifying
the input denotes the canonical output of $\Phi_i(Y_\ip,Z_\ip)$,
that is, we
refer to the family of RV's $\{(Y_i,Z_i)\}_{i=0, \ldots, N}$.
We introduce as well the notation $\widehat{Y}_i=Y_{i,(Y_\tip,\bar
{Z}_\tip)}$
and $\widehat{Z}_i=Z_{i,(Y_\tip,\bar{Z}_\tip)}$ as the output of
$\Phi_i(Y_\tip,\bar{Z}_\tip)$.

We decompose the local error into two parts: the one-step
time-discretization error and the propagation to time $\ti$ of
the error from time $\tip$ (the stability error).
So, given $i\in\{0,\ldots,N-1\}$, we write
\begin{eqnarray*}
Y_\ti-Y_i &=& (Y_\ti- \widehat{Y}_i
) + (\widehat{Y}_i- Y_{i} )
\\
&=&\underbrace{ ( Y_\ti- Y_{i,(Y_{t_{i+1}},\bar{Z}_{\tip})} )}_{\mathrm{one}\mbox{-}\mathrm{step\ discretization\ error}} +
\underbrace{ ( Y_{i,(Y_{t_{i+1}},\bar{Z}_{\tip})} -
Y_{i,(Y_{i+1},Z_{i+1})} )}_{\mathrm{stability\ of\ the\ scheme}},
\end{eqnarray*}
and similarly for $Z$
\begin{eqnarray*}
\bar{Z}_\ti-{Z}_i &=& (\bar{Z}_\ti-
\widehat{Z}_i ) + (\widehat{Z}_i- Z_{i} )
\\
&=& \underbrace{ ( \bar{Z}_\ti- Z_{i,(Y_{t_{i+1}},\bar{Z}_{\tip})} )
}_{\mathrm{one}\mbox{-}\mathrm{step\ discretization\ error}} + \underbrace{ (
Z_{i,(Y_{t_{i+1}},\bar{Z}_{\tip})} - Z_{i,(Y_{i+1},Z_{i+1})}
)}_{\mathrm{stability\ of\ the\ scheme}}.
\end{eqnarray*}
We now turn to the question of how to aggregate these errors in order to
estimate the global error $\mathrm{ERR}_\pi(Y,Z)$ [see
(\ref{eqerrornorm-sec4})].

\subsubsection{The Fundamental Stability Lemma}

The purpose of the Fundamental Lemma below is to formulate in a transparent
way the ingredients required to show convergence of
$\{(Y_i,Z_i)\}_{i=0,\ldots,N}$ to $\{(Y_\ti,\bar Z_\ti)\}_{\ti\in
\pi}$
in the error criterion (\ref{eqerrornorm-sec4}). To start with, we define
precisely our concept of stability, generalizing that
in \citet{Chassagneux2012} and \citet{Chassagneux2013}.

\begin{definition}[(Scheme stability)]
\label{defstab}
We say that the numerical scheme $\{(Y_i,Z_i)\}_{i=0,\ldots,N}$ is
\emph{stable} if
for some $\rho>0$ there exists a constant
$c>0$ such that
%
\begin{eqnarray}\label{eqdef-of-stability-sec4}
\qquad&& \bE\bigl[ \llvert Y_{i,(Y_{t_{i+1}},\bar{Z}_{\tip})} -
Y_{i,(Y_{i+1},Z_{i+1})} \rrvert
^{2} \bigr]\nonumber
\\
&&\quad{}  + \rho\bE\bigl[ \llvert Z_{i,(Y_{t_{i+1}},\bar{Z}_{\tip})} -
Z_{i,(Y_{i+1},Z_{i+1})} \rrvert^2 \bigr] h
\\
&&\qquad \le( 1 + c h )
\biggl( \bE\bigl[ \llvert
Y_{t_{i+1}} - Y_{i+1} \rrvert^{2} \bigr] +
\frac{\rho}{4} \bE\bigl[ \llvert\bar{Z}_{\tip} -Z_{i+1}
\rrvert^2 \bigr] h \biggr) + \bE[ H_{i} ],\nonumber
\end{eqnarray}
where $H^{}_i\in L^{1}(\F_i)$,
and moreover $\{H_i\}_{i=0,\ldots,N-1}$ satisfies
\begin{eqnarray*}
&&\cR^{\cS}(H):= \max_{i=0, \ldots, N-1} \sum
_{j=i}^{N-1} e^{c(j-i)h}
\bE[H_j] \longrightarrow0\qquad\mbox{as } h \rightarrow0.
\end{eqnarray*}
The quantity $\cR^\cS(H)$ is called the \emph{stability remainder}.
\end{definition}

%
\begin{remark}
In the case where $f$ is a globally Lipschitz
function, it can be shown for both implicit and explicit schemes
that $H_i=0$ [see \citet{CrisanManolarakis2012} or
\citet{Chassagneux2013}]. The scheme is then \emph{locally} stable.
Our definition of stability allows one to cope with schemes which
are not locally stable, as is the case when $f$ is a monotone function with
polynomial growth in~$y$,
provided we can control the term $\cR^\cS(H)$
(which we do in Section~\ref{sectionthetaimplicitschemes}). We also point
out that it is crucial that in (\ref{eqdef-of-stability-sec4}) we have
$\rho> \frac{\rho}{4}$
(compare LHS with RHS). This later allows the use of Gronwall type
inequalities (see Lemma~\ref{lemGron}).
\end{remark}

We now state the Fundamental Lemma which is the
basis of the error analysis throughout.

%
\begin{lemma}[(Fundamental Lemma)]\label{lemfundamental}
Assume that the numerical scheme $\{(Y_i,Z_i)\}_{i=0,\ldots,N}$ is
stable. Denoting the one-step discretization errors for
$i=0,\ldots,N-1$ by
%
\begin{equation}
\label{eqauxlocalestimateYYZZ} \cases{ \displaystyle\tau_i(Y):=\bE
\bigl[\llvert
Y_\ti- Y_{i,(Y_{t_{i+1}},\bar{Z}_{\tip})} \rrvert^{2} \bigr] = \bE\bigl[
\llvert Y_\ti- \widehat{Y}_i \rrvert^{2}
\bigr],
\cr
\displaystyle\tau_i(Z):= \bE\bigl[\llvert
\bar{Z}_\ti- Z_{i,(Y_{t_{i+1}},\bar{Z}_{\tip})} \rrvert^2 h\bigr] =\bE
\bigl[\llvert\bar{Z}_\ti- \widehat{Z}_i \rrvert
^2 h\bigr],}
\end{equation}
%
there exists a constant $C=C(\rho,T,c)$ such that
%
\begin{eqnarray}\label{eqauxfullestimateerrorYYZZ}
\qquad&& \bigl(\mathrm{ERR}_\pi(Y,Z) \bigr)^2\nonumber
\\
&&\qquad  \le C \Biggl\{ \bE\bigl[ \llvert
Y_\tN-
Y_N \rrvert^2\bigr] + \bE\bigl[\llvert\bar
{Z}_\tN- Z_N \rrvert^2\bigr] h
+ \sum_{i=0}^{N-1} \biggl(
\frac{\tau_i(Y)}{h} + \tau_i(Z) \biggr) \Biggr\}
\\
&&\quad\qquad{}+ (1+h)\cR^{\cS}(H).\nonumber
\end{eqnarray}
\end{lemma}

This result states in a rather clear fashion [although $\cR^{\cS}(H)$
is unknown at this point] what is required in order to have convergence of
the numerical scheme.
First, one needs a control on the approximation of the
terminal conditions [the first two terms in the RHS of
(\ref{eqauxfullestimateerrorYYZZ})].
Second, one needs a control on the sum of the one-step
time-discretization errors (\ref{eqauxlocalestimateYYZZ})
[the 3rd term in the RHS of (\ref{eqauxfullestimateerrorYYZZ})].
Third, one need a control on the stability remainder $\cR^{\cS}(H)$ arising
from the scheme
stability~(\ref{eqdef-of-stability-sec4}) [last term in the RHS of
(\ref{eqauxfullestimateerrorYYZZ})]. Of course, the form of $\cR
^{\cS}(H)$
depends on the specific scheme one is handling but in general the error
$\mathrm{ERR}_\pi(Y,Z)$ of the scheme is always dominated by
(\ref{eqauxfullestimateerrorYYZZ}).

The first element will be estimated in Lemma~\ref{lemmaerroronterminalconditons}.
The second is the subject of Section~\ref
{subsectionlocaldiscretizationerror} and the estimate is given in
Proposition~\ref{propglobaldiscretizationerror}.
Finally, the study of the stability of the schemes is done in
Sections~\ref{sectionthetaimplicitschemes} and~\ref
{sectionexplicitscheme}.
The convergence rate of the scheme will then follow by estimating
further the
RHS of (\ref{eqauxfullestimateerrorYYZZ}).

\begin{pf*}{Proof of Lemma \ref{lemfundamental}}
 We use throughout the following notation:
$\widehat{Y}_i=Y_{i,(Y_\tip,\bar{Z}_\tip)}$,
$\widehat{Z}_i=Z_{i,(Y_\tip,\bar{Z}_\tip)}$, $Y_i=Y_{i,(Y_{i+1},Z_{i+1})}$\vspace*{1pt}
and $Z_i=Z_{i,(Y_{i+1},Z_{i+1})}$
introduced in Section~\ref{subsubsectionabstractschemeanddecompositionoferror}.
We decompose the error as explained
above and use Young's inequality to get
$
\llvert Y_\ti- Y_i \rrvert^2 \le(1+\frac{1}{h})
\llvert Y_\ti-\widehat{Y}_i \rrvert^2
+ (1+h) \llvert\widehat{Y}_i - Y_i \rrvert^2
$
and
$
\llvert\bar{Z}_\ti- Z_i \rrvert^2 h
\le2
\llvert\bar{Z}_\ti- \widehat{Z}_i \rrvert^2 h
+ 2
\llvert\widehat{Z}_i - Z_i \rrvert^2 h$.

Using $\rho>0$ from (\ref{eqdef-of-stability-sec4}) and the definition
(\ref{eqauxlocalestimateYYZZ}) above, it then follows that
\begin{eqnarray*}
&&  \bE\bigl[\llvert Y_\ti- Y_i \rrvert^2
\bigr] + \frac{\rho}{2} \bE\bigl[\llvert\bar{Z}_\ti-
Z_i \rrvert^2\bigr] h
\\
&&\qquad \le(1+h)\bE\bigl[ \llvert\widehat{Y}_i - Y_i
\rrvert^2\bigr] 
+ \rho\bE\bigl[\llvert
\widehat{Z}_i - Z_i \rrvert^2\bigr] h +
\biggl( \biggl(1+\frac{1}{h}\biggr) \tau_i(Y) + \rho
\tau_i(Z) \biggr).
\end{eqnarray*}
Since $\rho\le(1+h) \rho$, by the stability of the scheme [see
(\ref{eqdef-of-stability-sec4})], it follows that
%
\begin{eqnarray}
\label{eqstabeq}
&& \bE\bigl[\llvert Y_\ti- Y_i \rrvert
^2\bigr] + \frac{\rho}{2} \bE\bigl[\llvert\bar{Z}_\ti-
Z_i \rrvert^2\bigr] h\nonumber
\\
&&\qquad  \le(1+h) (1+c h) \biggl( \bE\bigl[\llvert Y_{\tip} - Y_{\ip}
\rrvert^2\bigr] + \frac{\rho
}{4} \bE\bigl[\llvert
\bar{Z}_{\tip}
- Z_{\ip}\rrvert^2\bigr] h
\biggr)
\\
\nonumber
&&\quad\qquad {} + \biggl( \biggl(1+\frac{1}{h}\biggr)
\tau_i(Y) + \rho\tau_i(Z) + (1+h) \bE[H_i]
\biggr).
\end{eqnarray}
Taking $I_{i}:= \llvert Y_\ti- Y_i \rrvert^2
+ \frac{\rho}{4}\llvert\bar{Z}_\ti- Z_i \rrvert^2 h$,
we have
\begin{eqnarray*}
&&\bE[ I_i ]+ \frac{\rho}{4} \bE\bigl[\llvert
\bar{Z}_\ti- Z_i \rrvert^2\bigr] h
\\
&&\qquad \le(1+h)
(1+c h) \bE[I_{i+1}]
+ \biggl( \biggl(1+\frac{1}{h}\biggr)\tau_i(Y)
+ \rho\tau_i(Z) + (1+h) \bE[H_i] \biggr),
\end{eqnarray*}
and we complete the proof using Lemma~\ref{lemGron}.
\end{pf*}

\subsection{Discretization of the BSDE}\label{subsectiondiscretizationbackwardcomponent}

Let $\ti,\tip\in\pi$. To approximate the solution $(Y,Z)$
to (\ref{canonicFBSDE}), we need two approximations, one for the $Y$
component and one for the $Z$ component. Write (\ref{canonicFBSDE})
over the
interval $[\ti,\tip]$ and take $\cF_\ti$-conditional expectations
to obtain
[recalling that $\Theta_s=(X_s,Y_s,Z_s)$]
%
\begin{equation}
\label{eqnumericsBasicY} Y_{\ti}=\bE_{{\ti}} \biggl[ Y_{t_{i+1}} +
\int_{t_i}^{t_{i+1}} f(s,\Theta_s) \,\uds
\biggr].
\end{equation}
For the $Z$ component, one multiplies (\ref{canonicFBSDE}) (written over
the interval $[\ti,\tip]$) by the Brownian increment, $\Delta
W_\ip:=W_\tip-W_\ti$, and takes $\cF_\ti$-conditional expectations
to obtain
(using It\^o's isometry) the implicit formula
%
\begin{eqnarray}
\label{eqnumericsBasicZ} \qquad 0 &=& \bE_{{\ti}} \biggl[ \Delta W_\ip
\biggl(Y_{t_{i+1}} + \int_{t_i}^{t_{i+1}}f(s,
\Theta_s)\,\uds\biggr) \biggr] -\bE_{{\ti}} \biggl[ \int
_{t_i}^{t_{i+1}} Z_s\,\uds\biggr].
\end{eqnarray}
One now obtains a scheme by approximating the Lebesgue integral via the
\mbox{$\o$-}integration rule (indexed by a parameter $\theta\in[0,1]$),
that is, for
some function~$\psi$
\[
\int_\ti^\tip\psi(s)\,\uds\approx\bigl[\theta
\psi(\ti) + (1-\theta)\psi(\tip) \bigr](\tip-\ti),\qquad\theta\in[0,1].
\]
This type of approximation of the integral is generally known to be of
first order for $\theta\neq
1/2$ and of higher order for $\theta=1/2$ (see end of this section).
Unfortunately, with the results
obtained so far (see Section~\ref{secregularity}) we are not able to prove
the convergence of a general higher order approximation in its full
generality; roughly, the issue boils down to obtaining controls on
$\llvert \partial^2_{xx} v\rrvert $ where $v$ is
solution to (\ref{eqviscosityFBSDE}). However, under the results of
Section~\ref{secregularity}, we do not even know if $\partial^2_{xx} v$
exists. Under the assumption that $f$ is independent of $z$, we can
prove that
the scheme is indeed of higher order (in the $y$ component); the
general case
is left for future research.

From (\ref{eqnumericsBasicZ}) above, we have [compare with
(\ref{Z-bar-ti-pi})]
\[
\bar{Z}_{t_{i}}:= \frac{1}{h} \bE_{t_{i}} \biggl[ \int
_{t_i}^{t_{i+1}} Z_s \,\uds\biggr] =
\frac{1}{h} \bE_{t_{i}} \biggl[ \Delta W_{{i+1}} \biggl(
Y_{t_{i+1}} + \int_{t_i}^{t_{i+1}}f(s,
\Theta_s)\,\uds\biggr) \biggr],
\]
and we approximate $(Z_s)_{s\in[\ti,\tip]}$ via
$\bar Z_\ti$ and $\bar Z_\tip$
rather than $Z_\ti$ or $Z_\tip$. Following the notation for $\Theta$,
we
denote $\bT_\ti:=(X_\ti,Y_\ti,\bar{Z}_\ti)$ and using the
$\o$-integration rule, it follows
%
\begin{eqnarray}
\qquad Y_{t_i} &=& \bE_{t_{i}} \biggl[
Y_{t_{i+1}} + h \bigl[\theta f(\ti,\bT_\ti) + (1-\theta)f(\tip,
\bT_\tip) \bigr]
\nonumber\\[-8pt]\label{eqYexp} \\[-8pt]\nonumber
&&\hspace*{148pt}{} + \int_{\ti}^{\tip}R(s) \,\uds\biggr],
\\
\label{eqZexp} \bar{Z}_{t_{i}} &=& \bE_{t_{i}} \biggl[
\frac{\Delta W_{{i+1}} }{h} \biggl( Y_{t_{i+1}} + (1-\theta)f(\tip,
\bT_\tip)h + \int_{\ti}^{\tip}R(s)\,\uds
\biggr) \biggr],
\end{eqnarray}
where the error term is, for $s\in[\ti,\tip]$, defined as $R(s):=
\theta R^I(s)+(1-\theta)R^E(s)$ where
%
\begin{eqnarray}\label{eqerrorcomponentR-sec4}
R^I(s) &:=& f(s,\Theta_s)-f(\ti,\bar
\Theta_\ti) \quad\mbox{and}
\nonumber\\[-8pt]\\[-8pt]\nonumber
R^E(s)&:=& f(s,
\Theta_s)-f(\tip,\bar\Theta_\tip).
\end{eqnarray}
%

\begin{remark}
For the error analysis here and in the following section, we always understand
the set of RVs $\{(Y_\ti,\bar{Z}_\ti)\}_{\ti\in\pi}$ as the true
solution of the BSDE on the partition points $\ti\in\pi$ but in the
set-up of
(\ref{eqYexp}) and (\ref{eqZexp}). We emphasize that our numerical scheme
does not aim at approximating $Z$ itself over $\pi$ but the family $\{
\bar
Z_\ti\}_{\ti\in\pi}$.
\end{remark}

The order of the approximation depends on the smoothness of driver $f$ and
the properties of the other coefficients. Ignoring the error term $R$,
we find
the discretization scheme stated in (\ref{eqY})--(\ref{eqZ}).
We point out that we aim at first-order schemes, so setting $Z_N=0$ is not
an issue. For a higher order
schemes, $Z_T$ needs to be approximated in a more robust
fashion, for example, following (\ref{boundforZT}),
$Z_T =(\nabla_x g)(X_T)\sigma(T,X_T) \approx(\nabla_x
g)(X_N)\sigma(T,X_N)=Z_N$ (under the extra assumption that $\nabla g$ is
Lipschitz).

We can already estimate the error on the terminal conditions, which is the
first group of terms in the global error estimate from the Fundamental
Lemma~\ref{lemfundamental}.

%
\begin{lemma}
\label{lemmaerroronterminalconditons}
Let \textup{(HX0)}, \textup{(HY0)} hold. Then there exists a constant $c$ such that [recall
(\ref{Z-bar-ti-pi})]
%
\begin{eqnarray}\label{equationerroronterminalconditons}
\bE\bigl[\llvert Y_\tN-Y_N \rrvert
^p\bigr]^{1/p} &\le& c h^{\gamma}\qquad\mbox{for any }p
\geq2 \quad\mbox{and}
\nonumber\\[-8pt]\\[-8pt]\nonumber
\bE\bigl[\llvert\bar{Z}_\tN-Z_N \rrvert^2h\bigr] &\le& c h,
\end{eqnarray}
where $\gamma$ is the order of the approximation $\{X_i\}_{i = 0,
\ldots, N}$
of $X$ [according to (\ref{eqSDEdiscretizationL2})].

Assume that $g\in C^1_b$ and that $\nabla g$ is Lipschitz continuous. Define
$Z_N:=(\nabla_x g)(X_N)\sigma(T,X_N)$ then
$\bE[\llvert\bar{Z}_\tN- Z_N \rrvert^2h] \le c h^2$.
\end{lemma}

\begin{pf}
The error estimate on $Y_\tN$ results from the Lipschitz regularity of~$g$
and the estimate on $\bE[\llvert X_\tN-X_N\rrvert^2]$
given by
(\ref{eqSDEdiscretizationL2}).
For the error estimate on~$Z$, we have $Z_N=0$, and $\bar{Z}_\tN=
Z_T$, which
in turn implies $\bE[\llvert\bar{Z}_\tN- Z_N \rrvert^2h]
= \bE[\llvert Z_T\rrvert^2]
h \le c h$ where we have used (\ref{boundforZT}).

In the case where $g\in C^1_b$ and $\nabla g$ is Lipschitz,
the estimate follows easily using that $\bar Z_T=Z_T=\nabla
g(X_T)\sigma(T,X_T)$ and using the Lipschitz property of $\nabla g$ and
$\sigma$, the Cauchy--Schwarz inequality and (\ref{eqSDEdiscretizationL2}).
\end{pf}

\subsection{Existence and local estimates for the general \texorpdfstring{$\theta$}{theta}-scheme}\label{subsectionwelldefinednessandlocalestimate}

In this subsection, we start the study of the
$\o$-scheme (\ref{eqY})--(\ref{eqZ}) by analyzing \emph{one step}
of it,
that is, going from time $\tip$ to $\ti$. To simplify notation, we define
$f_\ip:=f(\tip,X_\ip,Y_\ip,Z_\ip)$ and $A_\ip:= Y_{i+1} +
(1-\theta) f_\ip
h$.

Along with \textup{(HX0)} and \textup{(HY0)}, we make the temporary assumption that
$Y_\ip,Z_\ip,f_\ip\in L^2$ (this integrability assumption is clearly
satisfied by $Y_N$, $Z_N$ and $f_N$) and analyze how, when $\theta>
0$, this
integrability carries on to
the next time step.

Note that for $\o=0$ (i.e., the explicit case) the scheme step is well defined
as $Y_i$
and $Z_i$ can be easily computed. For $\o>0$, there is no issue in defining
$Z_i$ from~(\ref{eqZ}), but unlike in the Lipschitz case, it is not
immediate that the solution $Y_i$ to the implicit equation (\ref{eqY})
exists. We need to show first that there exists a unique $Y_i$ solving
$Y_i = \bE_i [ A_{i+1} ]+\theta f(\ti,X_i,Y_{i}, Z_{i})h$,
where $\bE_i[A_\ip]$, $X_i$ and $Z_i$ are already known. This follows from
Theorem 26.A in \citeauthor{Zeidler1990B} [(\citeyear{Zeidler1990B}),  page~557].
Define (almost surely) the map $F\dvtx  y \mapsto y - \theta
f (\ti,X_i(\omega),y,
Z_{i}(\omega) )h$. This map is strongly monotone (increasing) in
the sense
of Definition 25.2 in \citet{Zeidler1990B}, that is, there exists a
$\mu> 0$ such
that for all $y,y'$,
\begin{eqnarray*}
\bigl\langle y'-y,F\bigl(y'\bigr)-F(y) \bigr
\rangle&\ge&\mu\bigl\llvert y'-y\bigr\rrvert^2.
\end{eqnarray*}
Indeed, from \textup{(HY0)} and Remark~\ref{rmkoneside} we have
\begin{eqnarray*}
\bigl\langle y'-y, F\bigl(y'\bigr)-F(y) \bigr
\rangle&\ge&(1-\o L_y h )\bigl\llvert y'-y\bigr\rrvert
^2,
\end{eqnarray*}
so if $h < 1/({\theta L_y})$ we can take
$\mu= (1-\o L_y h)>0$. This (almost surely) guarantees the existence of
a unique $Y_i(\omega) = F^{-1} [ \bE_i[(A_\ip)](\omega)]$, as needed.
By the monotonicity of F, $Y_i$ can be quickly computed using, for example,
Newton--Raphson-type methods.
Now, $Y_i$ so defined is only an $\cF_i$-measurable random
variable.\footnote{The previous explanation only justified the
existence of $Y_i$ as a
function
from $\Omega$ to $\bR^k$. To obtain that it is measurable, one
should rather consider the map
$G\dvtx  (a,y) \mapsto(a, y - \theta f(\ti,a,y) h )$, where
$a=(x,z) \in
\R^{d \times k \times d}$ and $f(t,a,y) = f(t,x,y,z)$. It is again
seen to be
strongly monotonous, so it is invertible and Theorem 26.A in
\citet{Zeidler1990B} asserts that
$G^{-1}$ is continuous (Lipschitz in fact), hence measurable.
}

The following proposition guarantees that if $\o>0$, the pair
$(Y_i,Z_i)$ and the term $f_i$ are square integrable provided the
corresponding random variables at $\tip$ also are. So for every $N$, by
iteration,
$(Y_i,Z_i)$ is well defined for $i=N-1,\ldots,0$.
For $\theta\ge1/2$,
this estimate also leads to a uniform bound, as will become clear in
the next
section (Proposition~\ref{thsizebound}).

%
\begin{proposition}
\label{thwellposed}
Let \textup{(HX0)}, \textup{(HY0)} hold, $\o\in[0,1]$ and take $h\le\break 
\min\{1,[4\theta(L_y+3d\theta L_z^2 )]^{-1}\}$. Then there
exists a
constant $c$ such that for any $i\in\{0,\ldots,N-1\}$
%
\begin{eqnarray}
\label{eqlocalsizeestimate}
\nonumber
&& \llvert Y_i\rrvert^2 +
\frac{1}{2d} \llvert Z_{i}\rrvert^2h + 2
\theta^2 \llvert f_i\rrvert^2h^2
\\
&&\qquad \le( 1+ c h ) \bE_i \biggl[ \llvert Y_{i+1}
\rrvert^2 + \frac{1}{8d}\llvert Z_{i+1}\rrvert
^2 h \biggr] + c h
\\
&&\quad\qquad{} + c \bigl(\llvert X_i\rrvert^2 +
\bE_{i} \bigl[\llvert X_\ip\rrvert^2 \bigr]
\bigr)h + 2 (1-\o)^2 \bE_i \bigl[ \llvert
f_\ip\rrvert^2 \bigr] h^2.\nonumber
\end{eqnarray}
\end{proposition}

\begin{pf*}{Proof of Proposition~\ref{thwellposed}}
Let
$i\in\{0,\ldots,N-1\}$.
First, we estimate $ Z_i$. The martingale property of $\Delta W_{i+1}$ yields
%
\begin{eqnarray}
\label{eqvariancereductionZ} Z_ih &=& \bE_i [ \Delta
W_{i+1} A_\ip] = \bE_i \bigl[ \Delta
W_{i+1} \bigl( A_\ip- \bE_i[A_\ip]
\bigr) \bigr].
\end{eqnarray}
By the Cauchy--Schwarz inequality,
%
\begin{eqnarray}
\label{eqestZ} \llvert Z_i\rrvert^2 h & \le& d \bigl\{
\bE_i\bigl[A_\ip^2\bigr]-
\bE_i[A_\ip]^2 \bigr\}.
\end{eqnarray}
We now proceed with the estimation of $Y_i$. We first rewrite
\[
Y_i = \bE_i[ A_\ip] + \theta
f_i h\quad \Longleftrightarrow\quad Y_i - \theta f_i h =
\bE_i[A_\ip]
\]
and then square both sides of the RHS of the above equivalence to
obtain
\[
\llvert Y_i \rrvert^2 = \bE_i[A_\ip]^2
+ 2\o\langle Y_i, f_i \rangle h - \theta^2
\llvert f_i\rrvert^2 h^2.
\]
This simple manipulation allows us to take advantage of the
monotonicity of $f$ [see (\ref{HY0monotonicity})] and will be
reused frequently. By the estimate of Remark~\ref{rmkoneside}, with an
$\alpha>0$ to be chosen later, the previous equality leads to
\begin{eqnarray*}
\llvert Y_i\rrvert^2 &\le& \bE_i[A_\ip]^2
+ 2 \o(L_{y}+ \alpha) \llvert Y_i\rrvert^2
h + \o B(i,\alpha) + \frac{ 3\o
L_{z}^2}{2 \alpha}\llvert Z_i\rrvert
^2 h - \theta^2 \llvert f_i\rrvert
^2 h^2,
\end{eqnarray*}
where $B(i,\alpha):=({3 L^2}h + {3 L_x^2}\llvert X_i\rrvert^2
h)/(2\alpha) $.
Now, for $\epsilon={1}/{d}$, we combine the above estimate with
(\ref{eqestZ}) to obtain
\begin{eqnarray*}
\llvert Y_i\rrvert^2 + \epsilon\llvert
Z_{i}\rrvert^2 h &\le& ( 1 - \epsilon d )
\bE_i[A_\ip]^2 + \epsilon d
\bE_i\bigl[A_\ip^2\bigr]
\\
&& {}+ 2\o(L_{y}+ \alpha) \llvert Y_i\rrvert
^2 h + \frac{ 3
\o
L_{z}^2}{2 \alpha}\llvert Z_i\rrvert
^2 h +\o B(i,\alpha) - \theta^2 \llvert f_i
\rrvert^2 h^2.
\end{eqnarray*}
Reorganizing the terms leads to
%
\begin{eqnarray}\label{eqwellposep1}
&& \bigl( 1 - 2 \o( L_{y}+ \alpha) h \bigr)
\llvert
Y_i\rrvert^2 + \biggl( \epsilon- \frac{ 3 \o L_{z}^2}{2\alpha}
\biggr) \llvert Z_{i}\rrvert^2 h
\nonumber\\[-8pt]\\[-8pt]\nonumber
&&\qquad  \le\bE_i
\bigl[A_\ip^2\bigr] + \o B(i,\alpha) -
\theta^2 \llvert f_i\rrvert^2
h^2.
\end{eqnarray}
Using again Remark~\ref{rmkoneside} with $\alpha'>0$, we obtain
\begin{eqnarray*}
A_\ip^2 &\le& \llvert Y_{i+1}\rrvert
^2 +(1-\o) 2\bigl(L_{y}+ \alpha'\bigr)
\llvert Y_{i+1}\rrvert^2 h
\\
&&{} + (1-\o)\frac{3 L_{z}^2}{2\alpha'}\llvert Z_{i+1}\rrvert^2 h
+ (1-\o) B\bigl(\ip,\alpha'\bigr) + (1-\o)^2 \llvert
f_\ip\rrvert^2 h^2,
\end{eqnarray*}
which in turns leads to
%
\begin{eqnarray}
\label{eqwellposep2} &&
\bigl( 1 - 2\o( L_{y} + \alpha) h \bigr)
\llvert Y_i\rrvert^2 + \biggl( \epsilon-
\frac{ 3\o L_{z}^2}{2\alpha} \biggr) \llvert Z_{i}\rrvert^2 h\nonumber
\\
&&\qquad \le\bigl( 1+ (1-\o) 2 \bigl(L_{y}+
\alpha'\bigr) h \bigr) \bE_i \bigl[ \llvert
Y_{i+1} \rrvert^2 \bigr]
\nonumber\\[-8pt]\\[-8pt]\nonumber
&&\quad\qquad{} + (1-\o)\frac{3 L_{z}^2}{2\alpha'} \bE_i \bigl[
\llvert Z_{i+1}\rrvert^2 \bigr] h + H^\theta_i
\\
&&\quad\qquad{} + \o B(i,\alpha) + (1-\o)\bE_i
\bigl[B\bigl(i+1,\alpha'\bigr)\bigr],\nonumber
\end{eqnarray}
where
%
\begin{eqnarray}
\label{eqwellposeH} H^{\o}_i&:=& (1-\o)^2
\bE_i \bigl[ \llvert f_\ip\rrvert^2 \bigr]
h^2 - \o^2 \llvert f_i\rrvert^2
h^2.
\end{eqnarray}
Now,\vspace*{-2pt} we choose $\alpha= 3 d \theta L_z^2$
(so that $\epsilon- \frac{ 3\o L_{z}^2}{2\alpha} = \frac{1}{2d}$)
and $\alpha' = 24 d(1-\theta)L_z^2$
[so\vspace*{2pt} that $(1-\o)\frac{3 L_{z}^2}{2\alpha'} \le\frac{1}{16d} $].
Since
$h\le\min\{1,[ 4\theta( L_y + 3d\theta L_z^2 )]^{-1}\} $ it is true
that $ 2\o( L_{y}+ \alpha) h \le{1}/{2}$.
We also observe that for $x\in[0,1/2]$, $1 \le{1}/({1-x}) \le
1+2x \le2$ and as a consequence
\begin{eqnarray*}
&& \llvert Y_i\rrvert^2 + \frac{1}{2 d} \llvert
Z_{i}\rrvert^2 h
\\
&&\qquad \le\bigl( 1 + 4 \o( L_{y}+ \alpha) h \bigr) \bigl( 1+ 2(1-
\o) \bigl(L_{y}+\alpha'\bigr) h \bigr)
\bE_i \bigl[ \llvert Y_{i+1} \rrvert^2 \bigr]
\\
&&\quad\qquad{} + \frac{1}{8 d} \bE_i \bigl[ \llvert
Z_{i+1}\rrvert^2 \bigr] h + 2 \o B(i,\alpha) + 2(1-\o)
\bE_i \bigl[ B\bigl(i+1,\alpha'\bigr) \bigr] + 2
H^{\o}_i.
\end{eqnarray*}
Defining $c:=
4\o( L_{y}+ \alpha) + 2 (1-\o)
(L_{y}+ \alpha') + 8\o( L_{y}+ \alpha)(1-\o) (L_{y}+ \alpha')$, we clearly
have
\begin{eqnarray*}
&& \bigl( 1 + 4 \o( L_{y}+ \alpha) h \bigr) \bigl( 1+ 2(1-\o)
\bigl(L_{y}+ \alpha'\bigr) h \bigr) \leq1+ch.
\end{eqnarray*}
We can now conclude to the announced estimate
%
\begin{eqnarray}
\label{eqwellposep3}
&& \llvert Y_i\rrvert^2 +
\frac{1}{2d} \llvert Z_{i}\rrvert^2 h\nonumber
\\
&&\qquad \le (1+c h)
\biggl( \bE_i \bigl[\llvert Y_{i+1} \rrvert^2
\bigr] +\frac{1}{8d} \bE_i \bigl[\llvert Z_{i+1}
\rrvert^2 \bigr] h \biggr)
\\
&&\quad\qquad{} + 2 \o B(i,\alpha) + 2(1-\o)\bE_i\bigl[B\bigl(i+1,\alpha
'\bigr)\bigr] + 2 H^{\o}_i,\nonumber
\end{eqnarray}
provided one passes the term $-2 \theta^2
\llvert f_i\rrvert^2 h^2$ in $2H^{\o}_i$ to the LHS. This
completes the proof.
\end{pf*}

\subsection{Local time-discretization error}\label{subsectionlocaldiscretizationerror}

As announced in Sections~\ref{subsectionmainresultsonconvergenceandorganisation} and
\ref{subsectionfundamentallemma},
we now proceed to estimating the one-step discretization errors $\tau_i(Y)$
and $\tau_i(Z)$
[see~(\ref{eqauxlocalestimateYYZZ}) for the definition], and then
their sum.
We thus obtain an estimate for the second group of terms in estimate
(\ref{eqauxfullestimateerrorYYZZ}),
which is summarized in Proposition~\ref{propglobaldiscretizationerror}.

We\vspace*{1pt} follow the notation of Section~\ref{subsectionfundamentallemma} and
write, for $i=0,1,\ldots,N-1$,
$\widehat{Y}_i=Y_{i,(Y_\tip,\bar{Z}_\tip)}$ and
$\widehat{Z}_i=Z_{i,(Y_\tip,\bar{Z}_\tip)}$; that is,
$(\widehat{Y}_i,\widehat{Z}_i)$ is the solution to
%
\begin{eqnarray}
 \widehat{Y}_i &=& \bE_{t_{i}} \bigl[
Y_{\tip} + (1-\theta)f(t_\ip,X_\ip,Y_{\tip},
\bar{Z}_{\tip})h \bigr]
\nonumber\\[-8pt]\label{eqdefinitionofhatY} \\[-8pt]\nonumber
&&{}  +\theta f(t_i,X_i, \widehat{Y}_{i},\widehat{Z}_{i})h,
\\
\label{eqdefinitionofhatZ} \widehat{Z}_i &=& \bE_{t_{i}} \biggl[
\frac{\Delta W_{{i+1}}}h \bigl( Y_{t_{i+1}} + (1-\theta)f(t_\ip,X_{\ip
},Y_{\tip},
\bar{Z}_{\tip})h \bigr) \biggr].
\end{eqnarray}

%
\begin{remark} \label{rmkhat}
We know from Proposition~\ref{thwellposed} that, under the assumption
$h\le\min\{1, [4\theta( L_y + 3 d\theta L_z^2
)]^{-1}\}$,
the RV's $\{(\widehat{Y}_i,\widehat{Z}_i)\}_{i=0,\ldots,N}$ are well
defined and square integrable.
Furthermore, estimate (\ref{eqlocalsizeestimate}),
together with the growth assumption on $f$ in \textup{(HY0)},
(\ref{eqSDEdiscretizationL2}) for $X_\ip$,
Theorem~\ref{theo-existenceuniquenesscanonicalFBSDE} for $Y_\tip$ and
Corollary~\ref{barZinLp} for $\bar{Z}_\tip$,
guarantee immediately that for any $p \ge2$, there exists a constant $c$
such that
%
\begin{eqnarray}
\label{equniformsizeestimateforhatY} &&\sup_{N \in\N} \max_{i= 0,
\ldots, N} \bE
\bigl[\llvert\widehat{Y}_i\rrvert^p\bigr] \le c.
\end{eqnarray}
This fact will be needed later in Section~\ref{sectionthetaimplicitschemes} (in Lemma~\ref
{lemcontrolsumlackofregularity}).
\end{remark}


The next result estimates the one-step discretization errors $\tau_i(Y)$
and $\tau_i(Z)$ of the approximation in terms of the error process $R$ [as
defined in (\ref{eqerrorcomponentR-sec4})]. Afterward, we discuss the
behavior of $R$ itself.

%
\begin{lemma}
\label{lemerrorR}
Let \textup{(HX0)} and \textup{(HY0)} hold and assume that $h\le1/ (4\theta L_y)$. Then for
any $\o\in[0,1]$ there
exists a constant $c$ such that for any $i\in\{0,\ldots,N-1\}$
%
\begin{eqnarray*}
\bE\bigl[\llvert Y_{t_i} - \widehat{Y}_i\rrvert
^2 + \llvert\bar{Z}_{t_{i}}-\widehat{Z}_i
\rrvert^2 h\bigr] & \le& c \bE\biggl[ \biggl( \int
_{t_i}^{t_{i+1}} R(s) \,\uds\biggr)^2 \biggr]
+ c L_x^2 \mathrm{ERR}_{\pi}(X)^2
h^2. 
\end{eqnarray*}
%
%
\end{lemma}

\begin{pf} Let $i\in\{0,\ldots,N-1\}$. Recalling (\ref{eqZexp}),
(\ref{eqdefinitionofhatZ}) and the definition $\bT_\ti:=(X_\ti,Y_\ti,\bar
{Z}_\ti)$
we have
\begin{eqnarray*}
\bar{Z}_{t_{i}}-\widehat{Z}_i &=& \bE_{{i}}
\biggl[ \frac{\Delta W_{{i+1}} }{h} \biggl( (1-\o) \bigl[f(\tip,\bT_\tip
) -
f(t_\ip,X_{\ip},Y_{\tip},\bar{Z}_{\tip})
\bigr] h
\nonumber\\[-8pt]\\[-8pt]\nonumber
&&\hspace*{221pt}{} + \int_{\ti}^{\tip}R(s)\,\uds\biggr) \biggr],
\end{eqnarray*}
which by the Cauchy--Schwarz inequality and the Lipschitz property of
the map $x\mapsto f(\cdot,x,\cdot,\cdot)$ leads to
\[
h \llvert\bar{Z}_{t_{i}}-\widehat{Z}_i\rrvert
^2 \le2 d \bE_i \biggl[ \biggl( \int
_{t_i}^{t_{i+1}} R_{u} \,\udu
\biggr)^2 \biggr] + 2 d (1-\o)^2 L_x^2
\bE_{i} \bigl[ \llvert X_{\tip} - X_{\ip
}\rrvert
^2 \bigr]h^2.
\]
For the $Y$-part, similarly by recalling (\ref{eqYexp}) and
(\ref{eqdefinitionofhatY}), we have
\begin{eqnarray*}
\hspace*{-5pt}&& Y_{t_i} - \widehat{Y}_i
\\
\hspace*{-5pt}&&\qquad =  \bE_{i} \biggl[
\int_{ti}^{\tip}R(s)\,\uds+ (1-\o) \bigl(f(\tip,
\bT_\tip) -f(t_\ip,X_{\ip},Y_{\tip},
\bar{Z}_{\tip}) \bigr)h \biggr]
\\
\hspace*{-5pt}&&\quad\qquad{}+ \o\bigl( f(\ti,\bT_\ti) - f(\ti,X_i,
\widehat{Y}_{i},\widehat{Z}_{i}) \bigr) h
\\
\hspace*{-5pt}&&\qquad =  \bE_{i} \biggl[ \int_{\ti}^{\tip}R(s)
\,\uds+ (1-\o) \bigl(f(\tip,\bT_\tip) - f(t_\ip,X_{\ip},Y_{\tip},
\bar{Z}_{\tip}) \bigr)h \biggr]
\\
\hspace*{-5pt}&&\quad\qquad{}+ \o\bigl( f(\ti,X_\ti,Y_{\ti},
\bar{Z}_{\ti}) - f(\ti,X_i,Y_{\ti},
\widehat{Z}_{i}) \bigr)h
\\
\hspace*{-5pt}&&\quad\qquad {}+ \o\bigl( f(\ti,X_i,Y_{\ti},\widehat{Z}_{i})
- f(\ti,X_i,\widehat{Y}_{i},\widehat{Z}_{i})
\bigr)h.
\end{eqnarray*}
To obtain the estimate for $\llvert Y_{t_i} - \widehat{Y}_i\rrvert
^2$, similarly as in
the proof of Proposition~\ref{thwellposed}, we pass the last term in
the RHS to
the LHS, square both
sides, expand the
square on the LHS, pass the cross term to the RHS and dominate it on
the RHS
using~(\ref{HY0monotonicity}). By collecting only the convenient terms
in the LHS and using assumption \textup{(HY0)} on the RHS, we get
\begin{eqnarray*}
\llvert Y_{t_i} - \widehat{Y}_i\rrvert^2
&\le& 3 \bE_{i} \biggl[ \int_{\ti}^{\tip}R(s)
\,\uds\biggr]^2 + 6 \o^2 L_{z}^2
\llvert\bar{Z}_{\ti} - \widehat{Z}_{i}\rrvert
^2 h^2 + 2 \o L_y \llvert
Y_{\ti}- \widehat{Y}_{i}\rrvert^2 h
\\
&&{}+ 6 \o^2 L_{x}^2\llvert X_{\ti} -
X_{i}\rrvert^2h^2 + 3 (1-\o)^2
L_{x}^2\bE_i\bigl[\llvert X_{\tip} -
X_{\ip}\rrvert^2\bigr] h^2,
\end{eqnarray*}
which implies, using the estimate for $\llvert\bar Z_{\ti} -
\widehat{Z}_i\rrvert^2$,
that
\begin{eqnarray*}
&& ( 1 - 2 \o L_y h ) \llvert Y_{t_i} - \widehat{Y}_i\rrvert^2
\\
&&\qquad \le\bigl( 3 + 12 d \o^2 L_{z}^2 h \bigr)
\bE_i \biggl[ \biggl( \int_{t_i}^{t_{i+1}}
R(s) \,\uds\biggr)^2 \biggr] + 6 \o^2 L_{x}^2
\llvert X_{\ti} - X_{i}\rrvert^2h^2
\\
&&\quad\qquad{}+ 3 (1-\o)^2 L_{x}^2 \bigl( 1 + 4d
\o^2 L_z^2 h\bigr) \bE_i \bigl[
\llvert X_{\tip} - X_{\ip}\rrvert^2 \bigr]
h^2.
\end{eqnarray*}
Noting\vspace*{1pt} that $h$ is such that $2 \o L_y h \le1/2$ and
by combining the estimates for $\llvert Y_{\ti} - \widehat
{Y}_i\rrvert^2$ and
$\llvert\bar Z_{\ti} - \widehat{Z}_i\rrvert^2$ the sought
result follows after taking
expectations and using~(\ref{eqSDEdiscretizationL2}) for
$X$.
\end{pf}


We now estimate the integral of the error function $R$ [see
(\ref{eqerrorcomponentR-sec4})].

%
\begin{lemma}
\label{theocontrollingIntR}
Let \textup{(HX0)}, \textup{(HY0$_{\mathrm{loc}}$)} hold.
Then there exists $c>0$ such that, for any $\o\in[0,1]$ and
$i\in\{0,\ldots,N-1\}$,
\begin{eqnarray*}
&& \bE\biggl[ \biggl( \int_{\ti}^{\tip}R(s)\,\uds
\biggr)^2 \biggr]
\\
&&\qquad  \le c L_t^2 h^3 + c L_x^2
\mathrm{REG}_\pi(X)^2 h^2 + c
L_y \mathrm{REG}_\pi(Y)^2 h^2
\\
&&\quad\qquad{}+ c L_z^2 \bE\biggl[ \int_{t_i}^{t_{i+1}}
\llvert Z_s - \bar{Z}_\ti\rrvert^2 \,\uds+
\int_{t_i}^{t_{i+1}} \llvert Z_s -
\bar{Z}_{\tip}\rrvert^2 \,\uds\biggr] h.
\end{eqnarray*}
\end{lemma}

\begin{pf}
Following from (\ref{eqerrorcomponentR-sec4}), we estimate $R$ via
$R^I$ and $R^E$: using \textup{(HY0$_{\mathrm{loc}}$)}, Cauchy--Schwarz's inequality and
Fubini's
theorems we have\vspace*{1pt}
[recall that $\Theta=(X,Y,Z)$ and $\bT_\ti=(X_\ti,Y_\ti,\bar
{Z}_\ti)$]
\begin{eqnarray*}
&& \bE\biggl[ \biggl(\int_\ti^\tip
R^I(s) \,\uds\biggr)^2 \biggr]
\\
&&\qquad = \bE\biggl[ \biggl( \int_{t_i}^{t_{i+1}} \bigl[f(s,
\Theta_s) \pm f(s,X_s,Y_\ti,Z_s) -
f(\ti,\bT_\ti)\bigr] \,\uds\biggr)^2 \biggr]
\\
&&\qquad  \le2 h \bE\biggl[ \int_{t_i}^{t_{i+1}} 3
L_y^2 \bigl( 1 + \llvert Y_s\rrvert
^{2(m-1)} + \llvert Y_\ti\rrvert^{2(m-1)} \bigr)
\llvert Y_s - Y_\ti\rrvert^2 \,\uds+
\alpha_i \biggr]
\\
&&\qquad  \le2 h \biggl( \int_{t_i}^{t_{i+1}}
L_y^2 \bE\bigl[ 3 \bigl( 1 + \llvert Y_s
\rrvert^{4(m-1)} + \llvert Y_\ti\rrvert^{4(m-1)}
\bigr) \bigr]^{1/2} \bE\bigl[ \llvert Y_s -
Y_\ti\rrvert^4 \bigr]^{1/2} \,\uds
\\[-3pt]
&&\hspace*{301pt}{} + \bE[\alpha_i] \biggr),
\end{eqnarray*}
where
$\alpha_i = 3 \int_{t_i}^{t_{i+1}} [ L_t^2\llvert s-\ti
\rrvert + L_x^2
\llvert X_s-X_\ti\rrvert^2 +
L_z^2 \llvert Z_s - \bar{Z}_\ti\rrvert^2 ] \,\uds
$.

Using Theorem~\ref{theo-existenceuniquenesscanonicalFBSDE}
to deal with the $Y$ component, this yields the estimate
\begin{eqnarray*}
&& \bE\biggl[ \biggl(\int_\ti^\tip
R^I(s) \,\uds\biggr)^2 \biggr]
\\
&&\qquad  \le 3 L_t^2
h^3 + 6 L_x^2 \mathrm{REG}_\pi(X)^2
h^2 + 18 c L_y^2 \mathrm{REG}_\pi(Y)^2
h^2
\\
&&\quad\qquad{} + 6 L_z^2 \bE\biggl[ \int
_{t_i}^{t_{i+1}} \llvert Z_s -
\bar{Z}_\ti\rrvert^2 \,\uds\biggr] h.
\end{eqnarray*}
Similar arguments allow a similar estimate for $R^E$ but with terms
$\tip$,
$X_\tip$,
$Y_\tip$ and $\bar{Z}_\tip$ instead of $\ti$, $X_\ti$, $Y_\ti$
and $\bar
Z_\ti$.
\end{pf}

\subsubsection*{The trapezoidal integration case}

Here, we refine the analysis of the local discretization error from Lemma~\ref{theocontrollingIntR} for the case $\theta=1/2$ in order to
obtain better global error estimates. We drop the $Z$-dependence in $f$
due to
lacking regularity results. Approximation (\ref{eqZ})
is found by approximating the last integral on the RHS of
(\ref{eqnumericsBasicZ}) by a first-order approximation and so it
should be
clear that at best the overall order of the scheme would be one (in the next
section we propose a candidate for higher order approximation of $Z$).
We point out nonetheless that many reaction--diffusion
equations have a driver $f$ that only depends on $Y$.
For ease of the presentation, we also assume that $f$ does
not depend on the forward process $X$ and omit the time dependence
(these can
be easily extended).

We write, similarly to (\ref{eqYexp}),
\begin{eqnarray*}
 \int_{t_i}^{t_{i+1}}f(Y_s)\,\uds&=&
\frac{h}2 \bigl[f(Y_{\ti}) + f(Y_{\tip}) \bigr] + \int
_{\ti}^{\tip}R(s)\,\uds,
\end{eqnarray*}
with
\[
R(s):= f(Y_s)-\tfrac{1}2 \bigl[f(Y_{\ti})
+ f(Y_{\tip
}) \bigr],
\]
where, using integration by parts, it can be shown [see
\citet{SuliMayers2003}] that
%
\begin{eqnarray}
\label{eqtraperr} &&\bE\biggl[ \biggl( \int_{\ti}^{\tip}R(s)
\,\uds\biggr)^2 \biggr] \leq\frac{h^6}{12^2} \bE\Bigl[\sup
_{\ti\le t\le\tip} \bigl\llvert\partial_{yy}^2f(Y_t)
\bigr\rrvert^2 \Bigr].
\end{eqnarray}
Hence, in the special case where the driver of FBSDE under
consideration does
not depend on the process $(Z_t)_{0\le t \le T}$ we can take full advantage
of
trapezoidal integration rule provided that the second derivatives of $f$
in the $y$ variable has polynomial growth, so that there exists a constant
$c$
for which
\[
\max_{\ti,\tip\in\pi} \bE\Bigl[ \sup_{\ti\le t\le\tip}\bigl
\llvert\partial_{yy}^2f(Y_t)\bigr\rrvert
^2 \Bigr]\le c.
\]


\subsubsection*{The result on the sum of local errors}

In view of the above lemmas [as well as estimate
(\ref{eqSDEdiscretizationL2}) and the path-regularity Theorem
\ref{thregularity}], we can state the following estimates on the sum
of the
one-step discretization errors, as appearing in the global error estimate
(\ref{eqauxfullestimateerrorYYZZ})
of Lemma~\ref{lemfundamental}.

%
\begin{proposition}
\label{propglobaldiscretizationerror}
Let \textup{(HX0)}, \textup{(HY0$_{\mathrm{loc}}$)} hold and $h\le\min\{1,[4\theta
(L_y+3d\theta
L_z^2 )]^{-1}\} $. For the scheme (\ref{eqY})--(\ref{eqZ})
we have the following local error estimates:
\begin{longlist}[(ii)]
\item[(i)] For any $\theta\in[0,1]$ $\exists c>0$ such that
$\sum_{i=0}^{N-1} \frac{\tau_i(Y)}{h} \le c h$
and $\sum_{i=0}^{N-1} \tau_i(Z) \le c h^2$.
%
%
\item[(ii)] Take $\theta=1/2$ and scheme (\ref{eqY}). Assume additionally
that $f\in C^2$ does not depend on $(t,x,z)$ and
$\partial^2_{yy}f$ has at most polynomial growth, then there exists
$c>0$ such
that
$\sum_{i=0}^{N-1} \frac{\tau_i(Y)}{h}
\le c h^4$.
\end{longlist}
\end{proposition}

%
%

\begin{pf} Recall the definition of $\tau_i(Y)$ and $\tau_i(Z)$
given in (\ref{eqauxlocalestimateYYZZ}). The proof of case (i) is simple:
inject in the estimate of Lemma~\ref{lemerrorR} that
of Lemma~\ref{theocontrollingIntR} and then sum over
$i=0$ to $i=N-1$.
On the resulting inequality,
\begin{eqnarray*}
%
%
\sum_{i=0}^{N-1}
\tau_i(Y) + \tau_i(Z) & \le& c L_t^2
h^2 + c L_x^2 \mathrm{REG}_\pi(X)^2
h + c L_y^2 \mathrm{REG}_\pi(Y)^2
h
\\
&&{}+ c L_z^2 \mathrm{REG}_\pi(Z)^2
h + c L_x^2 \mathrm{ERR}_\pi(X)^2
h,
\end{eqnarray*}
apply (\ref{eqSDEdiscretizationL2}) for $\mathrm{ERR}_{\pi}(X)$,
the path-regularity result (\ref{eqSDEpathregularity})
for $\mathrm{REG}_\pi(X)$, and the path-regularity Theorem~\ref
{thregularity}
for $\mathrm{REG}_\pi(Y)$
and $\mathrm{REG}_\pi(Z)$.
Under \textup{(HX0)} and (HY0$_{\mathrm{loc}}$)
the resulting inequality is
$\sum_{i=0}^{N-1} ( \tau_i(Y)+\tau_i(Z) ) \le c h^2$. The
statement's inequalities now follows.

For the proof of case (ii), remark that (\ref
{eqdefinitionofhatY}) is now
independent of $Z$, and hence using Lemma~\ref{lemerrorR} in combination
with (\ref{eqtraperr}) instead of Lemma~\ref{theocontrollingIntR} yields
the result.
\end{pf}

%
\begin{remark} \label{remarkhigherratewhenhigherregularityoff}
Under the assumption that $f$ only depends on $y$ (i.e., take
$L_t=L_x=L_z=0$) the methodology used above yields that the
first terms in the global error $\mathrm{ERR}_\pi(Y,Z)$ [see
(\ref{eqauxfullestimateerrorYYZZ})] is controlled only by
$\mathrm{ERR}_\pi(X)$ and $\mathrm{REG}_\pi(Y)$. The
term $\mathrm{REG}_\pi(Y)$ follows from the sum of the local discretization
errors, as can be seen from above, while $\mathrm{ERR}_\pi(X)$
follows from
the approximation of the terminal
condition.

These abstract estimates suggest that under stronger regularity assumptions
on $f$ [stronger than \textup{(HY0$_{\mathrm{loc}}$)}], one may improve the estimates
on $\tau(Y)$ and therefore obtain a higher convergence rate. Such
developments are left for future research.
\end{remark}

\section{Convergence of the implicit-leaning schemes \texorpdfstring{($1/2\leq\theta\leq1$)}{(1/2<=theta<=1)}}\label{sectionthetaimplicitschemes}

In this section, we complete the convergence proof of the
theta scheme (\ref{eqY})--(\ref{eqZ}) for $\theta\in[1/2,1]$ as
stated in
Theorem~\ref{thmainresult}. In view of the Fundamental Lemmas~\ref{lemfundamental},~\ref{lemmaerroronterminalconditons} and Proposition~\ref{propglobaldiscretizationerror}, what remains to study is the stability
of the scheme
and estimate $\cR^{\cS}(H)$.

\subsection{Integrability for the \texorpdfstring{$\theta$}{theta}-scheme, for \texorpdfstring{$1/2\leq\theta\leq1$}{1/2<=theta<=1}}\label{subsectionsizeestimatethetaimplicitschemes}

We now show that for $\theta\ge1/2$ the scheme cannot explode as $h$
vanishes. These $L^p$ estimates will be useful in obtaining the
stability of
the scheme.

%
\begin{proposition}
\label{thsizebound}
Let \textup{(HX0)}, \textup{(HY0)} hold, and $h \le\min\{1,[4\theta(L_y+3d\theta
L_z^2 )]^{-1}\} $ and let $\o\in[1/2,1]$. Then for any $p \ge1$,
there exists a constant $c$ such that
\[
\max_{i=0,\ldots,N} \bE\bigl[ \llvert Y_i\rrvert
^{2p} \bigr] +\sum_{i=0}^{N-1}\bE
\bigl[ \bigl(\llvert Z_{i}\rrvert^2 h\bigr)^p
\bigr] \le c \bigl(1 + \bE\bigl[ \llvert X_N\rrvert^{2 m p}
\bigr] \bigr).
\]
\end{proposition}

\begin{pf}
Take $i\in\{0,\ldots,N-1\}$ and define the quantity $I_i:=\llvert
Y_i\rrvert^2 + \frac{1}{8d} \llvert Z_{i}\rrvert^2 h +
\o^2\llvert f(\ti,X_i,Y_i,Z_i)\rrvert^2 h^2$.\vspace*{1pt}
By Proposition~\ref{thwellposed} and that $(1-\theta)^2 \le
\theta^2$,
for $\o\in[1/2,1]$, we have for $\beta_i:= c + c (\llvert X_i\rrvert
^2 + \llvert X_\ip\rrvert^2)$ the inequality
%
\begin{eqnarray}\label{eqLpZ}
&& I_i + \frac{3}{8d} \llvert Z_{i}
\rrvert^2h \le e^{ch} \bE_i [
I_\ip] + \bE_i [ \beta_i ] h.
\end{eqnarray}
As a consequence of Lemma~\ref{lemGron}, we know that, since $\beta
_j \ge0$,
\[
I_i + \frac{3}{8d} \bE_i \Biggl[ \sum
_{j=i}^{N-1} \llvert Z_j\rrvert
^2 h \Biggr] \le e^{c T} \Biggl( \bE_i[
I_N ] + \sum_{j=i}^{N-1}
\bE_i[\beta_j] h \Biggr),
\]
in particular, using Jensen's inequality, we obtain further
\begin{eqnarray*}
&&\llvert I_i\rrvert^p \le2^{p-1}
e^{c p T} \Biggl( \bE_i\bigl[ \llvert I_N\rrvert
^p \bigr] + (N h)^{p-1}\sum_{j=0}^{N-1}
\bE_i\bigl[ \llvert\beta_j\rrvert^p\bigr]
h \Biggr).
\end{eqnarray*}
This then implies, thanks to \textup{(HY0)},
\begin{eqnarray*}
&& \max_{i=0, \ldots, N} \bE\bigl[\llvert I_i\rrvert
^p\bigr]  \le c \bigl(1 + \bE\bigl[ \llvert X_N\rrvert
^{2 m p} \bigr] \bigr)
\\
&&\qquad \Rightarrow\quad \max_{i=0, \ldots, N} \bE\bigl[\llvert Y_i
\rrvert^{2 p}\bigr] \le c \bigl(1 + \bE\bigl[ \llvert X_N
\rrvert^{2 m p} \bigr] \bigr).
\end{eqnarray*}
From (\ref{eqLpZ}), we also have
\begin{eqnarray*}
&& I_i^p + \biggl(\frac{3}{8d} \biggr)^p
\bigl(\llvert Z_{i}\rrvert^2h \bigr)^p
\\
&&\qquad  \le
\biggl(I_i + \frac{3}{8d} \llvert Z_{i}\rrvert
^2h \biggr)^p
\\
&&\qquad  \le e^{ c p h} \bE_i\bigl[I_{i+1}^p
\bigr] + \sum_{j=1}^{p}\pmatrix{p
\cr
j}
\bigl( e^{ch} \bE_i [ I_\ip]
\bigr)^{p-j} \bigl(\bE_i[\beta_i] h
\bigr)^j,
\end{eqnarray*}
so that, applying again Lemma~\ref{lemGron} along with H\"{o}lder's and
Jensen's inequalities
we have
\begin{eqnarray*}
&& \biggl( \frac{3}{8d} \biggr)^p \bE\Biggl[\sum
_{i=0}^{N-1} \bigl(\llvert Z_{i}\rrvert
^2h \bigr)^p \Biggr]
\\
&& \qquad\le e^{c p T} \bE\bigl[\llvert I_{N}\rrvert
^p\bigr] + \sum_{i=0}^{N-1}
e^{c i h} \sum_{j=1}^{p}\pmatrix{p
\cr
j} \bE\bigl[ \bigl( e^{c h} \bE_i [ I_\ip
] \bigr)^{p-j} \bigl(\bE_i[\beta_i] h
\bigr)^j \bigr]
\\
&& \qquad\le e^{c p T} \bE\bigl[\llvert I_{N}\rrvert
^p\bigr] + e^{c p T} \sum_{i=0}^{N-1}
\sum_{j=1}^{p}\pmatrix{p
\cr
j} \bigl( \bE
\bigl[ \llvert I_\ip\rrvert^p \bigr]
\bigr)^{(p-j)/p} \bigl(\bE\bigl[\llvert\beta_i\rrvert
^{p}\bigr] \bigr)^{j/p} h
\\
&& \qquad\le e^{c p T} \bE\bigl[\llvert I_{N}\rrvert
^p\bigr]
\\
&&\quad\qquad{} + e^{c p T} T \sum_{j=1}^{p}
\pmatrix{p
\cr
j} \Bigl( \max_{i=0, \ldots, N} \bE\bigl[ \llvert
I_\ip\rrvert^p \bigr] \Bigr)^{(p-j)/p} \Bigl(\max
_{i=0, \ldots, N}\bE\bigl[ \llvert\beta_i\rrvert
^{p}\bigr] \Bigr)^{j/p}.
\end{eqnarray*}
Due to \textup{(HY0)} and the previous estimates we arrive, as required, at
\[
\bE\Biggl[\sum_{i=0}^{N-1} \bigl(\llvert
Z_{i}\rrvert^2h \bigr)^p \Biggr] \le c
\bigl(1 + \llvert X_N\rrvert^{2 m p} \bigr).
\]\upqed
\end{pf}

\subsection{Stability of the \texorpdfstring{$\theta$}{theta}-scheme for \texorpdfstring{$1/2\leq\theta\leq1$}{1/2<=theta<=1}}\label{subsectionstabilitythetaimplicitschemes}

We now study the stability of the scheme in the sense of
(\ref{eqdef-of-stability-sec4}). We fix $i \in\{ 0, \ldots, N-1 \}$ and
estimate the distance between the outputs
$(\widehat{Y}_i,\widehat{Z}_i)$ [see
(\ref{eqdefinitionofhatY})--(\ref{eqdefinitionofhatZ})] and $(Y_i,Z_i)$
[see (\ref{eqY})--(\ref{eqZ})] as
a function of the distance between the inputs $(Y_\tip,\bar{Z}_\tip
)$ and
$(Y_\ip,Z_\ip)$.

We use the notation
$\delta Y_\ip= Y_\tip- Y_\ip$,
$\delta Z_\ip:= \bar{Z}_\tip- Z_\ip$, as well as
\begin{eqnarray*}
\delta f_\ip&=& f(\tip,X_\ip,Y_\tip,
\bar{Z}_\tip) - f(\tip,X_\ip,Y_\ip,Z_\ip)
\end{eqnarray*}
and
\begin{eqnarray*}
\delta A_\ip&=& \delta Y_\ip+ (1-\theta) \delta
f_\ip h.
\end{eqnarray*}
Then, denoting by
$\widehat{\delta Y_i} = \widehat{Y}_i - Y_i$,
$\widehat{\delta Z_i} = \widehat{Z}_i - Z_i$ and
$\delta\widehat{f}_i = f(\ti,X_i,\widehat{Y}_i,\widehat{Z}_i) -
f(\ti,X_i,Y_i,Z_i)$,
we can write that [compare with
(\ref{eqdefinitionofhatY}), (\ref{eqdefinitionofhatZ}),
(\ref{eqY}) and (\ref{eqZ})]
\begin{eqnarray*}
&&\widehat{\delta Y_i} = \bE_i [ \delta
A_\ip] + \theta\delta\widehat{f}_i h \quad\mbox{and}
\quad\widehat{\delta Z_i} = \bE_i \biggl[
\frac{1}h{\Delta W_{i+1} } \delta A_\ip\biggr].
\end{eqnarray*}

%
\begin{proposition}
\label{thstability}
Let \textup{(HX0)} and \textup{(HY0)} hold. Then there exists a constant $c$ for any
$i\in\{0,\ldots,N-1\}$ and $h \le\min\{1,[4\theta(L_y+d\theta
L_z^2 )]^{-1}\} $ such that
\[
\llvert\widehat{\delta Y_i}\rrvert^2 +
\frac{1}{2d} \llvert\widehat{\delta Z_i}\rrvert
^2h \le( 1+ c h ) \bE_i \biggl[ \llvert\delta
Y_{i+1} \rrvert^2 + \frac{1}{8d}\llvert\delta
Z_{i+1}\rrvert^2 h \biggr] + 2H^{\o}_i,
\]
where
%
\begin{equation}
\label{eqH} H^{\o}_i = (1-\o)^2
\bE_i \bigl[ \llvert\delta f_\ip\rrvert^2
\bigr] h^2 - \theta^2 \bE_i \bigl[ \llvert
\delta\widehat{f}_i\rrvert^2 \bigr] h^2.
\end{equation}
\end{proposition}

\begin{pf}
This proof is very similar to that of Proposition
\ref{thwellposed}, therefore we omit~it.
\end{pf}

We want to control $\cR^{\cS}(H)$.
For the fully implicit scheme ($\theta=1$), we have
$ H^{\theta}_{i} = - \llvert\delta\widehat{f}_i\rrvert^2
h^2 \le0$ and hence the
implicit scheme is stable in the classical sense [of \citeauthor{Chassagneux2012} (\citeyear{Chassagneux2012,Chassagneux2013})]
as we have $\cR^{\cS}(H)\leq0$. The next lemma provides, in our
setting, a
control on $\cR^{\cS}(H)$ for any $\o\geq1/2$.

%
\begin{lemma}
\label{lemcontrolsumlackofregularity}
Let \textup{(HX0)}, \textup{(HY0$_{\mathrm{loc}}$)} hold and take the family $\{H_i\}_{i=0,\ldots,N-1}$
defined in (\ref{eqH}). Then for $\o\geq1/2$ there exists a
constant $c$
such
that
\begin{eqnarray*}
\cR^\cS(H) &=& \max_{i=0,\ldots,N-1} \bE\Biggl[ \sum
_{j=i}^{N-1} e^{c(j-i)h}
H^{\o}_{j} \Biggr]
\\
& \le& c \bE\bigl[\llvert Y_{t_{N}} -Y_N \rrvert
^4 \bigr]^{1/2} h^2 +c \bE\bigl[ \llvert
\bar{Z}_N - Z_N\rrvert^2\bigr]
h^2
\\
&&{}+ c \Biggl(\sum_{i=0}^{N-1}
\tau_i(Y) \Biggr)^{1/2}h + c \Biggl( \sum
_{i=0}^{N-1} \tau_i(Z)
\Biggr)^{1/2}h.
\end{eqnarray*}
\end{lemma}

\begin{pf}
Let $i\in\{0,\ldots,N-1\}$. For ${1}/{2} \le\theta
\le1$, we have
$(1-\theta)^2
\le\theta^2$ and, therefore,
\begin{eqnarray*}
&& \bE\Biggl[ \sum_{j=i}^{N-1}
e^{c(j-i)h} H^{\o}_{j} \Biggr]
\\
&&\qquad  \le\theta^2 \bE\Biggl[ \sum_{j=i}^{N-1}
e^{c(j-i)h} \bigl( \llvert\delta f_{j+1}\rrvert^2 -
\llvert\delta\widehat{f}_j\rrvert^2 \bigr)
h^2 \Biggr]
\\
&&\qquad = \o^2 \bE\Biggl[ \sum_{j=i}^{N-1}
e^{c(j-i)h} \bigl( \llvert\delta f_{j+1}\rrvert^2 -
\llvert\delta f_j + \beta_j \rrvert^2
\bigr) h^2 \Biggr]
\\
&&\qquad  \le\o^2 \bE\Biggl[ \sum_{j=i}^{N-1}
e^{c(j-i)h} \bigl( e^{ch} \llvert\delta f_{j+1}\rrvert
^2 - \llvert\delta f_j\rrvert^2 - 2
\langle\delta f_j, \beta_i \rangle- {
\beta_j }^2 \bigr) h^2 \Biggr]
\\
&&\qquad  \le\o^2 e^{c(N-i)h} \bE\bigl[\llvert\delta f_N
\rrvert^2 \bigr] h^2 - 2 \o^2 \sum
_{j=i}^{N-1} e^{c(j-i)h} \bE\bigl[ \langle\delta
f_j, \beta_j \rangle\bigr] h^2,
\end{eqnarray*}
where $ \beta_i:= \delta\widehat{f}_j - \delta f_j
= f(t_{i},X_j,\widehat{Y}_j,\widehat{Z}_j) -
f(t_{i},X_j,Y_{t_{i}},\bar{Z}_{t_{i}})$ and we used a telescopic sum. Using
now \textup{(HY0$_{\mathrm{loc}}$)} yields
\begin{eqnarray*}
\bE\bigl[ \llvert\delta f_N\rrvert^2 \bigr] &\le& c \bE
\bigl[ 1 + \llvert Y_{t_N}\rrvert^{4(m-1)} + \llvert
Y_{N}\rrvert^{4(m-1)} \bigr]^{1/2} \bE\bigl[\llvert
Y_{t_{N}} -Y_N \rrvert^4 \bigr]^{1/2}
\\
&&{} + c \bE\bigl[ \llvert\bar{Z}_N - Z_N\rrvert
^2 \bigr]
\end{eqnarray*}
and
\begin{eqnarray*}
\bE\bigl[ \langle\delta f_i, \beta_i \rangle\bigr]
h^2 & \le&\bE\bigl[ \llvert\delta f_i\rrvert\llvert
\beta_i \rrvert\bigr] h^2
\\
& \le&\bE\bigl[ \bigl( \llvert\delta f_i\rrvert L_y
\bigl(1+ \llvert\widehat{Y}_i\rrvert^{m-1}+\llvert
Y_\ti\rrvert^{m-1}\bigr) \bigr)^2
\bigr]^{1/2} \bE\bigl[\llvert\widehat{Y}_i -
Y_\ti\rrvert^2 \bigr]^{1/2} h^2
\\
&&{} + \bE\bigl[ \bigl(L_z \llvert\delta f_i
\rrvert\bigr)^2 \bigr]^{1/2} \bE\bigl[\llvert
\widehat{Z}_i - \bar{Z}_\ti\rrvert^2
\bigr]^{1/2} h^2
\\
& \le& c \bE\bigl[ B_i^1 \bigr]^{1/2} \bE
\bigl[\llvert\widehat{Y}_i - Y_\ti\rrvert^2
\bigr]^{1/2} h + c \bE\bigl[ B_i^2
\bigr]^{1/2} \bE\bigl[\llvert\widehat{Z}_i -
\bar{Z}_\ti\rrvert^2 h \bigr]^{1/2} h,
\end{eqnarray*}
where $B_i^2:= \llvert Y_\ti\rrvert^{2m}h + \llvert
Y_i\rrvert^{2m}h + \llvert\bar{Z}_\ti\rrvert^2 h +
\llvert Z_i\rrvert^2 h$ and
\[
B^{1}_i:= h^2 + \llvert
\widehat{Y}_i\rrvert^{4m}h^2 +\llvert
Y_\ti\rrvert^{4m}h^2 + \llvert
Y_i\rrvert^{4m} h^2 +\bigl( \llvert
\bar{Z}_\ti\rrvert^2 h \bigr)^2 + \bigl(
\llvert Z_i\rrvert^2 h \bigr)^2.
\]
From Theorem~\ref{theo-existenceuniquenesscanonicalFBSDE}, Corollary
\ref{barZinLp}, Remark~\ref{rmkhat} and Proposition
\ref{thsizebound}, we have for the first term of the above inequality
\begin{eqnarray*}
\sum_{i=0}^{N-1} \bE\bigl[
B_i^1 \bigr]^{1/2}  \bE\bigl[ \llvert
\widehat{Y}_i - Y_\ti\rrvert^2
\bigr]^{1/2} h
&\le&\Biggl( \sum_{i=0}^{N-1} \bE\bigl[
B_i^1 \bigr] \Biggr)^{1/2} \Biggl(\sum
_{i=0}^{N-1} \tau_i(Y)
\Biggr)^{1/2}h
\\
&\le& c \Biggl(\sum_{i=0}^{N-1}
\tau_i(Y) \Biggr)^{1/2}h
\end{eqnarray*}
and similarly for the second term
\begin{eqnarray*}
&& \sum_{i=0}^{N-1} \bE\bigl[
B_i^2 \bigr]^{1/2} \bE\bigl[\llvert
\widehat{Z}_i - \bar{Z}_\ti\rrvert^2 h
\bigr]^{1/2} h \le c \Biggl( \sum_{i=0}^{N-1}
\tau_i(Z) \Biggr)^{1/2}h.
\end{eqnarray*}\upqed
\end{pf}
%
\subsection{Convergence of the scheme}
By collecting the above results, we can now prove Theorem
\ref{thmainresult}.
\begin{pf*}{Proof of Theorem~\ref{thmainresult}}
The proof is a combination of the Fundamental Lemmas~\ref{lemfundamental} and~\ref{lemmaerroronterminalconditons},
Proposition~\ref{propglobaldiscretizationerror} and stability results
obtained in this section, namely Proposition~\ref{thstability} and Lemma
\ref{lemcontrolsumlackofregularity}.

We move to the proof of part (ii), the case $\o=1/2$. Since in this
case $f$
depends only on $y$, a quick rerun of arguments of the Fundamental Lemma
\ref{lemfundamental}, shows there exists a constant $c>0$ such that
\[
\max_{i=0, \ldots, N} \bE\bigl[\llvert Y_\ti-
Y_i \rrvert^2\bigr] \le c \Biggl\{ \bE\bigl[ \llvert
Y_\tN- Y_N \rrvert^2\bigr] + \sum
_{i=0}^{N-1} \frac{\tau_i(Y)}{h} \Biggr\} + (1+h)
\cR^{\cS}(H).
\]
The first two terms on the RHS can be bounded by $ch^{2\gamma} + c h^4 $,
$c>0$, using Lemma~\ref{lemmaerroronterminalconditons} and
Proposition~\ref{propglobaldiscretizationerror}, respectively.
By Lemma
\ref{lemcontrolsumlackofregularity}, there exists a constant
$c>0$ such
that
\[
\cR^{\cS}(H)\le c \bE\bigl[\llvert Y_{t_{N}} -Y_N
\rrvert^4 \bigr]^{1/2} h^2 + c \Biggl(\sum
_{i=0}^{N-1} \tau_i(Y)
\Biggr)^{1/2}h,
\]
and using again Lemma~\ref{lemmaerroronterminalconditons} and Proposition
\ref{propglobaldiscretizationerror} yields $\cR^{\cS}(H)\le c
h^{2\gamma+2} + c h^{7/2} $. By joining these results, the
theorem's conclusion follows.
\end{pf*}

\subsection{Further remarks}
Here, we discuss a true overall second-order scheme, namely a second-order
discretization for $Z$ and an intuitive variance reduction
technique which we have used throughout but not made formally explicit.

\subsubsection{The candidate for second-order scheme}
For the general case where the driver depends on $Z$, the approximation for
$Z_i$,
namely (\ref{eqZ}), is not
enough to obtain a higher order scheme as it is a first-order
approximation. The proper higher order scheme in its full generality
follows by applying the trapezoidal rule to \emph{all} integrals
present in
(\ref{eqnumericsBasicZ}); as is done for (\ref{eqnumericsBasicY}). With
some
manipulation (left to the reader), we
end up with the following approximation for $Z_i$ [compare with~(\ref{eqZ})]:
\begin{eqnarray*}
Z_i &=& \frac{2}h \bE_{{i}} \bigl[ \Delta
W_{{i+1}} \bigl( Y_{t_{i+1}} + (1-\theta)f(\ti,X_\ip,Y_{i+1},Z_{i+1})h
\bigr) \bigr] - \bE_{{i}} [Z_\ip],
\end{eqnarray*}
with $\theta=1/2$, the terminal condition $Y_N=g(X_N)$, along with
(\ref{eqY}) and a
suitable approximation for $Z_T$. An approximation for $Z_T$
is not trivial and could, for instance, be found via Malliavin calculus.
The
general
treatment of such a scheme is left for future research.

Another type of second-order scheme can be found in
Crisan and\break 
Manolarakis (\citeyear{CrisanManolarakis2010b}); the approximation there is based on
It\^o--Taylor expansions.

\subsubsection{Controlling the variance of the scheme}

If we use the notation set up in Section~\ref
{subsectionwelldefinednessandlocalestimate}, the approximation
(\ref{eqZ}) can be written out as
$ Z_{i\hspace*{-0.7pt}} = \bE_i [ \Delta W_{\hspace*{-0.5pt} i+1} A_\ip]/h$.
We point out that implementation-wise it is better to use the
lower variance approximation (\ref{eqvariancereductionZ}) instead of
(\ref{eqZ}), that is, to
use
\begin{eqnarray*}
&&Z_i = \frac{1}{h}\bE_i \bigl[ \Delta
W_{i+1} \bigl( A_\ip- \bE_i[A_\ip]
\bigr) \bigr],\qquad i=0,\ldots,N-1.
\end{eqnarray*}
This does not lead to a relevant additional computation effort, as
$\bE_i[A_\ip]$ must
be computed for the estimation of the $Y_i$ component. To avoid a long
analysis, we make some simplifying assumptions in order to better
explain the
gain: assume $X_t=x+W_t$ and that we are about to compute $Z_0$ (a standard
expectation);
assume further (via Doob--Dynkin lemma) that $A_1$ can be written
as\footnote{If the reader is aware of
how conditional expectations in the BSDE framework are calculated, say,
for example,
via projection over a basis of functions, having a function $\varphi$ is
expected.}
$A_1=\varphi(X_1)=\varphi(x+\Delta W_1)$ where $\varphi$ has some
regularity so that
\[
\varphi(x+\Delta W_1) = \varphi(x) + \varphi'(x) (
\Delta W_1) + \tfrac{1}{2}\varphi''
\bigl(x^{*}\bigr) (\Delta W_1)^2, %
\]
where $x^*$ lies between $x$ and $x+\Delta W_1$. Then the Monte Carlo (MC)
estimator for
$Z_0$ from (\ref{eqZ}), with $M$ samples of the
normal $\cN(0,1)$ distribution given by
$\{\cN^\lambda\}_{\lambda=1, \ldots,M}$, and its standard deviation
(St.d.) are
\[
Z^{\mathrm{MC},\mbox{\fontsize{8.36}{8.36}\selectfont{(\ref{eqZ})}}}_0 = \frac{1}{M }\sum
_{\lambda=1}^M \frac{\sqrt{h}\cN^\lambda}{h} \varphi\bigl( x+
\sqrt{h}\cN^\lambda\bigr) \qquad\mbox{with } \Std\approx
\frac{\llvert \varphi(x)\rrvert }{\sqrt{h}\sqrt{M}}.
\]
Using (\ref{eqvariancereductionZ}) instead of (\ref{eqZ}) to
compute $Z_0$
would
produce the MC estimator and its St.d.
\[
Z^{\mathrm{MC},\mbox{\fontsize{8.36}{8.36}\selectfont{(\ref{eqvariancereductionZ})}}}_0 = \frac{1}{M }\sum
_{\lambda=1}^M \frac{\sqrt{h}\cN^\lambda}{h} \bigl( \varphi\bigl(
x+\sqrt{h}\cN^\lambda\bigr) - \varphi(x) \bigr) \qquad\mbox{with } \Std
\approx\frac{ \llvert \varphi'(x)\rrvert }{\sqrt{M}}.
\]
Compare now the standard deviation of both estimators. It is crucial
for the
stability
that the denominator of the variance of
$Z^{\mathrm{MC},\mbox{\fontsize{8.36}{8.36}\selectfont{(\ref{eqvariancereductionZ})}}}_0$
lacks that $\sqrt{h}$ term.
If $M$ is
kept fixed
then as $h$ gets smaller we expect $Z^{\mathrm{MC},\mbox{\fontsize{8.36}{8.36}\selectfont{(\ref{eqZ})}}}_0$
to blow up
while
$Z^{\mathrm{MC},\mbox{\fontsize{8.36}{8.36}\selectfont{(\ref{eqvariancereductionZ})}}}_0$ will remain controlled
(assuming $\varphi$
can be controlled\footnote{In \citet{GobetTurkedjiev2011}, it is shown
for the
locally Lipschitz driver case that $\varphi$ is indeed a Lipschitz
function of its variables.}). This can be numerically confirmed in
\citet{AlankoAvellaneda2013}.

We point out that this simple trick can be adapted to the scheme
proposed in
the next section as well as to the computation of the second-order scheme
proposed previously.

%
\section{Convergence of the tamed explicit scheme}\label{sectionexplicitscheme}

We now turn our attention back to the explicit scheme.
Unlike the case $\theta\in[1/2,1]$, when $\theta< 1/2$, the
local estimates of Proposition~\ref{thwellposed} cannot be extended
to the
global ones (as in Proposition~\ref{thsizebound}). Consequently, we
also do
not have a control over the stability remainder $\cR^{\cS}(H)$ (see Definition
\ref{defstab}).
In fact, as the motivating example of the \hyperref[sec1]{Introduction} shows, the
scheme can
explode.
To remedy to this, we consider the tamed explicit scheme, described in
(\ref{eqY-tamedexplicit})--(\ref{eqZ-tamedexplicit}),
which in turn corresponds to a truncation procedure applied to the original
BSDE,
and show that this scheme converges.
Our analysis yields as a by-product sufficient conditions under which the
naive explicit scheme converges (see Remark
\ref{remarkwhentheexplicitschemeconverges}).

%
\begin{remark}[($m>1$)]
In this section, we focus exclusively on the case $m>1$ in assumption~\textup{(HY0)}.
The easier case $m=1$ does
not require taming and stability of the scheme results from a straightforward
adaptation of the proof of Proposition~\ref{propstabilityfortheexplicitscheme}.
\end{remark}

\subsection{Principle}

The idea is that with the truncation functions $T_{L_h}$ and $T_{K_h}$ [recall
the scheme
(\ref{eqY-tamedexplicit})--(\ref{eqZ-tamedexplicit})], one cannot
only obtain uniform integrability bounds for the scheme, but also a
pathwise
bound, ensuring that the output
$\{Y_i\}_{i=0, \ldots, N}$ stays under a certain threshold,
under which the scheme is found to be stable in the sense of
(\ref{eqdef-of-stability-sec4}) with $H_i=0$.

Note that this tamed scheme is not exactly the scheme
(\ref{eqY})--(\ref{eqZ}) with $\theta=0$. However, it can be seen
as the case $\theta=0$ with the functions $T_{L_h}\circ g$ and
$f(\cdot,T_{K_h}(\cdot),\cdot,\cdot)$ instead of $g$ and $f$.
They satisfy the same properties with the same constants,
so we can reuse the results of Section~\ref{sectionnumericaldiscretization}.

Because the scheme is controlled, we naturally compare first its output
$\{(Y_i,Z_i)\}_{i \in\{0,\ldots,N\} }$ to $(Y_\ti', \bar{Z}_\ti
')_{\ti\in
\pi}$, where
$(Y_t',Z_t')_{t\in[0,T]}$ is the solution to the BSDE
(\ref{canonicFBSDE}) with controlled coefficients, for $t\in[0,T]$
%
\begin{equation}
\label{eqsec6BSDEprime} Y'_t = T_{L_h}
\bigl(g(X_T) \bigr) + \int_t^T f
\bigl(u,T_{K_h}(X_u),Y'_u,Z'_u
\bigr) \,\udu- \int_t^T Z'_u
\,\ud W_u.
\end{equation}
This part of analysis follows the methodology used above.

In a second step, it is enough to estimate the distance between the
solution $(Y',Z')$ of the truncated BSDE (\ref{eqsec6BSDEprime}) and the
solution $(Y,Z)$ of the original BSDE~(\ref{canonicFBSDE}) in order to
conclude
to the convergence of the scheme.

In line with Sections~\ref{sectionnumericaldiscretization} and
\ref{sectionthetaimplicitschemes}, we define $\{\bar
Z_\ti'\}_{\ti\in\pi}$ as in (\ref{Z-bar-ti-pi}),
$\widehat{Y}_i=Y_{i,(Y'_\ip,\bar{Z}_\ip')}$ and
$\widehat{Z}_i=Z_{i,(Y'_\ip,\bar{Z}_\ip')}$ for $i=0, \ldots,
N-1$, more
precisely
%
\begin{eqnarray}
\label{eqdefinitionofhatYtamed} \widehat{Y}_i &:=& \bE_{{i}} \bigl[
Y'_{\tip} + f \bigl( t_\ip,T_{K_h}(X_\ip),Y'_{\tip},
\bar{Z}_{\tip}' \bigr) h \bigr],
\\
\label{eqdefinitionofhatZtamed} \widehat{Z}_i &:=& \bE_{{i}} \biggl[
\frac{\Delta W_{{i+1}} }{h} \bigl( Y'_{t_{i+1}} + f_h
\bigl( t_\ip,X_{\ip}, Y'_{\tip},
\bar{Z}_{\tip}' \bigr) h \bigr) \biggr].
\end{eqnarray}

\subsection{Integrability for the scheme}

We now show that the tamed Euler scheme has the property that
$\llvert Y_i\rrvert \le h^{-1/(2m-2)}
$ for all $i \in\{0,\ldots,N\}$. This is already true
for $Y_N =T_{L_h} (g(X_N) )$ by construction. In the next two
propositions, we will show that this bound propagates through time.

%
\begin{proposition} \label{propexplicitone-stepsizeestimate}
Assume \textup{(HX0)}, \textup{(HY0)} and that $h \le{1}/({32d L_z^2})$.
If for a given $i \in\{ 0, \ldots,N-1 \}$ one has
$\llvert Y_\ip\rrvert \le h^{-1/(2m-2)}$, then one also has
\begin{eqnarray*}
\llvert Y_i\rrvert^2 + \frac{1}{d} \llvert
Z_{i}\rrvert^2h
&\le& ( 1+ c_1 h ) \bE_i \biggl[ \llvert
Y_{i+1} \rrvert^2 + \frac{1}{4d}\llvert
Z_{i+1}\rrvert^2 h \biggr]
\\
&&{} + c_2 h +
c_2 h \bE_i \bigl[ \bigl\llvert T_{K_h}(X_\ip)
\bigr\rrvert^2 \bigr].
\end{eqnarray*}
\end{proposition}

\begin{pf}
Take $i \in\{ 0, \ldots,N-1 \}$. We have seen in the proof of Proposition~\ref{thwellposed}, equation
(\ref{eqwellposep2}) that,
since $\theta= 0$,
\begin{eqnarray*}
\llvert Y_i\rrvert^2 + \frac{1}{d} \llvert
Z_{i}\rrvert^2 h & \le&\bigl( 1+2 \bigl(L_{y}+
\alpha'\bigr) h \bigr) \bE_i \bigl[ \llvert
Y_{i+1} \rrvert^2 \bigr]
\\
&&{}+ \frac{3 L_{z}^2}{2\alpha'} \bE_i \bigl[ \llvert Z_{i+1}
\rrvert^2 \bigr] h + \bE_i\bigl[B\bigl(i+1,\alpha
'\bigr)\bigr] + H^0_i,
\end{eqnarray*}
where $B(i+1,\alpha'): =
( {3 L^2} h + {3 L_x^2} \llvert T_{K_h}(X_\ip)\rrvert^2
h)/{2 \alpha'}$
and
\[
H^{0}_i = \bE_i \bigl[ \llvert
f_\ip\rrvert^2 \bigr] h^2 =
\bE_i \bigl[ \bigl\llvert f\bigl(\tip,T_{K_N}(X_\ip),Y_\ip,Z_\ip
\bigr)\bigr\rrvert^2 \bigr] h^2.
\]
Using \textup{(HY0)} and the fact that $\llvert Y_\ip\rrvert
^{2(m-1)}h \le1$, we have
\begin{eqnarray*}
\llvert f_\ip\rrvert^2 h^2 & \le&4
L^2 h^2 + 4 L_x^2 \bigl\llvert
T_{K_h}(X_\ip)\bigr\rrvert^2 h^2
\\
&&{}+ 4 L_y^2 \bigl[\llvert Y_\ip
\rrvert^{2(m-1)}h \bigr] \llvert Y_\ip\rrvert^2 h
+ 4 L_z^2 \llvert Z_\ip\rrvert
^2 h^2
\\
& \le&4 L^2 h^2 + 4 L_x^2 \bigl
\llvert T_{K_h}(X_\ip)\bigr\rrvert^2
h^2 + 4 L_y^2 \llvert Y_\ip
\rrvert^2 h + 4L_z^2h \llvert
Z_\ip\rrvert^2 h,
\end{eqnarray*}
so we have in the end
\begin{eqnarray*}
\llvert Y_i\rrvert^2 + \frac{1}{d} \llvert
Z_{i}\rrvert^2 h & \le&\bigl( 1+2 \bigl(L_{y}+
\alpha' + 2L_y^2 \bigr) h \bigr)
\bE_i \bigl[ \llvert Y_{i+1} \rrvert^2 \bigr]
\\
&&{}+ \biggl( \frac{3 L_{z}^2}{2\alpha'} + 4L_z^2h \biggr)
\bE_i \bigl[ \llvert Z_{i+1}\rrvert^2 \bigr]
h
\\
&&{}+ \biggl( \frac{3 L^2}{2 \alpha' } + 4L^2 h \biggr) h + \biggl(
\frac{3 L_x^2}{2 \alpha'} + 4 L_x^2 h \biggr) \bE_i
\bigl[\bigl\llvert T_{K_h}(X_\ip)\bigr\rrvert
^2 \bigr] h.
\end{eqnarray*}
Choose now $\alpha' =12 d L_z^2$ [so that $3L_{z}^2/(2\alpha')
\le
{1}/({8d})$] and combine with the restriction $h\leq1/(32 d L_z^2)$
(so that $4L_z^2h \le\frac{1}{8d}$).
Taking $c_1 = 2 (L_{y}+ 12 d L_z^2 + 2L_y^2 )$ and
\[
c_2 = \max\biggl\{\frac{3 L^2}{24 d L_z^2 } + \frac{4L^2}{32d
L_z^2},
\frac{3 L_x^2}{24 d L_z^2 } + \frac{4 L_x^2
}{32d L_z^2} \biggr\} = \max\biggl\{\frac{L^2}{4 d L_z^2 }, \frac{L_x^2}{4 d L_z^2 } \biggr\},
\]
and noting that ${1}/({4d}) \le(1 + c_1 h)/(4d)$,
we find the required estimate
\begin{eqnarray*}
\llvert Y_i\rrvert^2 + \frac{1}{d} \llvert
Z_{i}\rrvert^2 h
&\le& ( 1+c_1 h )
\bE_i \biggl[ \llvert Y_{i+1} \rrvert^2
+ \frac{1}{4d} \llvert Z_{i+1}\rrvert^2 h
\biggr]
\\
&&{} + c_2 h + c_2 h \bE_i \bigl[ \bigl
\llvert T_{K_h}(X_\ip)\bigr\rrvert^2 \bigr].
\end{eqnarray*}\upqed
\end{pf}

We can then use this local bound to obtain the following pathwise bound.

%
\begin{proposition} \label{propositionsizeestimateexplicitscheme}
Let \textup{(HX0)} and \textup{(HY0)} hold. For any $i\in\{0,
\ldots,N-1\}$,
\begin{eqnarray*}
&& \llvert Y_i\rrvert^2 + \frac{1}{4d} \llvert
Z_i\rrvert^2 h + \frac{3}{4d} \bE_i
\Biggl[ \sum_{j=i}^{N-1} \llvert
Z_{j}\rrvert^2 h \Biggr]
\\
&& \qquad\le e^{c_1 (N-i)h} \bE_i \bigl[ \llvert Y_N
\rrvert^2 \bigr] + e^{c_1 (N-1-i)h} \Biggl( \sum
_{j=i}^{N-1} c_2 h + c_2 h
\bE_i \bigl[\bigl\llvert T_{K_h}(X_\ip)\bigr
\rrvert^2 \bigr] \Biggr).
\end{eqnarray*}
This implies in particular that $\llvert Y_i\rrvert \le
h^{-1/(2m-2)}
$.
\end{proposition}

\begin{pf}
The proof goes by induction. The case $i=N$ is clear. If the estimate
is true
for $i+1$, noting that
$\llvert Y_N\rrvert \le L_h$, $\llvert T_{K_h}(x)\rrvert \le K_h$ and
$e^{c_1 T} ( L_h^2 + c_2 T + c_2 T K_h^2 )
\le
h^{-1/(m-1)}
$,
we see that $\llvert Y_\ip\rrvert^2 \le h^{-1/(m-1)}
$.
Then, combining the estimate of Proposition
\ref{propexplicitone-stepsizeestimate}
and the estimate for $i+1$ (from the induction assumption),
in the same way as in Lemma~\ref{lemGron}, we obtain the desired
estimate for
$i$.
\end{pf}

In view of the previous bound, we can derive a similar estimate for the
solution $(Y',Z')$ to (\ref{eqsec6BSDEprime}). Namely, using
(\ref{eqpathwise-estim-for-polyy-BSDE}) with $\alpha=
12dL_z^2$ and combining it further with \textup{(HY0)}, we have
\begin{eqnarray*}
\bigl\llvert Y'_t\bigr\rrvert^2 & \le&
e^{2(L_y + 12dL_z^2) (T-t)}
\\
&&{}\times \bE_t \biggl[ \bigl\llvert T_{L_h}
\bigl(g(X_T)\bigr) \bigr\rrvert^2 + \int
_t^T \frac{1}{16d L_z^2} \bigl\llvert f
\bigl(u,T_{K_h}(X_u),0,0\bigr)\bigr\rrvert^2\,
\udu\biggr]
\\
& \le& e^{c_1 (T-t)} \bE_t \biggl[ \bigl\llvert T_{L_h}
\bigl(g(X_T)\bigr) \bigr\rrvert^2 + \int
_t^T \frac{1}{8dL_z^2} \bigl( L^2
+ L_x^2 \bigl\llvert T_{K_h}(X_u)
\bigr\rrvert^2 \bigr) \,\udu\biggr]
\\
& \le& e^{c_1 T} \bigl(L_h^2 + c_2T
+c_2 T K_h^2 \bigr)
\\
& \le&\biggl(
\frac{1}{h} \biggr)^{1/(m-1)},
\end{eqnarray*}
implying in particular that
$\llvert Y'_\ti\rrvert
\le
h^{-1/(2m-2)}
$
for all $i$.

These two estimates, ensuring that both $Y_i$ and $Y'_\ti$ are
bounded by $h^{-1/(2m-2)}$
will be useful in the analysis of
the global error, since the explicit scheme is found to be stable
under this threshold.

\subsection{Stability of the scheme}

As previously, for any $i\in\{0,\ldots,N-1\}$ we use the notation
$\delta Y_\ip:= Y'_\tip- Y_\ip$ and
$\delta Z_\ip:= \bar{Z}_\tip' - Z_\ip$, as well\vspace*{1pt} as $\delta A_\ip:= \delta
Y_\ip+ \delta f_\ip h$ where $\delta f_\ip$ is given by
\begin{eqnarray*}
\delta f_\ip&:=& f \bigl(\tip, T_{K_h}(X_\ip),Y_\tip',
\bar{Z}_\ip' \bigr) - f \bigl(\tip, T_{K_h}(X_\ip),Y_\ip,Z_\ip
\bigr).
\end{eqnarray*}
Then, denoting $\widehat{\delta Y}_i = \widehat{Y}_i -Y_i$ and
$\widehat{\delta Z}_i = \widehat{Z}_i -Z_i$,
we can write
\begin{eqnarray*}
&&\widehat{\delta Y}_i = \bE_i [ \delta
A_\ip] \quad\mbox{and}\quad\widehat{\delta Z}_i =
\bE_i \biggl[\frac{1}h {\Delta W_\ip} \delta
A_\ip\biggr].
\end{eqnarray*}
We now proceed to show that, because the two inputs satisfy $\llvert
Y_\ip\rrvert,
\llvert Y'_\tip\rrvert \le h^{-1/(2m-2)}$,
the scheme is
stable in the sense that we can obtain the estimate
(\ref{eqdef-of-stability-sec4}) with $H_i = 0$.

%
\begin{proposition} \label{propstabilityfortheexplicitscheme}
Assume \textup{(HX0)} and \textup{(HY0$_{\mathrm{loc}}$)}. Then there exists a
constant $c$ for any $h \le\min\{1,1/32d L_z^2 \}$, such that for $i\in\{0,\ldots,N-1\}$
\begin{eqnarray}
&&\llvert\widehat{\delta Y}_i\rrvert^2 +
\frac{1}{d} \llvert\widehat{\delta Z}_i\rrvert
^2h \le( 1+ c h ) \bE_i \biggl[ \llvert\delta
Y_{i+1} \rrvert^2 + \frac{1}{4d}\llvert\delta
Z_{i+1}\rrvert^2 h \biggr].
\end{eqnarray}
\end{proposition}

\begin{pf}
Let $i\in\{0,\ldots,N-1\}$. Just like for Proposition~\ref{thstability},
the
proof mimics the computations of the proof of Proposition
\ref{thwellposed} with only a small adjustment for the constants.
However, a
different argumentation for the term $H^0_i = \llvert\delta f_\ip
\rrvert^2 h^2$ is
required. Using (HY0$_{\mathrm{loc}}$), $h \le1$ and the bounds
$\llvert Y'_\tip\rrvert^{2(m-1)} h$,
$\llvert Y'_\tip\rrvert^{2(m-1)} h
\le1$, we have
\begin{eqnarray*}
\llvert\delta f_\ip\rrvert^2 h^2 & \le&2L_y^2 \bigl( 1+\bigl\llvert Y'_\tip
\bigr\rrvert^{2(m-1)} + \llvert Y_\ip\rrvert^{2(m-1)}
\bigr) \bigl\llvert Y'_\tip- Y_\ip\bigr
\rrvert^2 h^2
\\
&&{}+ 2L_z^2 \bigl\llvert\bar{Z}_\tip'
- Z_\ip\bigr\rrvert^2 h^2
\\
&=& 2L_y^2 \bigl( h +\bigl\llvert Y'_\tip
\bigr\rrvert^{2(m-1)}h + \llvert Y_\ip\rrvert
^{2(m-1)}h \bigr) h \bigl\llvert Y'_\tip-
Y_\ip\bigr\rrvert^2
\\
&&{}+ 2L_z^2 h \bigl\llvert\bar{Z}_\tip'
- Z_\ip\bigr\rrvert^2 h
\\
& \le&6L_y^2 h \llvert\delta Y_\ip\rrvert
^2 + 2L_z^2 h \llvert\delta
Z_\ip\rrvert^2 h.
\end{eqnarray*}
The rest follows as in the proof of Proposition~\ref{thwellposed}.
\end{pf}

\subsection{Convergence of the scheme}

The convergence of the scheme is\break achieved by controlling both the (squared)
error committed
by the truncation procedure, $\llVert Y-Y'\rrVert_{\cS
^2}^2+\llVert Z-Z'\rrVert_{\cH^2}^2$,
as a
function of the time step, and by controlling the numerical approximation
(\ref{eqY-tamedexplicit})--(\ref{eqZ-tamedexplicit}) of
the solution $(Y',Z')$ to (\ref{eqsec6BSDEprime}).

\subsubsection*{Distance between $(Y_i,Z_i)_i$ and $(Y'_{t_{i}},\bar{Z}_{t_{i}}')_i$}

We estimate this distance by using the Fundamental Lemma
\ref{lemfundamental}.

The tamed scheme (\ref{eqY-tamedexplicit})--(\ref
{eqZ-tamedexplicit}) is
the $\theta=0$
scheme (\ref{eqY})--(\ref{eqZ}) with the coefficient
$f(\cdot,\cdot,T_{K_h}(\cdot),\cdot)$
and terminal condition $T_{L_h}\circ g$ having the same Lipschitz
constant as
$f$ and $g$.
So the results of Section~\ref{sectionnumericaldiscretization} apply.
In particular, Lemma~\ref{lemmaerroronterminalconditons} controls
the error on the terminal condition.

Similarly, Lemmas~\ref{lemerrorR}~and~\ref{theocontrollingIntR} are
still valid with the same constants. The only difference is that the
path-regularity involved is now that of $(Y',Z')$,
but since $T_{L_h}\circ g$ is still Lipschitz, Theorem
\ref{thregularity} indeed applies to $(Y',Z')$.
So Proposition~\ref{propglobaldiscretizationerror} applies,
to control the sum of the one-step discretization errors.

Finally, we have just proven with Proposition~\ref
{propstabilityfortheexplicitscheme}
that the scheme is stable with $H_i^0 = 0$, so $\cR^\cS(H)=0$.
We can therefore conclude via
Lemma~\ref{lemfundamental} that
%
\begin{eqnarray}
\label{eqestimateforconvergenceexplicitscheme1} && \max_{i=0,
\ldots, N} \bE\bigl[\bigl\llvert
Y_\ti' - Y_i \bigr\rrvert^2
\bigr] + \sum_{i=0}^{N-1} \bE\bigl[\bigl
\llvert\bar{Z}_\ti' - Z_i \bigr\rrvert
^2\bigr] h\nonumber
\\
&&\qquad  \le c \bigl(\bE\bigl[\bigl\llvert Y_\tN' -
Y_N \bigr\rrvert^2\bigr] + \bE\bigl[\bigl\llvert
\bar{Z}_\tN' - Z_N\bigr\rrvert
^2\bigr]h \bigr)
\nonumber\\[-8pt]\\[-8pt]\nonumber
&&\quad\qquad{}  + c \sum_{i=0}^{N-1}
\biggl( \frac{1}{h} \tau_i(Y) + \tau_i(Z)
\biggr) +0
\\
\nonumber
&&\qquad  \le ch.
\end{eqnarray}
We remark that the thresholds $L_h$ and $K_h$ have no effect in this
estimation.

\subsubsection*{The distance between $(Y_{t_{i}}',\bar{Z}_{t_{i}}')_i$ and $(Y_{t_{i}},\bar{Z}_{t_{i}})_i$}

We now estimate the distance between
$(Y'_\ti,\bar{Z}_\ti')_i$ and $(Y_\ti,\bar{Z}_\ti)_i$, that is, between
(\ref{eqsec6BSDEprime}) and (\ref{canonicFBSDE}), which gathers all the
error induced by the taming. In order to estimate this
error, we need to have an estimation of the $L^2$-distance between
$X_u$ and
$T_{K_h}(X_u)$ on the one hand, and $g(X_T)$ and $T_{L_h}
(g(X_T) )$ on
the other.
We give a general estimation for this below.

\begin{proposition} \label{propestimationofeffectoftamingaRV}
Let $\xi$ be a random variable in $L^q$ for some $q > 2$, and $L>0$.
Then we have
\[
\bE\bigl[ \bigl\llvert\xi-T_{L}(\xi)\bigr\rrvert^2
\bigr] \le4 \bE\bigl[ \llvert\xi\rrvert^q \bigr] \biggl(
\frac{1}{L} \biggr)^{q-2}.
\]
\end{proposition}

\begin{pf}
Using the facts that $T_L(x)=x$ for $\llvert x\rrvert\le L$
and that
$\llvert T_L(\xi)\rrvert\le\llvert\xi\rrvert$,
together with the H\"{o}lder and the Markov
inequalities,
we have
\begin{eqnarray*}
\bE\bigl[\bigl\llvert\xi-T_{L}(\xi)\bigr\rrvert^2
\bigr] &=& \bE\bigl[\bigl\llvert\xi-T_{L}(\xi)\bigr\rrvert
^2 \1_{\{ \llvert
\xi\rrvert\ge L \}}\bigr]
\\
&\le& 4 \bE\bigl[\llvert\xi\rrvert
^2 \1_{\{\llvert\xi\rrvert \ge L\}}\bigr]
\\
& \le&4 \bE\bigl[\llvert\xi\rrvert^q\bigr]^{2/q} \bP
\bigl[\llvert\xi\rrvert\ge L\bigr]^{1-2/q}
\\
&\le& 4 \bE\bigl
[\llvert\xi
\rrvert^q \bigr]^{2/q} \biggl( \frac{\bE[\llvert\xi\rrvert^q]}{L^q}
\biggr)^{1-2/q}
\\
&=& 4 \bE\bigl[\llvert\xi\rrvert^q \bigr] \biggl(
\frac{1}{L} \biggr)^{q(1-2/q)}.
\end{eqnarray*}\upqed
\end{pf}
Now, via Jensen's inequality we have
\begin{eqnarray*}
\bigl\llvert\bar{Z}_\ti- \bar{Z}_\ti'
\bigr\rrvert^2 h &=& \biggl\llvert\frac{1}{h}
\bE_i \biggl[ \int_\ti^\tip
Z_u \,\udu\biggr] - \frac{1}{h}\bE_i \biggl[ \int
_\ti^\tip Z'_u \,\udu
\biggr] \biggr\rrvert^2 h
\\
& \le&\bE_i \biggl[ \int_\ti^\tip
\bigl\llvert Z_u - Z'_u\bigr\rrvert
^2 \,\udu\biggr],
\end{eqnarray*}
from which it clearly follows that
\begin{eqnarray*}
&& \max_{i=0, \ldots, N} \bE\bigl[\bigl\llvert Y_\ti-
Y'_\ti\bigr\rrvert^2\bigr] + \sum
_{i=0}^{N-1} \bE\bigl[\bigl\llvert
\bar{Z}_\ti- \bar{Z}_\ti' \bigr\rrvert
^2\bigr] h
\\
&&\qquad  \le\sup_{t \in[0,T]} \bE\bigl[ \bigl\llvert
Y_t- Y'_t \bigr\rrvert^2
\bigr] + \bE\biggl[ \int_0^T \bigl\llvert
Z_u - Z'_u\bigr\rrvert^2\,
\udu\biggr].
\end{eqnarray*}
From the a priori estimate (\ref{eqapriori-for-diff-of-polyy-BSDE}),
we have
\begin{eqnarray*}
&&\sup_{t \in[0,T]} \bE\bigl[\bigl\llvert Y_t-
Y'_t \bigr\rrvert^2\bigr] + \bE\biggl[
\int_0^T \bigl\llvert Z_u -
Z'_u\bigr\rrvert^2\,\udu\biggr]
\\
&&\qquad \le c \biggl( \bE\bigl[\bigl\llvert g(X_T) - T_{L_h}
\bigl(g(X_T)\bigr) \bigr\rrvert^2\bigr]
\\
&&\hspace*{9pt}\quad\qquad{} + \bE\biggl[\int_0^T \bigl
\llvert f \bigl(u,X_u,Y'_u,Z'_u
\bigr) - f \bigl(u,T_{K_h}(X_u),Y'_u,Z'_u
\bigr) \bigr\rrvert^2 \,\udu\biggr] \biggr)
\\
&&\qquad \le c \biggl( \bE\bigl[\bigl\llvert g(X_T) - T_{L_h}
\bigl(g(X_T)\bigr) \bigr\rrvert^2\bigr] +
L_x^2 \int_0^T \bE
\bigl[ \bigl\llvert X_u -T_{K_h}(X_u) \bigr
\rrvert^2 \bigr] \,\udu\biggr)
\\
&&\qquad \le c \biggl( 4 \biggl( \frac{1}{L_h} \biggr)^{2m-2} \bE\bigl[
\bigl\llvert g(X_T)\bigr\rrvert^{2m} \bigr] + \biggl(
\frac{1 }{K_h} \biggr)^{2m-2}4 L_x^2 \int
_0^T \bE\bigl[ \llvert X_u\rrvert
^{2m} \bigr] \,\udu\biggr),
\end{eqnarray*}
thanks to Proposition~\ref{propestimationofeffectoftamingaRV}. Now,
since $X \in\S^{2m}$ (Theorem~\ref{theo-existenceuniquenesscanonicalFBSDE}),
$g$ is of linear growth, and $L_h$ and $K_h$ are of order
$h^{-1/(2m-2)}$,
we can conclude that
%
\begin{eqnarray}
\label{eqestimateforconvergenceexplicitscheme2}&& \max_{i=0, \ldots, N} \bE\bigl[\bigl\llvert
Y_\ti- Y'_\ti\bigr\rrvert^2
\bigr] + \sum_{i=0}^{N-1} \bE\bigl[\bigl
\llvert\bar{Z}_\ti- \bar{Z}_\ti' \bigr
\rrvert^2\bigr] h \le c h.
\end{eqnarray}

\subsubsection*{The proof of the Theorem \texorpdfstring{\protect\ref{thconvergencerateforexplicitscheme}}{}}

By collecting the above results, we can now prove Theorem
\ref{thconvergencerateforexplicitscheme}.

\begin{pf*}{Proof of Theorem~\ref{thconvergencerateforexplicitscheme}}
To prove this theorem, that is, that $\mathrm{ERR}_\pi(Y,\break Z) \le
ch^{1/2}$ [see
(\ref{eqerrornorm-sec4})], we use the triangular inequality and dominate\break 
$\mathrm{ERR}_\pi(Y,Z)$ by the sum of:
(i) the distance between the solution $(Y,Z)$ to the
original BSDE (\ref{canonicFBSDE}) and
the solution $(Y',Z')$ to the truncated BSDE (\ref{eqsec6BSDEprime}),
and (ii)
the distance between $(Y_\ti',\bar{Z}_\ti')_{\ti\in\pi}$ and
the $\{(Y_i,Z_i)\}_{i \in\{0,\ldots,N\} }$
[from the scheme (\ref{eqY-tamedexplicit})--(\ref{eqZ-tamedexplicit})].
The estimate for the first difference is given by
(\ref{eqestimateforconvergenceexplicitscheme2}).
The estimate for the second is given by
(\ref{eqestimateforconvergenceexplicitscheme1}).
Hence, the result.
\end{pf*}

%
\begin{remark} \label{remarkwhentheexplicitschemeconverges}
We see from the proofs of Propositions
\ref{propexplicitone-stepsizeestimate} and
\ref{propositionsizeestimateexplicitscheme}
that if
$x \mapsto f(t,x,y,z)$ is bounded (say, by $K$)
uniformly in the other variables
and the terminal condition $g$ is bounded,
then the naive explicit scheme
[i.e., (\ref{eqY})--(\ref{eqZ}) with $\theta=0$] converges.
Under these conditions, it is suitable to use the explicit backward Euler
scheme.
\end{remark}
%
%
\section{Numerical experiments}\label{sectionnumerics}


We conclude with some numerical experiments for the convergence of
the introduced schemes. In this work, we are concerned only with the
time-discretization, but
in order to implement a scheme, we need to further
approximate the required conditional expectations.
For this, we use the method of regression on a basis
functions as in
\citet{GobetLemorWarin}, \citet{GobetTurkedjiev2011}.
Following \citet{GobetLemorWarin}, we work with (Hermite) polynomials
up to a
certain degree $K$.
Here, we do not aim at studying the effect of the number
$K$ of basis functions or the number $M$ of diffusion paths $\{X^m_i\}_{i=0, \ldots, N}^{m=1, \ldots, M}$.
Rather, we choose $K$ and $M$ big enough so that
(a) the variance of the results is small enough, and
(b) the effect of approximating the conditional expectation is
negligible and so
what we measure is indeed the effect on the time-discretization of the
time-step $h=T/N$.

In all the examples below, we fix terminal time $T$ and want to compute
an approximation of $u(t,X_t)=Y_{t}^{}=:Y_{t}^{\mathrm{true}}$.
Since in this section we use grids with different numbers $N$ of intervals,
we do not omit the superscripts and denote by $Y^{N}_i$
the scheme's approximation of $Y_{t_i}^{\mathrm{true}}$.
When the explicit solution to the FBSDE is known, we can
measure the error of the numerical approximation by estimating
$\mathrm{ERR}(Y^N)=\max_i \bE[ \llvert Y_{t_i}^{\mathrm{true}} -
Y^{N}_i\rrvert^2
]^{1/2}$.
When the explicit solution is not known, we can compute
%
\begin{eqnarray}
\label{eqerrorN} &&e(N):=\max_{i=0,\ldots,N}\bE\bigl[ \bigl\llvert
Y^{N}_i - Y^{2N}_{2i} \bigr\rrvert
^2 \bigr]^{1/2}.
\end{eqnarray}
By observing the convergence of
$e(N)$ we can measure the convergence rate of the scheme even when we
do not
know the true solution. Indeed,
assume that for constants $c$ and $\gamma$, for any $N$ and any
$i=0,\ldots,N$
we have
\begin{eqnarray*}
&&\bE\bigl[ \bigl\llvert Y^{N}_i - Y^{2N}_{2i}
\bigr\rrvert^2 \bigr]^{1/2} \le c \mathit{N}^{-\gamma},
\\
&&\qquad\Longrightarrow\quad \bE\bigl[ \bigl\llvert Y^{N}_i - Y^{\mathrm{true}}_{t_i}\bigr\rrvert^2 \bigr]^{1/2}
\le\sum_{k=0}^{\infty}c\bigl(2^{k}
\mathit{N}\bigr)^{-\gamma} =\frac{c
\mathit{N}^{-\gamma}}{1-(1/2)^{\gamma}} =c'
\mathit{N}^{-\gamma},
\end{eqnarray*}
given that the scheme converges.

We computed the approximation processes
$(Y^{N}_i)$ and $(Y^{2N}_i)$ using
the same sample of Brownian increments.
For each measurement, we launched the scheme 10 times and averaged the
results.


\subsection*{Example 1---Numerical approximation for Example \texorpdfstring{\protect\ref{exampleFH-N}}{}}

We consider the motivating FitzHugh--Nagumo PDE and the terminal
condition $g$
of
Example~\ref{exampleFH-N} with $a=-1$, for which $f(t,x,y,z)=-y^3+y$: a
cubic polynomial (without quadratic terms). To solve the implicit
equation [see (\ref{eqY})], we can use Cardano's formula to compute
the single real
root
of the polynomial equation.

We take $T=1$ and $x_0=3/2$. The solution to the PDE is given by
(\ref{exampleSOLtoPDE}). We compute the error for various values of
$N$, and
this for the explicit scheme
($\theta=0$, which converges in that case since $g$ is bounded---see Remark~\ref{remarkwhentheexplicitschemeconverges}),
the implicit scheme ($\theta=1$) and the trapezoidal scheme [$\theta=1/2$,
note that we are under the extra assumptions made in Theorem
\ref{thmainresult}(ii)].

%
\begin{figure}[b]

\includegraphics{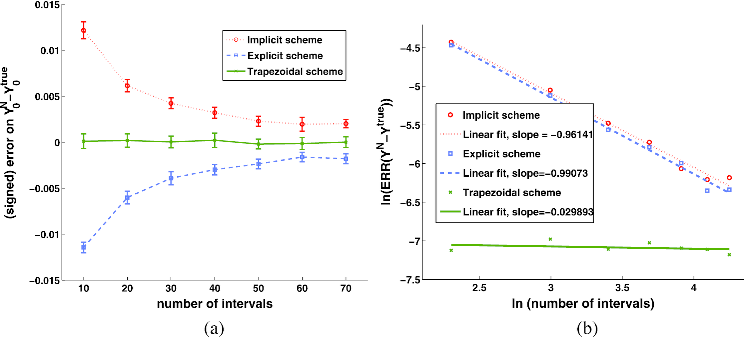}

\caption{\textup{(a)} Differences $Y^{N}_0-Y^{\mathrm{true}}_0$ for
each scheme
as functions on the number $N$ of time intervals.
\textup{(b)} Convergence rates obtained via linear fits on the log--log plots of
$\mathrm{ERR}(Y^N)$.
We used $N\in\{10,20,30,40,50,60,70\}$,
Hermite polynomials up to degree $K=7$,
$M=2\times10^5$ and $10$ simulations for each point.}
\label{figBM-error+conv-rates}\label{figBM-error+conv-rates-suba}\label{figBM-error+conv-rates-subb}
\end{figure}

In Figure~\ref{figBM-error+conv-rates-suba}(a), we see that the
implicit scheme overshoots the true solution while the explicit one
undershoots it; the trapezoidal scheme performs better in any
grid.
The convergence rates, as measured using $\mathrm{ERR}(Y^N)$, are presented in
Figure~\ref{figBM-error+conv-rates-subb}(b). For the trapezoidal
scheme, the error for
any $N$ is very small and the variance of the results is not negligible,
hence we are not able to measure the convergence rate as accurately.
The experimental rate seems to be lower than that of the explicit and implicit; see Table~\ref{rates-implicit-explicit-trap}.
We note, however, that the error is already much lower than those in
the other
schemes.

%
\begin{table}
\tabcolsep=0pt
\tablewidth=258pt
\caption{Estimated rates (value of slope) for the
experiment reported in Figure~\protect\ref{figBM-error+conv-rates}}\label{rates-implicit-explicit-trap}
\begin{tabular*}{\tablewidth}{@{\extracolsep{\fill}}@{}lcc@{}}
\hline
\textbf{Scheme} & \textbf{Rate via ERR$\bolds{(Y^N)}$} & \textbf{Rate via $\bolds{e(N)}$}
\\
\hline
Implicit & $-0.96141$ & $-1.00460$\\
Explicit & $-0.99073$ & $-0.98372$\\
Trapezoidal & $-0.02989$ & $-0.33775$\\
\hline
\end{tabular*}
\end{table}


Both the implicit and explicit schemes are found to converge with rate $1$.
This does not mean that the estimates in Theorems
\ref{thmainresult} and~\ref{thconvergencerateforexplicitscheme}
(or that in the Fundamental Lemma~\ref{lemfundamental}) are too conservative
in all generality, but is simply due to the particularity of the
equation studied.
On the one hand, the estimates of Theorems~\ref{thmainresult} and
\ref{thconvergencerateforexplicitscheme} rely on the estimate of
Proposition~\ref{propglobaldiscretizationerror} (on the local
discretization errors) and so
on the regularity of $b,\sigma,f$ and $g$.
We worked under the minimal assumption~(HY0$_{\mathrm{loc}}$) assuming no differentiability. Nonetheless, in this
example all
involved functions are smooth (leading to a smooth solution $u$ to the PDE)
and so this term ends up converging faster
(see also Remark~\ref{remarkhigherratewhenhigherregularityoff}).
On the other hand, the estimates of Theorems~\ref{thmainresult} and
\ref{thconvergencerateforexplicitscheme} also rely on the
estimate of
Lemma~\ref{lemmaerroronterminalconditons} (on the terminal condition
error) which again holds under the mere assumption \textup{(HX0)} for $b$ and
$\sigma$. But here $(X_t)$ is the Brownian motion an its approximation
$(X^N_i)$ is exact, instead of being only of order $\gamma=1/2$ in the
case of
Euler--Maruyama scheme.

As we could verify in our simulations, the computational time is the same
for all the schemes with $\theta> 0$, as expected. On the other hand, similar
to the case of ODEs and SDEs, the convergence rate for $\theta\in\,]1/2,1[$
is no better than for $\theta= 1$. However, the latter choice is more stable
[compare with the definition of $\cR^{\cS}(H)$ and $H^\theta_i$] while
$\theta=1/2$ provides the smallest error.
A more detailed comparison between the different implicit-dominating
schemes is left to a forthcoming work.

Finally, while we were able to compute $\mathrm{ERR}(Y^N)$ in this example, we also
computed $e(N)$. Since we approximated the solution using polynomials
up to
degree $K$, the full (implemented) scheme computes in fact an
approximated process $Y^{N,K}$. As $N \rightarrow+\infty$, this does not
strictly converge to $Y^{\mathrm{true}}$ but rather to some~$Y^K$. The
convergence of $e(N)$ therefore better captures the convergence
of $Y^{N,K}$ to its limit, $Y^K$ and, therefore, yields slightly
different rates.

\subsection*{Example 2---Unbounded terminal condition}

To emphasize the contribution of this work, we analyze in more detail the
unbounded terminal condition case for which one needs to take either
the implicit scheme or the explicit scheme with truncated
terminal condition.
More precisely, we take $g(x)=x$, together with the driver $f(y)=-y^3$. For
the forward process, we take the geometric
Brownian motion with $b(x)=x/2$ and $\sigma(x)= x/2$, started
at $x_0=2$. We choose $T=1$.

%
\begin{figure}

\includegraphics{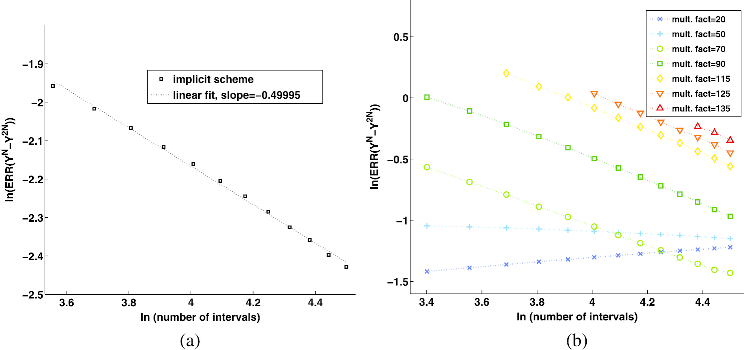}

\caption{\textup{(a)}~Convergence of $e(N)$ for the implicit scheme.
\textup{(b)} Convergence of $e(N)$ for the tamed explicit scheme and various values
of the multiplying factor.
In both cases, we used $N \in\{5i\dvtx i=7,\ldots,18\}$,
$K=4$, $M=10^5$ and $10$ simulations for each point. The results are plotted
in log--log scale.}\label{figpicture2}\label{figpicture21-unbdd-implicit}\label{figpicture22-unbdd-tamed-explicit}
\end{figure}

Figure~\ref{figpicture21-unbdd-implicit}(a) shows the convergence of $e(N)$
[see (\ref{eqerrorN})] for the implicit scheme, while Figure~\ref
{figpicture22-unbdd-tamed-explicit}(b)
shows the same computations for the truncated explicit scheme.
The implicit scheme converges with the rate $1/2$, as expected.
Concerning the
truncated explicit scheme [Figure~\ref{figpicture22-unbdd-tamed-explicit}(b)] we observed through several trials
that its behavior is quite sensitive to the truncation level $L_h$ (defined
in Section~\ref{subsec-tamed-euler}).\footnote{This echoes the
findings of
\citet{ChassagneuxRichou2013}.}
Our asymptotic, theoretical results [see
(\ref{eqY-tamedexplicit}), (\ref{eqZ-tamedexplicit}) and
Theorem~\ref{thconvergencerateforexplicitscheme}] suggest taking
for this particular example $L_h$ as
\[
L_h = \frac{1}{\sqrt{3}} e^{-(1/2)6 T} \biggl(
\frac{1}{h} \biggr)^{1/4}.
\]
We found, however, that this seems to be too conservative for practical
simulations. To better understand the impact of truncation, we
introduced a
multiplying factor $\alpha>0$ and truncate at the level $\alpha L_h$ instead
of $L_h$. In Figure~\ref{figpicture22-unbdd-tamed-explicit}(b) and
Table~\ref{tab-multfactandrates}, we sum up our findings. In
Table~\ref{tab-multfactandrates},
one sees the various multiplying factors and the corresponding estimated
rates [for the sequence $e(N)$ defined in (\ref{eqerrorN})].

%
\begin{table}
\tabcolsep=0pt
\caption{Estimated rate for the truncated
explicit scheme at truncation level $\alpha L_h$}\label{tab-multfactandrates}
\begin{tabular*}{\tablewidth}{@{\extracolsep{\fill}}@{}lccccccc@{}}
\hline
\textbf{Mult. factor} $\bolds{\alpha}$ & \textbf{20} & \textbf{50} & \textbf{70} & \textbf{90} & \textbf{115} & \textbf{125} & \textbf{135}
\\
\hline
Rate 
& 0.179 & $-$0.096& $-$0.801 & $-$0.896 & $-$0.929 & $-$0.970& $-$0.955\\
\hline
\end{tabular*}
\end{table}

%
\begin{figure}[b]

\includegraphics{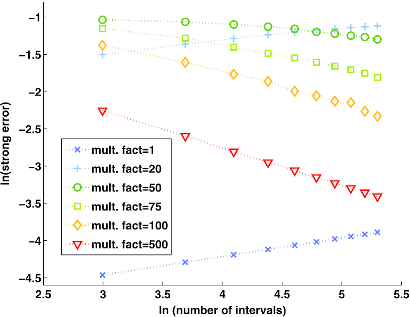}

\caption{Convergence of the error $\bE[ \llvert
T_{L_{1/N}}(g(X^{N}_{N}) )-
T_{L_{1/(2N)}}(g(X^{2N}_{2N})) \rrvert ^2
]^{1/2}$ on the
terminal condition,
computed for $N \in\{20i\dvtx i=1,\ldots,10\}$.
Plot in log--log scale with different levels of truncation
$L_{1/N}=\alpha L_h$,
done with
$M=10^5$ and $10$ simulations for each point.
The estimated slopes are, for the
corresponding multiplicative factors:
$0.25,0.17,-0.12,-0.29,-0.41,-0.50$ (reading the legend from
top to bottom).}\label{figpic3}
\end{figure}

By looking at Figure~\ref{figpicture22-unbdd-tamed-explicit}(b), we see that
the situation is complex and a separate argumentation is required for
``small'' and ``big'' multiplying factors. For $\alpha$ too small (up to
$40$), the scheme does not seem to converge. This is due to the fact
that a significant number of forward paths fall beyond truncation
levels $\alpha
L_{1/N}$ and $\alpha L_{1/{(2N)}}$.
Consequently, the strong convergence property for the forward approximation
does not guarantee that the quantity
$\bE[ \llvert T_{L_{1/N}}(g(X^{N}_T) )-
T_{L_{1/(2N)}}(g(X^{2N}_{T})) \rrvert ^2 ]^{1/2}$
decays\vspace*{1pt} with the rate
$1/2$, as is shown in Figure~\ref{figpic3}.
This lack of ``good convergence'' at the terminal time then translates
into a deterioration of the convergence rate for
the BSDE part of the scheme.
Note that there is no contradiction
with what is predicted by Theorem
\ref{thconvergencerateforexplicitscheme}. Indeed, it is expected that
for very large values of $N$ the asymptotic convergence will begin to take
place.
\footnote{In order to significantly increase $N$, we would also need to
increase $M$ to levels that are beyond our computational capabilities.}

For bigger values of $\alpha$ (between $40$ and $60$), we can finally observe
the transition to the asymptotic regime happening in our window of $N$'s.

Finally, for larger values of $\alpha$ ($60$ and above), we mark on
Figure~\ref{figpicture22-unbdd-tamed-explicit}(b) only the finite values
of $e(N)$ [defined in (\ref{eqerrorN})]. This shows in a
rather clear fashion that if we do not truncate strongly enough (for a given
value of $N$)
the scheme ``blows up'' (the code produces \emph{NaN} values).
One also observes that the bigger the multiplying factor $\alpha$ the
smaller the
time-step must be in order to make sure that $e(N)$ decays appropriately
(converges). This depicts very well the scenario described in our
counter-example.
We believe that the high convergence rates 
appearing in Table~\ref{tab-multfactandrates} when $\alpha$ is big is
due to
the smoothness of the driver $f$ we chose for Example~2 (similar to
Example~1)
and its damping effect on the dynamics of the scheme.
We leave an in-depth analysis of this fact for future research.


\begin{appendix}\label{app}
\section*{Appendix}

\subsection{Motivating example}\label{appdxmotivatignexample}
\setcounter{equation}{0}

Before we state the main result, we recall a result on the behavior of
Gaussian random variables [which we do not prove, but the reader is
invited to
try, in any case see Lemma 4.1 in \citet{HutzenthalerJentzenKloeden2011}]. The
notation and probability spaces we work with in this \hyperref[app]{Appendix} are as stated
in Section~\ref{secPreliminaries}.

%
\begin{lemma}
\label{lemmapropertiesofBM}
Let $(\Omega, \cF,\bP)$ be a probability space and let $Z\dvtx \Omega\to
\bR$ be
an $\cF/\cB(\bR)$-measurable mapping with standard normal
distribution. Then
for any $x\in[0,\infty)$ it holds that
\[
\bP\bigl[ \llvert Z\rrvert\geq x \bigr]\geq\tfrac{1}4 x
e^{-x^2}.
\]
\end{lemma}

The statement of Lemma~\ref{lemmaeulerdoesntwork} follows from the next
lemma.

\begin{lemma}
\label{lemma-counterexample}
Let $\pi^N$ denote the uniform grid of the time interval $[0,1]$ with $N+1$
points and step size $h:=1/N$, where $N\in\bN$. Define the driver $f(y):=-y^3$
and the terminal condition $\xi\in L^p(\cF_1)$ for any $p\geq2$. Let $(Y,Z)$
be the unique solution to (\ref{counterexampleFBSDE}). Denote by
$\{Y^{(N)}_i\}_{i\in\{0,\ldots,N\}}$ the Euler approximation of
$(Y_t)_{t\in[0,1]}$ defined via (\ref{eqeulerschememotivating}) over
the grid
$\pi^N$.

Assume that $N$ is fixed and that $\xi$ verifies $\llvert \xi
\rrvert \geq2\sqrt{N}$
$\bP$-a.s. then:
\begin{longlist}[(iii)]
\item[(i)] For any $i \in\{0,\ldots,N\}$ it holds that $\llvert
Y_i\rrvert \ge
2^{2^{N-i}}\sqrt{N}$.
\end{longlist}

Assume now that $N$ is an \emph{even} number (hence $t=1/2$ is common
to all grids $\pi^N$) and denote by $Y^{(N)}_{1/2}$ the
approximation at the
time point $t=1/2$ (corresponding to $i=N/2$). Define $\xi$ as
$\xi:=W_{1/2}\in L^p(\cF_1)\setminus L^\infty(\cF_1)$ for any
$p\geq
1$.
\begin{longlist}[(iii)]
\item[(ii)] For any $i \in\{\frac{N}{2},\ldots,N\}$, on the set
$\{\omega\dvtx \xi(\omega)\ge2\sqrt{N}\}$ it holds that $\llvert
Y_i(\omega)\rrvert \ge
2^{2^{N-i}}\sqrt{N}$.
\item[(iii)]
Moreover, $\lim_{N \to\infty} \bE[ \llvert Y_{1/2}^{(N)}\rrvert
]=+\infty$.
\end{longlist}
\end{lemma}

\begin{pf}
For the given $f$ and $\xi$, the results from Section~2 in
\citet{Pardoux1999} combined with the a priori estimates stated in our
Section~\ref{secPreliminaries} ensure the existence and uniqueness of a solution $(Y,Z)\in
\cS^p\times\cH^p$ to BSDE (\ref{counterexampleFBSDE}) for any
$p\geq2$. We
now fix $N$ and drop the superscript $(N)$ from $Y^{(N)}$.

\begin{longlist}
\item[\emph{Proof of Part} (i).]
Without loss of generality, assume that $\xi
=Y_N \ge2
\sqrt{N}$. Then
\[
Y_{N-1} = \bE_{N-1}\bigl[Y_N -
Y_N^3 h \bigr] = \bE_{N-1}\bigl[Y_N
\bigl( 1 - Y_N^2 h \bigr) \bigr].
\]
Observe that $Y_N^2 \ge2 N$ which implies $(1- Y_N^2 h) \le(1-2^2) <0$.
Hence (since $Y_N>0$),
\[
Y_{N-1} = \bE_i \bigl[ Y_N\bigl( 1 -
Y_N^2 h \bigr) \bigr]\le- 2\sqrt{N}\bigl(2^2
-1\bigr) \le- 2^2 \sqrt{N}.
\]
Next (since $Y_{N-1}<0$) $Y_{N-1}^2 \ge2^4 N$ which implies
$1 - Y_{N-1}^2h \le(1 - 2^4) <0$. Hence,
\begin{eqnarray*}
Y_{N-2} &=& \bE_i \bigl[ Y_{N-1}
\bigl(1-Y_{N-1}^2 h\bigr) \bigr] = \bE_i \bigl[
(-Y_{N-1}) \bigl(Y_{N-1}^2 h-1\bigr) \bigr]
\\
& \ge& 2^2 \sqrt{N}\bigl(2^4-1\bigr) \ge2^{2^2}
\sqrt{N}.
\end{eqnarray*}
Proceeding by induction, we can show that
\[
\llvert Y_i\rrvert\ge2^{2^{N-i}}\sqrt{N}.
\]
Indeed, assume $\llvert Y_{i+1}\rrvert \ge
2^{2^{N-i-1}}\sqrt
{N}$ (in the light of
above calculations; the negative case is analogous), then
\[
Y_i = \bE_i \bigl[Y_N\bigl( 1 -
Y_N^2 h \bigr) \bigr] \le2^{2^{N-i-1}}\sqrt{N} \bigl(
\bigl(2^{2^{N-i-1}} \bigr)^2 -1 \bigr) \le2^{2^{N-i}}
\sqrt{N}
\]
and statement (i) is proved.

Before proving (ii) and (iii), we remark that no conditional
expectation needs to be computed for the scheme
(\ref{eqeulerschememotivating}) for $i \in
\{{N}/{2},\ldots,N\}$ because $\xi=W_{1/2}$ is
$\cF_t$-adapted for any $t\in[1/2,1]$.
The scheme's approximations up to $Y^{(N)}_{1/2}$ can be written as
\begin{eqnarray*}
Y^{(N)}_{N} &=&W_{1/2},\qquad Y^{(N)}_{N-1}
=\psi(W_{1/2} ),
\\
Y^{(N)}_{N-2} &=&\psi
\bigl(\psi(W_{1/2} ) \bigr),\qquad\ldots,\qquad
Y^{(N)}_{N/2} =\psi^{\circ(N/2)} (W_{1/2} ),
\end{eqnarray*}
where $\psi(x):=x-hx^3$ and $\psi^{\circ(n)}$ denotes the
composition of
$\psi$ with
itself $n$-times ($n\in\bN$).
\end{longlist}

\begin{longlist}
\item[\emph{Proof of Part} (ii).]
We work on the event that $\xi=Y_N \ge2
\sqrt{N}$.
We
have first
\[
Y_{N-1} = \bE_{N-1}\bigl[Y_N -
Y_N^3 h \bigr] = Y_N\bigl( 1 -
Y_N^2 h \bigr).
\]
Observe that $Y_N^2 \ge2^2 N$ which implies $(1- Y_N^2 h) \le(1-2^2) <0$.
Hence (since $Y_N>0$),
\[
Y_{N-1} = Y_N\bigl( 1 - Y_N^2 h
\bigr) \le- 2\sqrt{N}\bigl(2^2 -1\bigr) \le- 2^2
\sqrt{N}<0.
\]
Next, since $Y_{N-1}<0$, $Y_{N-1}^2 \ge2^4 N$ which implies
$1 - Y_{N-1}^2h \le(1 - 2^4) <0$. Hence,
\[
Y_{N-2} = Y_{N-1}\bigl(1-Y_{N-1}^2 h
\bigr) = -Y_{N-1}\bigl(Y_{N-1}^2 h-1\bigr)
\ge2^2 \sqrt{N}\bigl(2^4-1\bigr) \ge2^{2^2}
\sqrt{N}.
\]
Proceeding by induction we can easily show that
\[
\llvert Y_i\rrvert\ge2^{2^{N-i}}\sqrt{N},\qquad i=
\frac{N}{2},\ldots,N.
\]
Indeed, assume $Y_{i+1} \ge2^{2^{N-i-1}}\sqrt{N}$ (note that in the light
of the above calculations the negative case is analogous). Then
\[
Y_i = Y_{i+1}\bigl( 1 - Y_{i+1}^2 h
\bigr) \le2^{2^{N-i-1}}\sqrt{N} \bigl( 1 - \bigl(2^{2^{N-i-1}}
\bigr)^2 \bigr) \le- 2^{2^{N-i}}\sqrt{N}.
\]
\end{longlist}

\begin{longlist}
\item[\textit{Proof of Part} (iii).]
It follows easily from Lemma~\ref{lemmapropertiesofBM} that
\begin{eqnarray*}
&&\bP\bigl[ \llvert W_{1/2}\rrvert\ge2\sqrt{N} \bigr] \ge
\frac{\sqrt{2}}{2} \sqrt{N} e^{-8N}.
\end{eqnarray*}
\end{longlist}

Then, using part~(i) (to go from the first to the second line)
and the above remark (on the third line), we have
\begin{eqnarray*}
&& \lim_{N \rightarrow\infty}\bE\bigl[ \bigl\llvert Y^{(N)}_{1/2}
\bigr\rrvert\bigr]
\\
&&\qquad = \lim_{N \rightarrow\infty} \bE\bigl[\1_{\{\xi\ge2\sqrt{N}\}}\bigl
\llvert
Y^{(N)}_{1/2}\bigr\rrvert+ \1_{\{\xi< 2\sqrt{N}\}}\bigl\llvert
Y^{(N)}_{1/2}\bigr\rrvert\bigr]
\\
&&\qquad \ge\lim
_{N \rightarrow\infty} \bE\bigl[\1_{\{\xi\ge2\sqrt{N}\}} \bigl\llvert
Y^{(N)}_{1/2}\bigr\rrvert\bigr]
\\
&&\qquad \ge\lim_{N \rightarrow\infty} \bE\bigl[\1_{\{\xi\ge2\sqrt{N}\}}
2^{2^{N-N/2}} \sqrt{N} \bigr]
\\
&&\qquad = \lim_{N\rightarrow\infty} 2^{2^{N/2}} \sqrt{N} \bP\bigl[ \llvert
W_{1/2} \rrvert\ge2\sqrt{N} \bigr]
\\
&&\qquad \geq\lim_{N\rightarrow\infty}
2^{(2^{N/2})} \frac{\sqrt{2}}{2} N e^{-8N} =+\infty.
\end{eqnarray*}\upqed
\end{pf}

\subsection{Basics of Malliavin's calculus}
\label{appendix-malliavin-calculus}
We briefly introduce the main notation of the stochastic calculus of
variations also known as Malliavin's calculus. For more details, we
refer the reader to \citet{nualart2006}, for its application to BSDEs
we refer
to \citet{Imkeller2008}. Let ${\bolds\cS}$ be the space
of random variables of the form
\[
\xi= F \biggl(\biggl(\int_0^T
h^{1,i}_s \,\ud W^1_s
\biggr)_{1\le i\le
n},\ldots,\biggl(\int_0^T
h^{d,i}_s \,\ud W^d_s
\biggr)_{1\le i\le n}\biggr),
\]
where $F\in C_b^\infty(\bR^{n\times d})$, $h^1,\ldots,h^n\in L^2([0,T];
\bR^d)$, $n\in\bN$.
To simplify notation, assume that all $h^j$ are written as row
vectors.
For $\xi\in{\bolds\cS}$, we define $D = (D^1,\ldots, D^d)\dvtx {\bolds\cS
}\to
L^2(\Omega\times[0,T])^d$ by
\[
D^i_\theta\xi= \sum_{j=1}^n
\frac{\partial F}{\partial x_{i,j}} \biggl( \int_0^T
h^1_t \,\ud W_t,\ldots,\int
_0^T h^n_t\,\ud
W_t \biggr)h^{i,j}_\theta,\qquad0\leq\theta\leq
T, 1\le i\le d,
\]
and for $k\in\bN$ its $k$-fold iteration by
$D^{(k)} = (D^{i_1}\cdots D^{i_k})_{1\le i_1,\ldots, i_k\le d}$. For
$k\in\bN$, $p\ge1$ let $\bD^{k,p}$ be the closure of $\cS$ with
respect to
the norm
\[
\llVert\xi\rrVert_{k,p}^p = \bE\Biggl[\llVert\xi\rrVert
^p_{L^p} + \sum_{i=1}^{k}
\bigl\llVert\bigl\llvert D^{(k)} \xi\bigr\rrvert\bigr\rrVert
_{(\cH^p)^i}^p \Biggr].
\]
$D^{(k)}$ is a closed linear operator on the space $\mathbb{D}^{k,p}$.
Observe that if $\xi\in\bD^{1,2}$ is $\cF_t$-measurable then
$D_\theta
\xi=0$ for $\theta\in(t,T]$. Further denote
$\mathbb{D}^{k, \infty}=\bigcap_{p>1}\mathbb{D}^{k,p}$.

We also need Malliavin's calculus for $\bR^m$ valued smooth stochastic
processes. For $k\in\bN, p\ge1$, denote by $\mathbb{L}^{k,p}(\bR
^m)$ the set
of $\bR^m$-valued progressively measurable processes $u = (u^1,\ldots, u^m)$
on $[0,T]\times\Omega$ such that:
\begin{longlist}[(iii)]
\item[(i)]
For Lebesgue-a.a. $t\in[0,T]$, $u(t,\cdot)\in(\mathbb{D}^{k,p})^m$;
\item[(ii)] $[0,T]\times\Omega\ni(t,\omega)\mapsto D^{(k)}
u(t,\omega)\in
(L^2([0,T]^{1+k}))^{d\times n}$ admits a progressively measurable
version;
\item[(iii)] $\llVert u \rrVert_{k,p}^p
= \llVert u\rrVert_{\cH^p}^p+\sum_{i=1}^{k} \llVert D^i
u \rrVert_{(\cH^p)^{1+i}}^p <\infty$.
\end{longlist}

Note that Jensen's inequality gives\footnote{The reason behind this last
inequality is that within the BSDE framework the
usual tools to obtain a priori estimates yield with much difficulty the LHS
while with relative ease the RHS.} for all $p\geq2$
%
\begin{equation}
\label{jensenforMallCalculus} \bE\biggl[ \biggl( \int_0^T
\int_0^T\llvert D_u
X_t\rrvert^2\,\ud u \,\ud t \biggr)^{p/2} \biggr]
\leq T^{p/2-1}\int_0^T \llVert
D_u X\rrVert_{\cH^p}^p \,\ud u.
\end{equation}
We recall a result from \citet{Imkeller2008} concerning the rule for the
Malliavin differentiation of It\^o integrals which is of use in
applications of Malliavin's calculus to stochastic analysis.

\begin{theorem}[{[Theorem 2.3.4 in \citet{Imkeller2008}]}]
\label{malldiffstochintegrals}
Let $(X_t)_{t\in[0,T]}\in\cH^2$ be an adapted process
and define $M_t:=\int_0^t X_r\,\udwr$ for $t\in[0,T]$. Then
$X\in\bL^{1,2}$ if and only if $M_t\in\bD^{1,2}$ for any $t\in[0,T]$.

Moreover, for any $0\leq s, t\leq T$ we have
%
\begin{equation}
D_s M_t = X_s\1_{\{s\leq t\}}(s) +
\1_{\{s\leq t\}}(s)\int_s^t D_s
X_r \,\udwr.
\end{equation}
\end{theorem}

\subsection{A particular Gronwall lemma}
We state here a ``discrete Gronwall lemma'' of some kind, particularly useful
for the numerical analysis of BSDEs, and which we use extensively in this
work.

\begin{lemma}
\label{lemGron}
Let $a_i$, $b_i$, $c_i$, be such that $a_i,b_i \ge0$, $c_i\in\bR$
for $i=0,1,\ldots,N$.
Assume that, for some constant $c > 0$ and $h>0$, we have
%
\begin{equation}
\label{eqrec} a_i + b_i \le(1+ch)a_{i+1} +
c_i\qquad\mbox{for } i=0,1,\ldots,N-1.
\end{equation}
Then the following inequality holds for every $i$:
\[
a_i + \sum_{j=i}^{N-1}
b_j \le e^{c(N-i)h} a_N + \sum
_{j=i}^{N-1} e^{c(j-i)h} c_j.
\]
\end{lemma}

\begin{pf}
The estimate is clearly true for $i=N-1$ (even for $i=N$ in fact). Then,
for any $i \le N-2$, if it is true for $i+1$, by multiplying both sides by
$e^{ch}$ we find that
\[
e^{ch} a_\ip+ e^{ch} \sum
_{j=i+1}^{N-1} b_j \le e^{c(N-i)h}
a_N + \sum_{j=i+1}^{N-1}
e^{c(j-i)h} c_j.
\]
Summing\vspace*{2pt} this inequality with (\ref{eqrec}) and noting that
$\sum_{j=i+1}^{N-1}
b_j \le e^{ch} \sum_{j=i+1}^{N-1} b_j$ due to the positivity of the
$b_j$ terms gives the sought estimate for any $i$.
\end{pf}
\end{appendix}

%

\section*{Acknowledgments}
We would like to thank Samuel Cohen (University of
Oxford), Gechun Liang (King's college) and Joscha Diehl (TU-Berlin) for
helpful discussions.



%

\printaddresses
\end{document}